\documentclass[11pt]{amsart}

\usepackage{amsfonts}
\usepackage{amsmath}
\usepackage{setspace}
\usepackage{amsthm}
\usepackage{amssymb}
\usepackage[all]{xy}
\usepackage{color}
\usepackage{mathrsfs}
\usepackage{fancyhdr}
\usepackage{colortbl}
\usepackage{tikz}
\usetikzlibrary{arrows.meta}
\usetikzlibrary{patterns}
\usepackage[margin=1in]{geometry}
\newcommand{\Z}{{\mathbb{Z}}}

\newcommand{\R}{{\mathbb{R}}}

\newcommand{\Q}{{\mathbb{Q}}}

\newcommand{\e}{\varepsilon}
\newcommand{\pushout}{\ar@{}[ul(0.35)]|-{\ulcorner}}
\newcommand{\pullback}{\ar@{}[dr(0.35)]|-{\lrcorner}}

\DeclareMathOperator{\Rep}{Rep}
\DeclareMathOperator{\rep}{rep}
\newcommand{\repAR}{\rep_k(A_\R)}

\newcommand{\hp}{\hat{p}}
\newcommand{\dR}{d_{\R^2}}
\newcommand{\dAR}{d_{A_\R}}

\newcommand{\deriv}{\frac{\text{d}}{\text{d}z}}

\newcommand{\ReppwfAR}{\Rep_k^{\rm{pwf}}(A_\R)}

\newcommand{\MM}{\mathop{\boldsymbol{\Gamma}}}
\newcommand{\DbAR}{{\mathcal{D}^b(A_\R)}}

\DeclareMathOperator{\Hom}{Hom}

\definecolor{purple}{RGB}{155,0,155}
\usepackage{hyperref}

\definecolor{gordanagreen}{RGB}{0,155,0}

\definecolor{readableyellow}{RGB}{200,200,0}

\newcommand{\coker}{\mathop{\text{coker}}}
\newcommand{\supp}{\mathop{\text{supp}}}
\newcommand{\Ext}{\mathop{\text{Ext}}}
\newtheorem{lemma}{Lemma}[subsection]
\newtheorem{proposition}[lemma]{Proposition}
\newtheorem{theorem}[lemma]{Theorem}
\newtheorem{cor}[lemma]{Corollary}

\newtheorem{thm}{Theorem}
\theoremstyle{definition}
\newtheorem{definition}[lemma]{Definition}
\newtheorem{remark}[lemma]{Remark}
\newtheorem{example}[lemma]{Example}
\newtheorem{notation}[lemma]{Notation}

\newtheorem{construction}[lemma]{Construction}

\title[Continuous Quivers of Type $A$ (II)]{Continuous Quivers of Type $A$ (II)\\The Auslander-Reiten Space}
\author{J.D.\ Rock}
\date{\today}

\begin{document}

\begin{abstract}
This work is the sequel to Continuous Quivers of Type A (I).
In this paper we define the Auslander-Reiten space of a continuous type $A$ quiver, which generalizes the Auslander-Reiten quiver of type $A_n$ quivers.
We prove that extensions, kernels, and cokernels of representations of type $A_\R$ can be described by lines and rectangles in a way analogous to representations of type $A_n$.
We provide a similar description for distinguished triangles in the bounded derived category whose first and third terms are indecomposable.
Furthermore, we provide a complete classification of Auslander-Reiten sequences in the category of finitely generated representations of $A_\R$.
This is part of a longer work; the other papers in this series are with Kiyoshi Igusa and Gordana Todorov.
The goal of this series is to generalize cluster categories, clusters, and mutation for type $A_n$ quivers to continuous versions for type $A_\R$ quivers.
\end{abstract}

\maketitle

\tableofcontents

\section*{Introduction}
\subsection*{History}
Auslander-Reiten sequences were introduced by Auslander and Reiten in \cite{ARSequences} with further study by the same authors in \cite{ARSequences2, ARSequences3}.
Named after these early works, Auslander-Reiten theory is still an active area of research to understand the structure of certain categories via its irreducible morphisms \cite{ARTheory1, ARTheory2, ARTheory3, ARTheory4, ARTheory5, ARTheory6, ARTheory7, ARTheory8, ARTheory9, ARTheory10}.
One particular tool is the Auslander-Reiten quiver, which has recently appeared in the study of Specht modules \cite{DanzErdmann2014}, equipped graphs \cite{CrawleyBoevey2018}, and higher homological algebra \cite{Jorgensen2018}.

Along with Igusa and Todorov, in the previous paper the author constructed continuous quivers of type $A$, denoted $A_\R$ \cite{IgusaRockTodorov2019}.
Basic results were proven about the category of point-wise finite representations ($\ReppwfAR$) and finitely generated representations ($\repAR$) over a field $k$.
In particular it was shown that all pointwise finite-dimensional representations decompose into a direct sum of indecomposable representations similar to those indecomposable representations of $A_n$.
This essentially recovers the result of Botnan and Crawley-Boevey in \cite{BotnanCrawley-Boevey}, though by a different technique.
The previous paper concluded with results about the category of finitely generated representations, denoted $\repAR$.
In particular, $\repAR$ is \emph{not} artinian.

\subsection*{Contributions}
We generalize the Auslander-Reiten quiver to the Auslander-Reiten space (Definition \ref{def:AR space}).
To do this we define the Auslander-Reiten topology and an extra generalized metric (Definitions \ref{def:AR topology} and \ref{def:nonstandard metric}) on the (isomorphism classes of) indecomposable representations using a mapping to $\R^2$ and irreducible morphisms.

The first result is the classification of Auslander-Reiten sequences in $\repAR$.
A complete list of 16 types of Auslander-Reiten sequences is provided in Table \ref{tab:AR sequence table}.
%\begin{thm}[Theorem \ref{thm:AR sequence classification}]
%Let $0\to U\to V \to W \to 0$ be an Auslander-Reiten sequence in $\repAR$.
%Then it is one of the 16 types in Table \ref{tab:AR sequence table}.
%\end{thm}
%
%The corollary after the theorem classifies which indecomposable representations belong to an Auslander-Reiten sequence.
%Further, if an indecomposable representation appears in an Auslander-Reiten sequence it appears in exactly one sequence and in exactly one place (kernel, extension, or cokernel).
\begin{thm}[Corollary \ref{cor:unique AR sequence}]
Let $M_{|a,b|}$ be an indecomposable in $\repAR$ such that
\begin{itemize}
\item $M_{|a,b|}$ is not projective,
\item $M_{|a,b|}$ is not injective, and
\item $M_{|a,b|}$ is neither simple nor has support of the form $[s_n,s_{n+1}]$, where $s_n$ and $s_{n+1}$ are an adjacent sink and source.
\end{itemize}
Then, there exists a unique Auslander-Reiten sequence in $\repAR$ of one of the 16 types in Table \ref{tab:AR sequence table} containing $M_{|a,b|}$. That is, an Auslander-Reiten sequence $0\to U\hookrightarrow V\twoheadrightarrow W\to 0$ in $\repAR$ such that $M_{|a,b|}\cong U$, $M_{|a,b|}\cong W$, or there exists $M_{|c,d|}$ such that $V\cong M_{|a,b|}\oplus M_{|c,d|}$.

If $M_{|a,b|}$ does not satisfy the above conditions then it does not belong to any Auslander-Reiten sequence.
\end{thm}
As a consequence, there can be no Auslander-Reiten translation in $\repAR$ that takes on the traditional properties.
See Remark \ref{rem:no AR translation} for more discussion.

The next contribution justifies the name ``Auslander-Reiten space.''
We show that rectangles and almost complete rectangles (Definitions \ref{def:rectangle} and \ref{def:almost complete rectangle}) are in one-to-one correspondence with nontrivial extensions of indecomposable representations.
The description of the extensions coincides with middle exact sequences in \cite[Section 5]{BotnanCrawley-Boevey}.

\begin{thm}[Theorem \ref{thm:extensions are rectangles}]
Let $V=M_{|a,b|}$ and $W=M_{|c,d|}$ be indecomposables in $\repAR$ such that $V\not\cong W$.
Then there is a nontrivial extension $V\hookrightarrow E\twoheadrightarrow W$ if and only if there exists a rectangle or almost complete rectangle whose corners are the indecomposables in the sequence with $V$ as the left-most corner and $W$ as the right-most corner.
\begin{itemize}
\item If the rectangle is complete $E$ is a direct sum of two indecomposables.
\item If the rectangle is almost complete $E$ is indecomposable.
\end{itemize}
Furthermore, there is a bijection
\begin{displaymath}\begin{tikzpicture}
\draw (0,0) node {$\{$rectangles and almost complete rectangles with ``good'' slopes of sides in AR-space of $\repAR\}$};
\draw [<->, thick] (0,-.2) -- (0,-1.2);
\draw (0,-.7) node[anchor=west] {$\cong$};
\draw (0,-1.4) node {$\{$nontrivial extensions of indecomposables by indecomposables up to scaling and isomorphisms$\}$};
\end{tikzpicture}\end{displaymath}
\end{thm}
\noindent The ``good'' slopes in the theorem above are defined in Section \ref{sec:AR space} (Definition \ref{def:slope}).

The final contribution extends Theorem \ref{thm:extensions are rectangles} to the bounded derived category: $\DbAR$.
In a triangulated category, we consider a distinguished triangle to be distinct from its rotations for the purposes of the statement of the theorem.
\begin{thm}[Theorem \ref{thm:triangles are rectangles}]
Let $V=M_{|a,b|}[m]$ and $W=M_{|c,d|}[n]$ be indecomposables in $\DbAR$ such that $V\not\cong W$.
Then there is a nontrivial distinguished triangle $V\to U\to W\to$ if and only if there exists a rectangle or almost complete rectangle in the AR-space of $\DbAR$ whose corners are the indecomposables in the triangle with $V$ as the left-most corner and $W$ as the right-most corner.

\begin{itemize}
\item If the rectangle is complete $E$ is a direct sum of two indecomposables.
\item If the rectangle is almost complete $E$ is indecomposable.
\end{itemize}

Furthermore, there is a bijection
\begin{displaymath}\begin{tikzpicture}
\draw (0,0) node {$\{$rectangles and almost complete rectangles with ``good'' slopes of sides in AR-space of $\DbAR\}$};
\draw [<->, thick] (0,-.2) -- (0,-1.2);
\draw (0,-.7) node[anchor=west] {$\cong$};
\draw (0,-1.4) node {$\{$nontrivial triangles with first and third term indecomposable up to scaling and isomorphisms$\}$};
\end{tikzpicture}\end{displaymath}
\end{thm}

It should be noted that the bijections in Theorems B and C are each a \emph{pair} of bijections.
Almost complete rectangles  correspond to indecomposable extensions or middle terms in a triangle.
Complete rectangles correspond to extensions or middle terms in a triangle with two indecomposable summands.

\subsection*{Future Work}
In the forthcoming Continuous Quivers of Type $A$ (III), the author, along with Kiyoshi Igusa and Gordana Todorov classify which continuous type $A$ quivers are derived equivalent.
The Auslander-Reiten space of the bounded derived category of $\repAR$ is an essential tool to the proof.
Further, they will define a generalization of the continuous cluster category constructed by Igusa and Todorov in \cite{IgusaTodorov2015}.
The new category will come with a continuous generalization of clusters.

Future work in this series will also include a continuous generalization of mutation to handle the new cluster-like objects.
The continuous generalizations allow for the embedding of existing discrete structures: cluster categories of type $A_n$ \cite{BMRRT, CalderoChapatonSchiffler2006} and of type $A_\infty$ \cite{HolmJorgensen2012}.
The continuous mutation is in a rigorous sense compatible with existing mutation.
Ordinary mutation and even transfinite mutation \cite{BaurGratz2018} commutes with these embeddings in a well-defined way.

\subsection*{Acknowledgements}
The author would like to thank Kiyoshi Igusa and Gordana Todorov for their guidance and for their work on the other papers in this series.
They would like to thank Ralf Schiffler for organizing the school on cluster algebras where the author, Kiyoshi Igusa, and Gordana Todorov, first thought of this series of papers.
They would like to thank Eric Hanson and Shijie Zhu for helpful discussions.
Finally, they would like to thank Magnus Bakke Botnan and Bill Crawley-Boevey for references to related work.

\section{The Category $\repAR$}
Fix a field $k$.
In this section we recall the definitions and theorems from \cite{IgusaRockTodorov2019} that we need for the rest of the paper.
\subsection{Continuous Quivers of Type $A$}
The first necessary definition is that of a continuous quiver of type $A$. We include a picture to give intuition followed by the definition from \cite{IgusaRockTodorov2019}. Afterwards we succinctly define a representation.
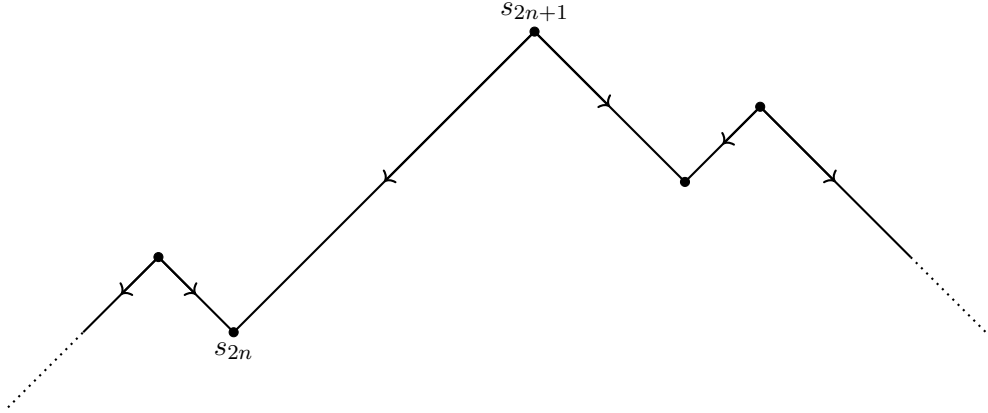
\begin{displaymath}\begin{tikzpicture}
\draw[thick, dotted] (-3,-2) -- (-2,-1);
\draw[thick, dotted] (9,0) -- (10,-1);
\draw[thick] (-2,-1) -- (-1,0) -- (0,-1) -- (4,3) -- (6,1) -- (7,2) -- (9,0);
\filldraw[fill=black](-1,0) circle[radius=.6mm];
\filldraw[fill=black](0,-1) circle[radius=.6mm];
\filldraw[fill=black](4,3) circle[radius=.6mm];
\filldraw[fill=black](6,1) circle[radius=.6mm];
\filldraw[fill=black](7,2) circle[radius=.6mm];
\draw[->, thick] (-1,0) -- (-1.5,-0.5);
\draw[->, thick] (-1,0) -- (-.5,-0.5);
\draw[->, thick] (4,3) -- (2,1);
\draw[->, thick] (4,3) -- (5,2);
\draw[->, thick] (7,2) -- (6.5,1.5);
\draw[->, thick] (7,2) -- (8,1);
\draw(0,-1) node[anchor=north] {$s_{2n}$};
\draw(4,3) node[anchor=south] {$s_{2n+1}$};
\end{tikzpicture}\end{displaymath}

\begin{definition}\label{def:AR}A \underline{quiver of continuous type $A$}, denoted by $A_\R$, is a triple $(\R,S,\preceq)$, where:
\begin{enumerate}
\item 
\begin{enumerate}
\item$S\subset \R$ is a discrete subset, possibly empty, with no accumulation points.
\item Order on $S\cup\{\pm\infty\}$ is induced by the order of $\R$, and $-\infty<s<+\infty$ for $\forall s\in S$.
\item Elements of $S\cup\{\pm\infty\}$ are indexed by a subset of $\Z\cup\{\pm\infty\}$ so that $s_n$ denotes the element of 
$S\cup\{\pm\infty\}$ with index $n$. The indexing must adhere to the following two conditions:
\begin{itemize}
\item[i1] There exists $s_0\in S\cup\{\pm\infty\}$.
\item[i2] If $m\leq n\in\Z\cup\{\pm\infty\}$ and $s_m,s_n\in S\cup\{\pm\infty\}$ then for all $p\in\Z\cup\{\pm\infty\}$ such that $m\leq p \leq n$ the element $s_p$ is in $S\cup\{\pm\infty\}$.
\end{itemize}
\end{enumerate}
\item New partial order $\preceq$ on $\R$, which we call  the \underline{orientation} of $A_\R$, is defined as:
\begin{itemize}
\item[p1\ ] The $\preceq$ order between consecutive elements of $S\cup\{\pm\infty\}$ does not change.
\item[p2\ ] Order reverses at each element of $S$.
\item[p3\ ] If $n$ is even $s_n$ is a sink.
\item[p3'] If $n$ is odd $s_n$ is a source.
\end{itemize}
\end{enumerate}
\end{definition}

\begin{definition}\label{def:representation}
Let $A_\R=(\R,S\preceq)$ be a continuous quiver of type $A$.
A \underline{representation} $V$ of $A_\R$ is the following data:
\begin{itemize}
\item A vector space $V(x)$ for all $x\in \R$.
\item For every pair $y\preceq x$ in $A_\R$ a linear map $V(x,y):V(x)\to V(y)$ such that if $z\preceq y \preceq x$ then $V(x,z)=V(y,z)\circ V(x,y)$.
\end{itemize}
We say $V$ is \underline{pointwise finite-dimensional} if $\dim V(x) < \infty$ for all $x\in\R$.
\end{definition}

We also need the definition of an interval indecomposable representation.

\begin{definition}\label{def:indecomposable}
For any interval $I$ in $\R$ let $M_I$ be the representation of $A_\R$ given as follows.
\begin{align*}
M_I(x) &= \left\{\begin{array}{ll} k & x\in I
 \\ 0 & \text{otherwise} \end{array}
 \right. 
 & 
M_I(x,y) &= \left\{\begin{array}{ll} 1_k & y\preceq x \text{ and }x,y\in I \\ 0 & \text{otherwise} \end{array}\right.
\end{align*} 
If $V\cong M_I$ we call $V$ an \underline{interval indecomposable} or \underline{interval indecomposable representation}.
\end{definition}

Theorem 2.4.13 in \cite{IgusaRockTodorov2019} (essentially \cite[Corollary 5.9]{BotnanCrawley-Boevey}), states that all pointwise finite-dimensional representations are direct sums of interval indecomposables and indecomposables themselves must interval indecomposables.

\begin{notation}\label{note:indefinite intervals}
Throughout this paper, it will often be useful to refer to an interval without knowing which endpoints are included.
We use the notation $|a,b|$ to mean any of the four valid possibilities, depending on whether or not $a$ or $b$ is $-\infty$ or $+\infty$, respectively.
The vertical bar $|$ can be thought of as an indication that the inclusion of that end point is indeterminate or inconsequential.
\end{notation}

\subsection{Projectives and Finitely Generated Represntations}
In the previous paper all projective indecomposables in the category of pointwise finite-dimensional representations were classified.
\begin{theorem}[Theorem 2.1.16 and Remark 2.4.14 from \cite{IgusaRockTodorov2019}]\label{thm:projectives}
The following is a complete classification of all indecomposable projectives in the category of pointwise finite-dimensional representations of $A_\R$. They come in three forms up to isomorphism:
\begin{enumerate}
\item $P_a$ given by
\begin{align*}
P_a(x) &= \left\{\begin{array}{ll} k & x\preceq a \\ 0 & \text{otherwise} \end{array}\right. & 
P_a(x,y) &= \left\{\begin{array}{ll} 1_k & y\preceq x \preceq a \\ 0 & \text{otherwise} \end{array}\right.
\end{align*} %(((
\item $P_{a)}$ given by
\begin{align*} P_{a)} &= \left\{ \begin{array}{ll}k & x\preceq a, x<a \\ 0 & \text{otherwise} \end{array}\right. & 
P_{a)}(x,y) &= \left\{\begin{array}{ll} 1_k & y\preceq x \preceq a, y\leq x < a \\ 0 & \text{otherwise} \end{array}\right. \end{align*}
\item $P_{(a}$ given by
\begin{align*} P_{(a} &= \left\{ \begin{array}{ll}k & x\preceq a, a<x \\ 0 & \text{otherwise} \end{array}\right. & 
P_{(a}(x,y) &= \left\{\begin{array}{ll} 1_k & y\preceq x \preceq a, a<x\leq y \\ 0 & \text{otherwise} \end{array}\right.\end{align*}
\end{enumerate} %)))
\end{theorem}

\begin{definition}\label{def:little rep}
Let $A_\R$ be a continuous quiver of type $A$.
The category $\repAR$ is the category of finitely generated pointwise finite-dimensional representations.
That is, there exists a finite sum $P=\bigoplus_{i=1}^n P_i$ of indecomposable projectives in Theorem \ref{thm:projectives} and an epimorphism $P\twoheadrightarrow V$.
\end{definition}

In the previous paper it is also proved that $\repAR$ is abelian, Krull-Schmidt, and each object $V$ is a finite direct sum of interval indecomposables (\cite{IgusaRockTodorov2019}, Theorem 3.0.1).
However, $\repAR$ is not artinian.

\section{The Mapping $\MM:\text{Ind}(\repAR) \to \R\times[-\frac{\pi}{2},\frac{\pi}{2}]$}\label{sec:Mapping}
In this section we define a function $\MM$ from the (isomorphism classes of) indecomposables of $\repAR$ to $\R^2$ in order to define the Auslander-Reiten space in Section \ref{sec:AR space} (Definition \ref{def:AR space}).
We extend $\tan^{-1}$ in the obvious way, $\R\cup\{\pm\infty\}\to [-\frac{\pi}{2},\frac{\pi}{2}]$.
In each of the definitions, propositions, etc., we assume that we have chosen a particular continuous quiver of type $A$.

\subsection{Projectives}
Recall that $S$ is the set of sinks and sources in $\R$ and $\bar{S}$ includes $\pm\infty$.
In this subsection we start defining our map $\MM$ by first defining it on projectives $P_s$ (see Theorem \ref{thm:projectives}) where $s\in \bar{S}$ is a sink or source.
We will then fill in the rest of the projective indecomposables.

\begin{definition}\label{def:M on projectives at s} 
If the indecomposable projective $P_{s_0}$ exists, map it to $(0,\tan^{-1} s_0 )$.
For any $s_n$ in $\bar{S}$ where $n\neq\pm\infty$, we want the slopes from $P_{s_n}$ to $P_{s_{n+1}}$ and to $P_{s_{n-1}}$ to be $\pm 1$.
In particular, we want these two slopes to be negatives of each other.
The idea is to ''wiggle'' away from the image of $P_{s_0}$ on slopes of $\pm 1$ so that the sinks sit behind their adjacent sources.
Here is one possible desired outcome:
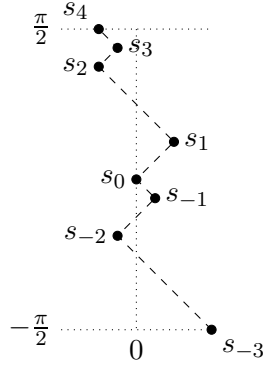
\begin{displaymath}\begin{tikzpicture}
\draw[dotted] (-1,2) -- (1,2);
\draw[dotted] (-1,-2) -- (1,-2);
\draw[dotted] (0,2) -- (0,-2);
\draw[dashed](1,-2) --  (-0.25, -0.75) -- (0.25,-0.25) -- (0,0) -- (0.5,0.5) -- (-0.5,1.5) -- (-0.25,1.75) -- (-0.5,2);
\draw (-1,2) node [anchor=east] {$\frac{\pi}{2}$};
\draw (-1,-2) node [anchor=east] {$-\frac{\pi}{2}$};
\draw (0,-2) node [anchor=north] {$0$};
\filldraw (1,-2) circle[radius=0.6mm];
\draw (1,-2) node [anchor=north west] {$s_{-3}$};
\filldraw (-0.25,-0.75) circle[radius=0.6mm];
\draw (-0.25,-0.75) node[anchor=east] {$s_{-2}$};
\filldraw (0.25,-0.25) circle[radius=0.6mm];
\draw (0.25,-0.25) node[anchor=west] {$s_{-1}$};
\filldraw (0,0) circle[radius=0.6mm];
\draw (0,0) node[anchor=east] {$s_0$};
\filldraw (0.5,0.5) circle[radius=0.6mm];
\draw (0.5,0.5) node [anchor=west] {$s_1$};
\filldraw (-0.5,1.5) circle[radius=0.6mm];
\draw (-0.5,1.5) node[anchor=east] {$s_2$};
\filldraw (-0.25,1.75) circle[radius=0.6mm];
\draw (-0.25,1.75) node[anchor=west] {$s_3$};
\filldraw (-0.5,2) circle[radius=0.6mm];
\draw (-0.5,2) node [anchor=south east] {$s_4$};
\end{tikzpicture}\end{displaymath}

In general, we use one of two formulas, depending on whether or not $n$ is positive or negative:
\begin{align*} 
\text{if }n>0 & \\
P_n &\mapsto  \left( \sum_{j=1}^n (-1)^{j+1} (\tan^{-1} s_j - \tan^{-1} s_{j-1}) , \quad \tan^{-1}(s_n) \right) \\
\text{if }n<0 & \\
P_n &\mapsto \left( \sum_{j=1}^{-n} (-1)^j (\tan^{-1} s_{-j}  - \tan^{-1} s_{-j+1}), \quad \tan^{-1}(s_n) \right)
\end{align*}
Note that if $s_n=\pm\infty$ for some $n\in\Z$ then the formulae are using the completed $\tan^{-1}$.
\end{definition}

\begin{example}\label{xmp:M on s projectives slope example}
For example, let us examine the formulas for $s_3$ and $s_4$.
We see that the $x$-coordinates of $\MM P_{s_3}$ and $\MM P_{s_4}$ are, respectively:
\begin{align*}
(\tan^{-1} s_1 - \tan^{-1} s_0) - (\tan^{-1} s_2 - \tan^{-1} s_1) + (\tan^{-1} s_3 - \tan^{-1} s_2) & \\
(\tan^{-1} s_1 - \tan^{-1} s_0) - (\tan^{-1} s_2 - \tan^{-1} s_1) + (\tan^{-1} s_3 - \tan^{-1} s_2) & - (\tan^{-1} s_4 - \tan^{-1} s_3).
\end{align*}
When displayed this way, it is clear that the absolute difference in $x$-coordinates of $\MM P_{s_3}$ and $\MM P_{s_4}$ is the same as the absolute difference in their $y$-coordinates.
\end{example}

\begin{proposition}\label{prop:M is well-defined on s projectives} 
The formulae at the end of Definition \ref{def:M on projectives at s} are well-defined.
The slope between $\MM P_{s_n}$ and $\MM P_{s_{n+1}}$ is $\pm 1$.
In particular, if $s_n$ is a sink $\MM P_{s_n}$ is mapped to the left of the images for its adjacent sources.
\end{proposition}
\begin{proof}
If $s_0=\pm\infty$ then $P_{s_1}$ and/or $P_{s_{-1}}$ are projective and the formulae work as defined for these values.
We can see that we are just summing the difference in $y$-coordinates from one $P_{s_n}$ to the next, but with alternating signs.
The slope of any line connecting $P_{s_n}$ and $P_{s_{n+1}}$ is $\pm 1$, since the difference in $x$-coordinates is given by $\pm (\tan^{-1} s_n - \tan^{-1} s_{n+1})$.

For $n\leq -1$ we obtain a similar result, except when $n$ is even the slope is negative and when $n$ is odd the slope is positive.
Since we're moving down when $n\leq -1$ and up when $n\geq 1$, the effect is the same.
The projectives at sinks are mapped to the left of the projectives at the adjacent sources and vice versa.
\end{proof}

\begin{notation}\label{note:images of M on s projectives}
We will denote by $(x_n,y_n)$ the image of $P_{s_n}$.
If $s_n = +\infty$ then it will be useful later to have $(x_{+\infty},y_{+\infty}) = (x_n,y_n)$.
Similarly, if $s_n=-\infty$ then $(x_{-\infty},y_{-\infty}) = (x_n,y_n)$.

If $S$ is unbounded above (respectively\ below) the sequence $\{(x_n,y_n)\}_{n\to+\infty}$ converges to $(x_{+\infty}, y_{+\infty})$ (respectively\ $\{(x_n,y_n)\}_{n\to-\infty}$ converges to $(x_{-\infty}, y_{-\infty})$).

Regardless of whether or not $S$ is bounded, $y_{-\infty}=-\frac{\pi}{2}$ and $y_{+\infty}=\frac{\pi}{2}$.
\end{notation}

Now we move on to mapping the rest of the projectives.
\begin{definition}\label{def:M on the rest of the projectives}
If $s_{+\infty}\in \bar{S}$ then there is no projective (or injective) at $+\infty$.
This is similarly true for $-\infty$.
For a source $s_n\in S$, we map $P_{(s_n}$ and $P_{s_n)}$ to the same point as $P_{s_n}$.
Sinks in $S$ and sources in $\bar{S}\setminus S$ have only one projective which have already been mapped.

For each $a\notin \bar{S}$, there is some $n\in\Z$ such that $s_{n}<a<s_{n+1}$ in $\R$.
Then there exists $t_a\in (0,1)$ such that $\tan^{-1} a = (1-t_a)\tan^{-1} s_n + t_a \tan^{-1}s_{n+1}$.
\begin{align*}
x_a &= (1-t_a)x_n + t_ax_{n+1} \\
y_a &= (1-t_a)y_n + t_ay_{n+1}
\end{align*}
We map both projectives at $a$ to $(x_a,y_a)$.
\end{definition}
\begin{example}\label{xmp:projective line example}
Here is one possibility of images of projectives under $\MM$:
\begin{displaymath}\begin{tikzpicture}
\draw[dotted] (-1,2) -- (1,2);
\draw[dotted] (-1,-2) -- (1,-2);
\draw[dotted] (0,2) -- (0,-2);
\draw(1,-2) --  (-0.25, -0.75) -- (0.25,-0.25) -- (0,0) -- (0.5,0.5) -- (-0.5,1.5) -- (-0.25,1.75) -- (-0.5,2);
\draw (-1,2) node [anchor=east] {$\frac{\pi}{2}$};
\draw (-1,-2) node [anchor=east] {$-\frac{\pi}{2}$};
\draw (0,-2) node [anchor=north] {$0$};
\draw (1,-2) node [anchor=north west] {$s_{-3}$};
\draw (-0.25,-0.75) node[anchor=east] {$s_{-2}$};
\draw (0.25,-0.25) node[anchor=west] {$s_{-1}$};
\draw (0,0) node[anchor=east] {$s_0$};
\draw (0.5,0.5) node [anchor=west] {$s_1$};
\draw (-0.5,1.5) node[anchor=east] {$s_2$};
\draw (-0.25,1.75) node[anchor=west] {$s_3$};
\draw (-0.5,2) node [anchor=south east] {$s_4$};
\end{tikzpicture}\end{displaymath}
\end{example}

\begin{proposition}\label{prop:M is sort of injective}
Suppose $a\neq b\in \R\setminus S$.
Then the following are true of Definition \ref{def:M on the rest of the projectives}.
\begin{enumerate}
\item $\MM P_a$ and $\MM P_b$ are well-defined.
\item $\MM P_a\neq \MM P_b$.
\item If $s_n < a< s_{n+1}$ then $(x_a,y_a)$ lies on the line segment from $(x_n,y_n)$ to $(x_{n+1},y_{n+1})$.
\end{enumerate}
\end{proposition}
\begin{proof}
For (1) the formulas are compositions of well-defined formulas.
For (2) we see that the $y$-coordinates of $\MM P_a$ and $\MM P_b$ will be different and so the points must be also.
We see (3) is clear by definition.
\end{proof}

\subsection{The $\lambda$ Functions}
In this subsection we define a collection of functions that we will use in Section \ref{sec:M is finally defined} to map the rest of the indecomposable representations to $\R\times[-\frac{\pi}{2},\frac{\pi}{2}]$.
We'll use values denoted $\kappa_a^-$ and $\kappa_a^+$ to define the collection of functions based on $\lambda$ (Definition \ref{def:undecorated lambda}).

\begin{definition}\label{def:undecorated lambda}
For any $z\in\R$, $z=2n\pi + w$ for $n\in\Z$ and $0\leq w \leq 2\pi$.
So let the function $\lambda:\R\to\R$ be given by
\begin{displaymath}
\lambda(z)=\lambda(2n\pi + w) = \left\{ \begin{array}{ll} w-\frac{\pi}{2} & 0\leq w \leq \pi \\ -w + \frac{3\pi}{2} & \pi\leq w \leq 2\pi. \end{array} \right.
\end{displaymath}
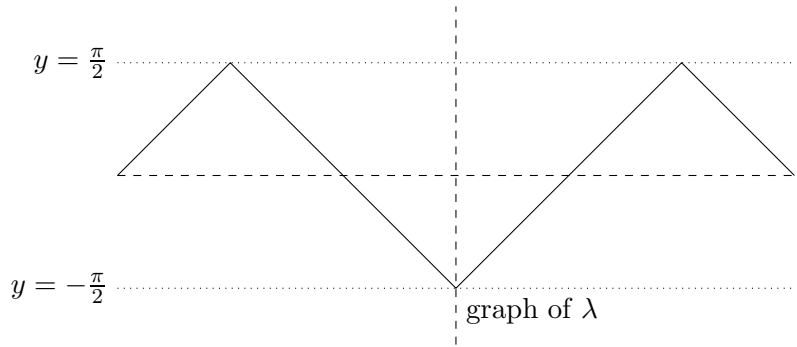
\begin{displaymath}
\begin{tikzpicture}[scale=.75]
\draw[dashed] (-6,0) -- (6,0);
\draw[dashed] (0,-3) -- (0,3);
\foreach \x in {-2,2}
	\draw[dotted] (-6,\x) -- (6,\x);
\draw (-6,0) -- (-4,2) -- (0,-2) -- (4,2) -- (6,0);
\draw (-6,2) node[anchor=east] {$y=\frac{\pi}{2}$};
\draw (-6,-2) node[anchor=east] {$y=-\frac{\pi}{2}$};
\draw (0,-2) node[anchor=north west] {graph of $\lambda$};
\end{tikzpicture}
\end{displaymath}
Note $\lambda$ is continuous and its derivative, when it is defined, it is $\pm 1$.
Furthermore, $\lambda(z)=\lambda(z+2\pi)$  and $\lambda(z)=\lambda(-z)$ for all $z\in\R$.
\end{definition}

\begin{definition}\label{def:kappas}
Let $a\in\R\cup\{\pm\infty\}$, possibly a sink or source.
Define the kappa values of $a$ as
\begin{align*}
\kappa_a^- &:= x_a+y_a+\frac{\pi}{2} \\
\kappa_a^+ &:= x_a-y_a-\frac{\pi}{2} 
\end{align*}
For sinks and sources we'll use both values and for other vertices in $A_\R$ we'll only use one.
\end{definition}

\begin{proposition}\label{prop:kappa order}
Let $a<b\in\R$.
Then \[ \kappa_b^+ \leq \kappa_a^+ \leq \kappa_a^- \leq \kappa_b^-.\]
\end{proposition}
\begin{proof}
By Proposition \ref{prop:M is well-defined on s projectives} and Proposition \ref{prop:M is sort of injective},
\[ |x_a-x_b| \leq y_b-y_a.\]
This yields
\begin{align*}
x_a + y_a &\leq x_b + y_b \\
x_b - y_b &\leq x_a - y_a.
\end{align*}
So, immediately, we have $\kappa_b^+ \leq \kappa_a^+$ and $\kappa_a^- \leq \kappa_b^-$.
Note that $-y_a-\frac{\pi}{2} \in[-\pi,0]$ and $y_a+\frac{\pi}{2} \in[0,\pi]$.
Thus, $\kappa_a^+ \leq \kappa_a^-$, concluding the proof.
\end{proof}

\begin{definition}\label{def:lambda functions}
We define $\lambda_{s_n}^-$ and $\lambda_{s_n}^+$ for each sink and source $s_n$:
\begin{align*}
\lambda_{s_n}^-(z) &:= \lambda(z - \kappa_{s_n}^-) \\
\lambda_{s_n}^+(z) &:= \lambda(z - \kappa_{s_n}^+)
\end{align*}
Note that adjacent sinks and sources share one of these functions (Proposition \ref{prop:sink and source share}).

If $a$ is not a sink or source we only define $\lambda_a$.
Let $s_n$ and $s_{n+1}$ be the sink and source pair such that $s_n < a < s_{n+1}$ in $\R$.
Then
\begin{align*}
\lambda_a(z) &:= \lambda(z-\kappa_a^-) \text{ if }s_n\text{ is a sink} \\
\lambda_a(z) &:= \lambda(z-\kappa_a^+)  \text{ if }s_n\text{ is a source}
\end{align*}

Finally, if $s_{+\infty}\in \bar{S}$ set $\lambda_{+\infty}(z) = \lambda(z - x_{\infty})$.
If $s_{-\infty}\in \bar{S}$ set $\lambda_{-\infty}(z) = \lambda(z - x_{\infty})$.
\end{definition}

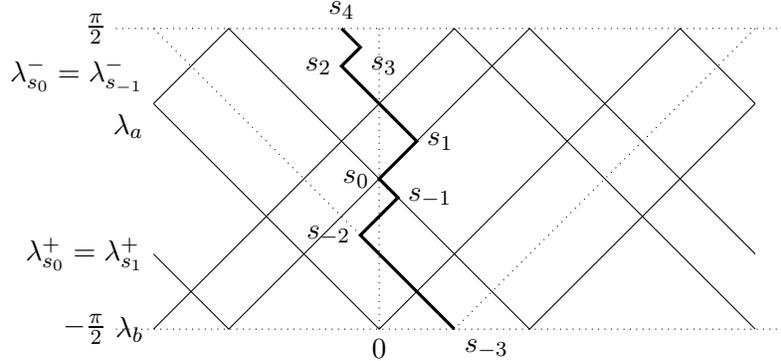
\begin{displaymath}\begin{tikzpicture}
\draw[dotted] (-3.5,2) -- (5.5,2);
\draw[dotted] (-3.5,-2) -- (5.5,-2);
\draw[dotted] (0,2) -- (0,-2);
\draw[dotted] (5,2) -- (1,-2) -- (-3,2);
\draw[very thick](1,-2) --  (-0.25, -0.75) -- (0.25,-0.25) -- (0,0) -- (0.5,0.5) -- (-0.5,1.5) -- (-0.25,1.75) -- (-0.5,2);
\draw (-3.5,2) node [anchor=east] {$\frac{\pi}{2}$};
\draw (-3.5,-2) node [anchor=east] {$-\frac{\pi}{2}$};
\draw (0,-2) node [anchor=north] {$0$};
\draw (1,-2) node [anchor=north west] {$s_{-3}$};
\draw (-0.25,-0.75) node[anchor=east] {$s_{-2}$};
\draw (0.25,-0.25) node[anchor=west] {$s_{-1}$};
\draw (0,0) node[anchor=east] {$s_0$};
\draw (0.5,0.5) node [anchor=west] {$s_1$};
\draw (-0.5,1.5) node[anchor=east] {$s_2$};
\draw (-0.25,1.75) node[anchor=north west] {$s_3$};
\draw (-0.5,2) node [anchor=south] {$s_4$};
\draw (-3,-2) -- (1,2) -- (5,-2);
\draw (-3,1) -- (0,-2) -- (4,2) -- (5,1);
\draw (-3,-1) -- (-2,-2) -- (2,2) -- (5,-1);
\draw (-3,1) -- (-2,2) -- (2,-2) -- (5,1);
\draw (-3,-2) node[anchor=east]{$\lambda_b$};
\draw (-3,-1) node[anchor=east] {$\lambda_{s_0}^+=\lambda_{s_1}^+$};
\draw (-3,1) node[anchor=north east] {$\lambda_a$};
\draw (-3,1) node[anchor=south east] {$\lambda_{s_0}^-=\lambda_{s_{-1}}^-$};
\end{tikzpicture}\end{displaymath}

For counting purposes in Proposition \ref{prop:lambda covers}, we will count the function shared by an adjacent sink and source (Proposition \ref{prop:sink and source share}) once.
If $\lambda_{s_n}^-=\lambda_{s_n}^+$ (Proposition \ref{prop:lambda properties} (4)) we also do not count these as separate functions.

\begin{notation}\label{note:lambda_*^*}
When we choose a particular $\lambda$ function without knowing which $\lambda$ function it is we will write $\lambda_*^*$.
For example, $\lambda_*^*(z)=\lambda_*^*(z+2\pi)$ for all $z\in\R$.
If we know $a$ but do not know whether or not to decorate $\lambda$ with $+$ or $-$, we write $\lambda_a^*$.

We similarly use the $*$ for the kappa values.
We use $\kappa_a^*$ to mean either $\kappa_a^-$ or $\kappa_a^+$.
\end{notation}

\begin{proposition}\label{prop:sink and source share}
Let $s_{2n}$ be a sink such that $s_{2n-1}$ and $s_{2n+1}$ are sources.
Then for all $z\in\R$
\begin{align*}
\lambda_{s_{2n}}^- (z) &= \lambda_{s_{2n-1}}^- (z) \\ 
\lambda_{s_{2n}}^+(z) &= \lambda_{s_{2n+1}}^+(z).
\end{align*}
\end{proposition}
\begin{proof}
The proofs of the equations are similar so we only prove the top equation.
Since $s_{2n}$ is a sink, we know by Proposition \ref{prop:M is well-defined on s projectives}  that $x_{2n}-x_{2n-1} = y_{2n-1}-y_{2n}$.
Then we have $x_{2n} + y_{2n} = x_{2n-1} + y_{2n-1}$ and so $\kappa_{s_{2n}}^- = \kappa_{s_{2n-1}}^-$ or $\kappa_{s_{2n}}^-= \kappa_{s_{2n-1}}^- \pm 2\pi$.
\end{proof}

\begin{proposition}\label{prop:lambda properties}
Let $s_n < a < s_{n+1}$ in $\R$. Then the following are true.
\begin{enumerate}
\item $\lambda_a(x_a)=y_a$ and $\lambda_{s_n}^-(x_n)=\lambda_{s_n}^+(x_n)=y_n$
\item If $s_n$ is a sink then $\deriv \lambda_a$ at $x_a$ is $-1$.
\item If $s_n$ is a source then $\deriv \lambda_a$ at $x_a$ is $+1$.
\item If $s_n=\pm\infty$ then $\lambda_{s_n}^-=\lambda_{s_n}^+$.
\end{enumerate}
\end{proposition}
\begin{proof}
(1) If $s_n$ is a sink we start with $x_a- \kappa_a^- = -y_a-\frac{\pi}{2}$ and if $s_n$ is a source we start with $x_a-\kappa_a^+ = y_a+\frac{\pi}{2}$.
Since $|y_a|<\frac{\pi}{2}$ we know $0<y_a+\frac{\pi}{2}<\pi$.
When $s_n$ is a sink,
\[ \lambda_a(x_a) = \lambda(-y_a-\frac{\pi}{2}) = \lambda(\underbrace{2\pi-y_a-\frac{\pi}{2}}_{\in (\pi,2\pi)}) = -2\pi + y_a + \frac{\pi}{2}+\frac{3\pi}{2} = y_a.\]
When $s_n$ is a source,
\[ \lambda_a(x_a) = \lambda(y_a+\frac{\pi}{2}) %= y_a+\frac{\pi}{2} - \frac{\pi}{2} 
= y_a.\]
Since $\kappa_{s_n}^-$ and $\kappa_{s_n}^+$ are defined similarly the rest of (1) is true by the same arguments.

(2) and (3) Since the derivative of $\lambda$ is constant between multiples of $\pi$, $a\neq \pm\infty$, and sinks and sources do not accumulate, we see that the derivative must be constant on $[x_a,x_a+\e]$, for $\e>0$.
Then we can just check the slope of the line from $(x_a,\lambda_a(x_a))$ to $(x_a+\e,\lambda_x(x_a+\e))$.

If $s_n$ is a sink then we use $\kappa_a^-$:
\[ \lambda_a(x_a+\e) = \lambda(2\pi-y_a-\frac{\pi}{2}+\e) = y_a-\e.\]
If $s_n$ is a source then we use $\kappa_a^+$:
\[ \lambda_a(x_a+\e) = \lambda(y_a+\frac{\pi}{2} +\e) = y_a+\e.\]
Therefore, if $s_n$ is a sink the derivative is negative at $x_a$ and if $s_n$ is a source it is positive.

(4) Suppose $s_n = -\infty$. Then $\kappa_{s_n}^- = x_n$ and $\kappa_{s_n}^+ = x_n$.
If $s_n=+\infty$ then $\kappa_{s_n}^- = x_n+\pi$ and $\kappa_{s_n}^+ = x_n-\pi = \kappa_{s_n}^- -2\pi$.
\end{proof}

\begin{proposition}\label{prop:lambda intersections}
Let $\lambda^*_a$, $\lambda^*_b$ be functions from Definition \ref{def:lambda functions} such that $\lambda^*_a\neq \lambda^*_b$.
Then, the intersection points of $\lambda^*_a$ and $\lambda^*_b$ are given by, for all $n\in\Z$,
\begin{displaymath}
\left(n\pi + \frac{1}{2}(\kappa_a^*+\kappa^*_b), \pm\left( \frac{1}{2}|\kappa_a^*-\kappa_b^*| - \frac{\pi}{2}\right)\right).
\end{displaymath}
Furthermore, if $\kappa_b^*>\kappa_a^*$ then
\begin{align*}
\left(\deriv \lambda^*_a\right)\left(\frac{1}{2}(\kappa_a^*+\kappa_b^*)\right) &= +1 \\
\left(\deriv \lambda^*_b\right)\left(\frac{1}{2}(\kappa_a^*+\kappa_b^*)\right) &= -1 \\
\end{align*}
\end{proposition}
\begin{proof}
First note that $\lambda^*_a(\kappa_a^*)=\lambda^*_b(\kappa_b^*)=-\frac{\pi}{2}$ and $\frac{1}{2}|\kappa_a^*-\kappa_b^*| < \pi$.
Then we see
%\begin{align*}
\[ \lambda^*_a\left(2n\pi + \frac{1}{2}(\kappa_a+\kappa_b)\right)
%&= \lambda^*_a\left( \frac{1}{2}(\kappa_a+\kappa_b)\right) \\
= \lambda\left( \frac{1}{2}(\kappa_b^* - \kappa_a^*)\right) \\
= \lambda\left( \frac{1}{2}(\kappa_a^*-\kappa_b^*)\right) \\
%&= \lambda^*_b\left( \frac{1}{2} (\kappa_a+\kappa_b)\right) \\
= \lambda^*_b\left( 2n\pi + \frac{1}{2}(\kappa_a^*+\kappa_b^*) \right). \]
%\end{align*}
A similar calculation shows $\lambda^*_a(2(n+1)\pi + \frac{1}{2}(\kappa_a^*+\kappa_b^*)) = \lambda^*_b(2(n+1)\pi + \frac{1}{2}(\kappa_a^*+\kappa_b^*))$.
Since this is true for all evens and odds, it's true for all of $\Z$.
%To see that the $y$-coordinates are as described, note that if $\kappa_b^*>\kappa_a^*$ then if both are $+$ or both are $-$ we have $\frac{1}{2}(\kappa_b^*-\kappa_b^*) \leq \pi$.
%If $\kappa_b^*$ is $-$ and $\kappa_a^*$ is $+$ then easy calculations show that when the difference is between $\pi$ and $2\pi$ we obtain the desired description of the $y$-coordinate again.
By symmetry, suppose $\kappa_b^*>\kappa_a^*$.
Then 
\begin{align*}
\lambda^*_a\left(\frac{1}{2}(\kappa_a^*+\kappa_b^*)\right) &= \frac{1}{2}(\kappa_b^*-\kappa_a^*) -\frac{\pi}{2} + \e \\
\lambda^*_b\left(\frac{1}{2}(\kappa_a^*+\kappa_b^*)\right) &= \frac{1}{2}(\kappa_b^*-\kappa_a^*) -\frac{\pi}{2} - \e \qedhere
\end{align*}
%Thus, the described points of intersection are the only points of intersection and the derivatives are as stated.
\end{proof}

We define values $\hp_z$ for all $z\in\R$ that will help make some proofs easier.

\begin{definition}\label{def:p hat}
Let $s_n$ be a sink or source such that $s_n\neq\infty$.
\begin{itemize}
\item if $s_n$ is a sink then let $\hp_{s_n}=x_n - (\frac{\pi}{2} - y_n)$.
\item if $s_n$ is a source then let $\hp_{s_n}=x_n + (\frac{\pi}{2} - y_n)$.
\end{itemize}
If $s_n=+\infty$ let $\hp_{s_n}=\hp_{+\infty}=x_n$.
If $s_n=-\infty$, (i) let $\hp_{s_n}=\hp_{-\infty} = x_n-\pi$ if $s_n$ is a sink and (ii) let $\hp_{s_n} = \hp_{-\infty} = x_n+\pi$ if $s_n$ is a source.

Let $a\in\R$ such that $a$ is neither a sink nor a source.
If $\kappa_a^*  =\kappa_a^+$ let $\hp_a = x_a + (\frac{\pi}{2} - y_a)$.
If $\kappa_a^* = \kappa_a^-$ let $\hp_a = x_a - (\frac{\pi}{2} - y_a)$.
\end{definition}

\begin{example}\label{xmp:p hat example}
Let us continue with Example \ref{xmp:projective line example}.
The points below are labeled by their $x$-coordinates, which are the $\hp$ values for each sink and source.
Of course, we must prove the picture is an accurate representation (Proposition \ref{prop:p hat order}).
\begin{displaymath}\begin{tikzpicture}
\draw[dotted] (-3.5,2) -- (5.5,2);
\draw[dotted] (-3.5,-2) -- (5.5,-2);
\draw[dotted] (0,2) -- (0,-2);
\draw[dotted] (5,2) -- (1,-2) -- (-3,2);
\draw(1,-2) --  (-0.25, -0.75) -- (0.25,-0.25) -- (0,0) -- (0.5,0.5) -- (-0.5,1.5) -- (-0.25,1.75) -- (-0.5,2);
\draw (-3.5,2) node [anchor=east] {$\frac{\pi}{2}$};
\draw (-3.5,-2) node [anchor=east] {$-\frac{\pi}{2}$};
\draw (0,-2) node [anchor=north] {$0$};
\draw (1,-2) node [anchor=north west] {$s_{-3}$};
\draw (-0.25,-0.75) node[anchor=east] {$s_{-2}$};
\draw (0.25,-0.25) node[anchor=west] {$s_{-1}$};
\draw (0,0) node[anchor=east] {$s_0$};
\draw (0.5,0.5) node [anchor=west] {$s_1$};
\draw (-0.5,1.5) node[anchor=east] {$s_2$};
\draw (-0.25,1.75) node[anchor=north west] {$s_3$};
\draw (-0.5,2) node [anchor=south] {$s_4$};
\filldraw[fill=black] (-3,2) circle[radius=0.6mm]; %s_-2
\draw[dotted] (-3,2) -- (-.25,-.75);
\draw (-3,2) node[anchor=south] {$\hp_{s_{-2}}$};
\filldraw[fill=black] (-2,2) circle[radius=0.6mm]; %s_0
\draw[dotted] (-2,2) -- (0,0);
\draw (-2,2) node[anchor=south] {$\hp_{s_0}$};
\filldraw[fill=black] (-1,2) circle[radius=0.6mm]; %s_2
\draw[dotted] (-1,2) -- (-.5,1.5);
\draw (-1,2) node[anchor=south] {$\hp_{s_2}$};
\filldraw[fill=black] (5,2) circle[radius=0.6mm]; %s_-3
\draw[dotted] (5,2) -- (1,-2);
\draw (5,2) node[anchor=south] {$\hp_{s_{-3}}$};
\filldraw[fill=black] (2.5,2) circle[radius=0.6mm]; %s_-1
\draw[dotted] (2.5,2) -- (.25,-.25);
\draw (2.6,2) node[anchor=south] {$\hp_{s_{-1}}$};
\filldraw[fill=black] (2,2) circle[radius=0.6mm]; %s_1
\draw[dotted] (2,2) -- (.5,.5);
\draw (2,2) node[anchor=south] {$\hp_{s_1}$};
\filldraw[fill=black] (0,2) circle[radius=0.6mm]; %s_3
\draw[dotted] (0,2) -- (-.25,1.75);
\draw (0,2) node[anchor=south] {$\hp_{s_3}$};
\end{tikzpicture}\end{displaymath}
\end{example}

\begin{proposition}\label{prop:the correct p hat}
Let $s_n$ be a sink, $s_m$ a source, and $a\in\R$ be neither.
Then
\begin{enumerate}
\item $\lambda_{s_n}^- (\hp_{s_n}) = \frac{\pi}{2}$,
\item $\lambda_{s_m}^+(\hp_{s_m}) = \frac{\pi}{2}$, and
\item $\lambda_a(\hp_a) = \frac{\pi}{2}$.
\end{enumerate}
\end{proposition}
\begin{proof}
We prove (1); proofs of (2) and (3) are similar. Recall $\hp_{s_n} = x_n - (\frac{\pi}{2} - y_n)$.
\[ \lambda_{s_n}^- \left(x_n - \left(\frac{\pi}{2} - y_n\right)\right) = \lambda\left(x_n - \left(\frac{\pi}{2} - y_n\right) - x_n - y_n - \frac{\pi}{2}\right) = \lambda(-\pi)  = \frac{\pi}{2}.\qedhere \]
\end{proof}

Recall that $s_n$ is a sink if $n$ is even and a source if $n$ is odd and $\bar{S}$ is the set of sinks and sources in $A_\R$ union $\{\pm\infty\}$.
\begin{proposition}\label{prop:p hat order}
Let $\cdots < s_{2n-1} < s_{2n} < s_{2n+1} < s_{2n+2} < \cdots $ be the sinks and sources of $A_\R$. Then
\[ (x_{-\infty}-\pi) \leq \hp_{s_{2n}} < \hp_{s_{2n+2}} \leq x_{+\infty} \leq \hp_{s_{2n+1}} < \hp_{s_{2n-1}}\leq (x_{-\infty}+\pi) \]
Equalities on the ends can only occur if $S$ is bounded below.
Equalities in the middle can only occur if $S$ is bounded above.
\end{proposition}
\begin{proof}
We'll first show $\hp_{s_{2n+1}}<\hp_{s_{2n-1}}$.
By Proposition \ref{prop:M is well-defined on s projectives} and Definition \ref{def:M on projectives at s} we know \[ x_{2n+1} - x_{2n-1} \leq |x_{2n-1} - x_{2n+1} | < y_{2n+1} - y_{2n-1}. \]
This immediately yields \[ \hp_{s_{2n+1}} = x_{2n+1} + \frac{\pi}{2} - y_{2n+1} < x_{2n-1} + \frac{\pi}{2} - y_{2n_1} = \hp_{s_{2n-1}} .\]
By a similar argument $\hp_{s_{2n}} < \hp_{s_{2n+2}}$.

We have two cases: either $s_m=+\infty$ for some $m$ or $s_{+\infty}\in\bar{S}$.
If there exists $m$ such that $s_m$ is a sink or source, then $\hp_{s_m}=x_{+\infty}$.
Suppose $s_m$ is a sink.
Then $s_{m-1}$ is a source and $\lambda_{s_{m-1}}^- = \lambda_{s_m}^-$ by Proposition \ref{prop:sink and source share}.
Then $\lambda_{s_{m-1}}^-(x_{m-1} - (\frac{\pi}{2} - y_{m-1})) = \frac{\pi}{2}$.
We know $\lambda_{s_{m-1}}^+(\hp_{s_{m-1}})=\frac{\pi}{2}$ (Proposition \ref{prop:the correct p hat}) and so $x_{+\infty} < \hp_{s_{m-1}}$.
Similarly, if $s_m$ is a source then $\hp_{s_{m-1}} < x_{+\infty}$.
Note the equality.

Now suppose $s_{+\infty}\in\bar{S}$.
But then $|x_{+\infty} - x_m| < \frac{\pi}{2} - y_m$ for all for all sources $s_m$.
Then $x_{+\infty} < \hp_{s_m}$ for all sources $s_m$.
Similarly, $\hp_{s_m} < x_{+\infty}$ for all sinks $s_m$.
Note the equality.

It remains to show the inequalities involving $x_{-\infty}$.
Suppose there exists $m$ such that $s_m=-\infty$.
If $s_m$ is a source then $\lambda_{s_m}^+ = \lambda_{s_m}^- = \lambda_{s_{m+1}}^-$ by Propositions \ref{prop:sink and source share} and \ref{prop:lambda properties}.
Then $\hp_{s_{m+1}} = x_{-\infty} - \pi$.
From above we know $\hp_{s_{m+2}} < \hp_{s_m} = x_{-\infty}+\pi$.
If $s_m$ is a sink the inequalities hold by a similar argument.
Note the possibility of an equality.

Finally, suppose $s_{-\infty}\in\bar{S}$.
Then $|x_m- x_{-\infty}| < y_m + \frac{\pi}{2}$ for all sources $s_m$.
Thus, $\hp_{s_m}< x_{-\infty} + \pi$.
Similarly, $x_{-\infty} - \pi < \hp_{s_m}$ for all sinks $s_m$.
Note the impossibility of an equality.
\end{proof}

\begin{proposition}\label{prop:p hat sides}
Let $a\in\R$ such that $a$ is neither a sink nor a source.
Let $s_n < a < s_{n+1}$ be the sink and source surrounding $a$.
If $s_n$ is a sink then $\hp_{s_n} < \hp_a < \hp_{s_{n+2}}$.
If $s_n$ is a source then $\hp_{s_{n+2}} < \hp_a < \hp_{s_n}$.
\end{proposition}
\begin{proof}
Note that $\tan^{-1}s_n < \tan^{-1} a < \tan^{-1} s_{n+1}$ so $-y_{n+1} < -y_a < -y_n$.
Suppose $s_n$ is a source.
Then $x_{n+1} < x_a < x_n$.
Therefore, \[ x_{n+1} +\frac{\pi}{2} - y_{n+1} < x_a + \frac{\pi}{2} -y_a = \hp_a < x_n +\frac{\pi}{2} - y_n = \hp_{s_n}.\]
However, by Proposition \ref{prop:sink and source share} \[ x_{n+1} - \frac{\pi}{2} - y_{n+1} = x_{n+2} + \frac{\pi}{2} - y_{n+2} = \hp_{s_{n+2}} .\]
Therefore, the inequality holds.
When $s_n$ is a sink the proof is similar.
\end{proof}

\subsection{The Codomain of $\MM$}\label{sec:codomain of M}
This subsection deals with the values that $\MM$ can take. We first show that the graphs of the $\lambda$ functions cover $\R\times[-\frac{\pi}{2},\frac{\pi}{2}]$ in a convenient way.
Then we describe the range of values of $x$-coordinates that we use in Section \ref{sec:M is finally defined}.

\begin{proposition}\label{prop:lambda covers}
The $\lambda$ functions cover $\R\times\{-\frac{\pi}{2},\frac{\pi}{2}\}$ once and $\R\times(-\frac{\pi}{2},\frac{\pi}{2})$ twice.
\end{proposition}
\begin{proof}
We first prove the $\lambda$ functions in Definition \ref{def:lambda functions} cover $\R\times\{\frac{\pi}{2}\}$.
Note that each $\lambda$ function is $2\pi$-periodic.
So, it suffices to show there exists a half open interval $I$ of length $2\pi$ such that $I\times\{\frac{\pi}{2}\}$ is covered exactly once. %(
We will show that the interval $[x_{-\infty} - \pi, x_{-\infty}+\pi)$ is covered exactly once by the $\lambda$ functions. %]
In particular, $(x_{-\infty}-\pi, \frac{\pi}{2})$ and $(x_{+\infty},\frac{\pi}{2})$ are only in the image of $\lambda_{-\infty}$ and $\lambda_{+\infty}$, respectively.

We now show that if $\hp_{s_{n+2}} < z < \hp_{s_{n}}$ for $s_n$ a source then there exists $a\in\R$ such that $\hp_a=z$.
There exists a unique $t\in(0,1)$ such that \[ z = (1-t)\hp_{s_{n+2}}  + t\, \hp_{s_n} .\]
We know $x_{n+2} + \frac{\pi}{2} - y_{n+2} = x_{n+1} + \frac{\pi}{2} - y_{n+1}$.
Then \[ z = ((1-t) x_{n+1} + tx_n )  + \frac{\pi}{2}   - ((1-t)y_{n+1} + ty_n).\]
Since $\tan^{-1}:\R\to[-\frac{\pi}{2},\frac{\pi}{2}]$ is bijective and order preserving, there is a unique $a\in\R$ such that $s_{n} < a < s_{n+1}$ and \[ \tan^{-1} a = (1-t)\tan^{-1} s_{n+1} + t \tan^{-1} s_n. \]
Then $z =\hp_a$ and if $b\neq a$ then $\hp\neq z$.

This argument also shows that if $+\infty$ is a sink  and $x_{+\infty} < z < \hp_{s_n}$ then there is a unique $a> s_n\in\R$ such that $\hp_a = z$.
If $+\infty$ is not a sink or source, then note that $\lim_{n\to+\infty} x_n = x_{+\infty}$ and $\lim_{n\to+\infty} y_n=\frac{\pi}{2}$.
Thus, $\lim_{n\to+\infty} \hp_{s_{2n+1}} = x_{+\infty}$.
Therefore, for all $z\in (x_{+\infty},x_{-\infty}+\pi)$ there is a unique $a\in\R$ such that $\hp_a=z$.
By a similar argument we have such a unique $a$ for each $z$ between $x_{-\infty}-\pi$ and $x_{+\infty}$. %(

Therefore, for all $z\in[x_{-\infty}-\pi,x_{+\infty}+\pi)$ there is a unique $\lambda$ function such that $\lambda_*^*(z)=\frac{\pi}{2}$ (recall Notation \ref{note:lambda_*^*}).
Thus $(z,\frac{\pi}{2})$ is the image of a unique $\lambda$ function for all $z\in\R$.
For a point $(z,-\frac{\pi}{2})$ use the technique above to find $z-\pi$.
Then $\lambda_*^*(z-\pi) = \frac{\pi}{2}$ and $\lambda_*^* (z) = -\frac{\pi}{2}$ as desired.

Choose some point $(x,y)\in\R\times(-\frac{\pi}{2},\frac{\pi}{2})$.
Then there are exactly two perpendicular lines, with slope $\pm 1$ that intersect at $(x,y)$.
Let $(z_1,\frac{\pi}{2})$ and $(z_2,\frac{\pi}{2})$ be the points where these two lines intersect $\R\times\{\frac{\pi}{2}\}$.
There are unique $\lambda$ functions that hit these points and so $(x,y)$ is in the image of both.
By Proposition \ref{prop:lambda intersections}, no other $\lambda$ functions will intersect these two at $(x,y)$.
\end{proof}

\begin{construction}\label{con:interior of M}
For most indecomposable representations $M$, we'll map $M$ to a specific intersection of distinct $\lambda$ functions from Definition \ref{def:lambda functions}.
Recall Notation \ref{note:indefinite intervals}.
Let $M_{|a,b|}$ be an indecomposable that is not projective, not simple, and $|a,b|\neq[s_n,s_n+1]$.
\begin{itemize}
\item If $a\notin S$, use $\lambda_a^*= \lambda_a$.
\item If $a\in S$ and $a\notin |a,b|$, use $\lambda_a^*= \lambda^+_{a}$.
If $a\in S$ and $a\in |a,b|$, use $\lambda_a^*=\lambda^-_{a}$.
\item If $b\notin S$, use $\lambda_b^*= \lambda_b$.
\item If $b\in S$ and $b\notin |a,b|$, use $\lambda_b^*=\lambda^-_{b}$.
If $b\in S$ and $b\in |a,b|$, use $\lambda_b^*=\lambda^+_{b}$. \hfill $\diamond$
\end{itemize}
\end{construction}

\begin{lemma}\label{lem:good intersection point}
Let $M_{|a,b|}$, $\lambda_a^*$, and $\lambda_b^*$ be as in Construction \ref{con:interior of M}.
Then
\begin{align*}
x_a < n&\pi +\frac{1}{2}(\kappa_a^* + \kappa_b^*) \leq x_a+\pi \\
x_b < n&\pi +\frac{1}{2}(\kappa_a^* + \kappa_b^*) \leq x_b+\pi
\end{align*}
where $n=0$ if $\kappa_b^* = \kappa_b^-$ and $n=1$ if $\kappa_b^* = \kappa_b^+$.
The inequalities are false for other values of $n$.
\end{lemma}
\begin{proof}
There are four cases to check, based on which kappa values have $-$ or $+$.
We start with $\kappa_a^-$ and $\kappa_b^-$.
By Proposition \ref{prop:kappa order}, and Proposition \ref{prop:lambda covers} we know $\kappa_a^- < \kappa_b^-$.
By definition, $x_a \leq \kappa_a^-$ and $\kappa_b^- \leq x_b+\pi$.
Then $x_a < \frac{1}{2} (\kappa_a^- + \kappa_b^-) < x_b+\pi$.

We explicitly show the following technique as we will use it several times throughout the rest of the proof of the proposition.
By Definition \ref{def:M on the rest of the projectives} and Proposition \ref{prop:M is sort of injective}, $|x_a-x_b| \leq y_b-y_a$.
If $y_a = -\frac{\pi}{2}$ then $|x_a-x_b|$ must be strictly less than $y_b-y_a$ or else $M_{|a,b|}$ would be projective (Theorem \ref{thm:projectives}).
If $|x_a-x_b|=y_b-y_a$ then we know $y_a > -\frac{\pi}{2}$.
In either case we begin the same:
\begin{align*}
0 &< \frac{1}{2}(y_a + y_a) + \frac{\pi}{2} + \frac{1}{2}(y_b-y_a) + \frac{1}{2}(x_a - x_b) \\
0 &< \frac{1}{2} (y_a + y_b) + \underbrace{\frac{1}{2}(y_a-y_a)}_{=0} + \frac{1}{2} (x_a-x_b) + \frac{\pi}{2} \\
0 &< \frac{1}{2} (y_a + y_b) + \frac{1}{2}(x_a-x_b) + \underbrace{\frac{1}{2}(x_b-x_b)}_{=0} + \frac{\pi}{2} \\
0 &< \frac{1}{2} (y_a + y_b) + \frac{1}{2}(x_a+x_b) +\frac{\pi}{2} - x_b \\
x_b &< \frac{1}{2}(\kappa_a^- + \kappa_b^-) 
\end{align*}
Finally we show $\frac{1}{2}(\kappa_a^- + \kappa_b^-) \leq x_a + \pi$.
Since $x_b-x_a\leq y_b-y_a$, we have $x_b - x_a \leq \pi -(y_a+y_b)$.
Then
\begin{align*}
\frac{1}{2}(x_a - x_a) + \frac{1}{2}(x_b-x_a) &\leq \frac{\pi}{2} - \frac{1}{2}(y_a+y_b) \\
\frac{1}{2}(x_a + x_b)+ \frac{1}{2}(y_a + y_b) &\leq x_a + \frac{\pi}{2}.
\end{align*}
Thus the equations hold for $\kappa_a^-$ and $\kappa_b^-$.

Now we check $\kappa_a^+$ and $\kappa_b^-$.
If $a$ or $b$ is a sink then there must be another sink or source between $a$ and $b$ or else $M_{|a,b|}$ is projective (Theorem \ref{thm:projectives}) or $|a,b|=[s_n,s_{n+1}]$.
Since we're checking $\kappa_b^-$, if $a$ is a source then $b$ cannot be between $a$ and the next sink.
Similarly, if $b$ is a source then $a$ cannot be between $b$ and the previous sink.
Thus, $(x_a,y_a)$ and $(x_b,y_b)$ do not lie on the same line segment from Definition \ref{def:M on the rest of the projectives}.
So we may start with $|x_a-x_b|<y_b-y_a$ and use the same technique as above:
\begin{align*}
\frac{1}{2}(x_a - x_b) &< \frac{1}{2}(y_b-y_a)\\
x_a &<  \frac{1}{2}(x_b+x_a) + \frac{1}{2}(y_b-y_a) +\frac{1}{2}\left(\frac{\pi}{2} - \frac{\pi}{2}\right)
\end{align*}
Thus, $\frac{1}{2}(\kappa_a^+ + \kappa_b^-) > x_a$.
By the same technique we see $\frac{1}{2}(\kappa_a^+ + \kappa_b^-) > x_b$.

Starting with $x_b-x_a < y_b - y_a \leq \pi$, we use the same technique:
\begin{align*}
\frac{1}{2}(x_a - x_b)  + \frac{1}{2}(y_b-y_a) &\leq \pi \\
\frac{1}{2}(x_a+ x_b) + \frac{1}{2}(y_b-y_a) + \frac{1}{2}\left(\frac{\pi}{2} - \frac{\pi}{2}\right) &\leq x_b + \pi
\end{align*}
Thus, $\frac{1}{2}(\kappa_a^+ + \kappa_b^-) \leq x_b+\pi$.
If we start with $x_b-x_a$ instead, we obtain $\frac{1}{2}(\kappa_a^+ + \kappa_b^-) \leq x_a+\pi$.

We now check $\kappa_a^-$ and $\kappa_b^+$.
As with the first case, if $y_a=-\frac{\pi}{2}$ then $|x_a-x_b|<y_b-y_a$ and if $|x_a-x_b|=y_b-y_a$ then $y_a>-\frac{\pi}{2}$.
In either case
\begin{align*}
0 &< \frac{1}{2}(x_a-x_b) - \frac{1}{2}(y_b-y_a) + \pi \\
x_b &< \frac{1}{2}(x_a+x_b) + \frac{1}{2}(y_a - y_b) + \pi
\end{align*}
Thus, $x_b < \frac{1}{2}(\kappa_a^- + \kappa_b^+) + \pi$.
Using the same technique starting with $x_b-x_a$ instead we obtain $x_b < \frac{1}{2}(\kappa_a^- + \kappa_b^+) + \pi$.
We also see that just using the technique with $x_a - x_b \leq y_b-y_a$ yields \[ \frac{1}{2}(x_a + x_b) -\frac{1}{2}(y_a-y_b) < x_b.\]
Adding $\pi$ to both sides yields $\frac{1}{2}(\kappa_a^- + \kappa_b^+) + \pi \leq x_b + \pi$.
By the same argument $\frac{1}{2}(\kappa_a^- + \kappa_b^+) + \pi < x_a + \pi$.

Finally we check $\kappa_a^+$ and $\kappa_b^+$.
Similar to above if $y_b = \frac{\pi}{2}$ then $|x_a-x_b|<y_b-y_a$ and if $|x_a-x_b| = y_b-y_a$ then $y_b<\frac{\pi}{2}$.
In either case $x_b - x_a < \pi - (y_a+y_b)$.
Then we have
\begin{align*}
x_b -\frac{1}{2}(x_a + x_b) &< \frac{\pi}{2} - \frac{1}{2}(y_a+y_b) \\
x_b &< \pi + \frac{1}{2}(x_a+x_b) - \frac{1}{2}(y_a+y_b) - \frac{\pi}{2}
\end{align*}
Thus, $x_b < \frac{1}{2}(\kappa_a^+ + \kappa_b^+) + \pi$.
By a similar argumetn $x_a < \frac{1}{2}(\kappa_a^+ + \kappa_b^+) + \pi$.

Since $y_a \geq -\frac{\pi}{2}$, we have $\frac{1}{2}(x_a - x_b) \leq \frac{1}{2}(y_a+y_b) + \frac{\pi}{2}$.
Then
\begin{align*}
-\frac{1}{2}(y_a+y_b) - \frac{\pi}{2} &\leq \frac{1}{2}(x_b-x_a) \\
%-\frac{1}{2}(y_a+y_b) - \frac{\pi}{2} &\leq x_b - \frac{1}{2}(x_a+x_b)\\
\frac{1}{2}(x_a+x_b) - \frac{1}{2}(y_a+y_b) - \frac{\pi}{2} &\leq x_b
\end{align*}
Adding $\pi$ to both sides yields $\frac{1}{2}(\kappa_a^+ + \kappa_b^+) +\pi <x_b + \pi$.
By a similar argument $\frac{1}{2}(\kappa_a^+ + \kappa_b^+) + \pi < x_a + \pi$.

In each case the range from $x_a$ to $x_a+\pi$ is $\pi$.
Thus, if a different value $n'$ is chosen in any case, the value of $n'\pi + \frac{1}{2}(\kappa_a^* + \kappa_b^*)$ is outside the given range.
This concludes the proof.
\end{proof}

\subsection{The Mapping}\label{sec:M is finally defined}
In this subsection we finish defining $\MM$ and prove two basic properties about its image.

\begin{definition}\label{def:the rest of M}
Let $M_{|a,b|}$, $\lambda_a^*$, and $\lambda_b^*$ be as in Construction \ref{con:interior of M}. Define $\MM M_{|a,b|}$ to be $\frac{1}{2}(\kappa_a^* + \kappa_b^*) + n\pi$ where $n=0$ if $\kappa_b^*=\kappa_b^-$ and $n=1$ if $\kappa_b^*=\kappa_b^+$.

Now consider $M_{\{a\}}$.
If $M_{\{a\}}$ is a simple injective then $a=s_n$ is a source and so define $M_{\{a\}}=(x_n+\pi,-y_n)$.
If $M_{\{a\}}$ is simple but not injective let $s_n, s_{n+1}$ be a sink and source such that $s_n<a<s_{n+1}$ in $\R$.
\begin{itemize} \item If $s_n$ is a source define $\MM M_{\{a\}} = (\hp_a, \frac{\pi}{2})$. \item If $s_n$ is a sink define $\MM M_{\{a\}} = (\hp_a+\pi,-\frac{\pi}{2})$.\end{itemize}

Finally, we consider $M_{[s_n,s_{n+1}]}$.
If $s_n$ is a source then $\lambda_{s_{n+1}}^- = \lambda_{s_n}^-$ so define $\MM M_{[s_n,s_{n+1}]} = (\hp_{s_n}+\pi,-\frac{\pi}{2})$.
If $s_n$ is a sink then $\lambda_{s_{n+1}}^+ = \lambda_{s_n}^+$ so define $\MM M_{[s_n,s_{n+1}]} = (\hp_{s_{n+1}},\frac{\pi}{2})$.
Since we already defined $\MM$ on projectives this concludes the definition of $\MM$.
\end{definition}

\begin{proposition}\label{prop:M on injectives}
Let $P_a$ be a projective indecomposable and $I_a$ and injective indecomposable in $\repAR$, both at $a$.
Then $\MM I_a = (x_a+\pi,-y_a)$ and so $y_{-a} = -y_a$.
\end{proposition}
\begin{proof}
If $I_a$ is a simple this is clear by Definition \ref{def:the rest of M}.
If $a=s_n$ is a sink then let $s_{n-1}$ and $s_{n+1}$ be the adjacent sources.
By the dual classification to Theorem \ref{thm:projectives} the support of $I_a$ is $[s_{n-1},s_{n+1}]$.
Then $(x_n,y_n)$ is one intersection point of $\lambda_{s_{n-1}}^-$ and $\lambda_{s_{n+1}}^+$.

By Lemma \ref{lem:good intersection point}, the image of $I_a$ must have $x$-coordinate greater than $x_{n+1}$, which is greater than $x_n$.
If $\repAR$ has this type of injective then $x_{n+1} - x_n < \pi$.
Then by Lemma \ref{lem:good intersection point} the next intersection is $(x_n+\pi,-y_n)$, which must be the coordinates of $\MM I_a$.

Now suppose $a$ is neither a sink nor a source.
Let $s_n, s_{n+1}$ be the sink and source such that $s_n < a < s_{n+1}$. %[
If $s_n$ is a sink then $I_a$ has support $|a,s_{n+1}]$ (recall Notation \ref{note:indefinite intervals}).
Whether or not $a$ is included, $\kappa_a^* = \kappa_a^-$.
The formula stipulates to find an intersection between $\lambda_a$ and $\lambda_{s_{n+1}}^+$.

However, $\lambda_{s_{n+1}}^+ = \lambda_{s_n}^+$ and $(x_a,y_a)$ is already one intersection point.
By Proposition \ref{prop:lambda intersections} the next intersection must be $(x_a+\pi,-y_a)$.
Since $x_a < x_{n+1}$ (Proposition \ref{prop:M is sort of injective}) this must be the coordinates of $\MM I_a$.
If $s_n$ is a source a similar argument shows the same result.
Finally, $-y_a = -\arctan(a) = \arctan(-a) = y_{-a}$.
\end{proof}

The next proposition shows us that $\MM$ maps all the indecomposables between the ``projective line'' and ``injective line'' (the images of the projectives and injectives, respectively).

\begin{proposition}\label{prop:between projectives and injectives}
Let $M_{|a,b|}$ be an indecomposable in $\repAR$ and $(x_M,y_M)$ the coordinates of $\MM M_{|a,b|}$.
Then there exists $c\in\R$ such that $x_c < x_M < x_{-c}+\pi$ and $y_c = y_M$.
\end{proposition} 
\begin{proof}
If $M_{|a,b|}$ is projective or injective the statement is trivially true.
Then suppose $M_{|a,b|}$ is neither.
If $M_{|a,b|}$ is simple or $|a,b|=[s_n,s_{n+1}]$ then by Proposition \ref{prop:p hat order} and Definition \ref{def:the rest of M} the statement is true.
So we assume $M_{|a,b|}$ is neither of these types of indecomposables as well.

Let $y_c = y_M$.
Then $c = \tan y_c$ and so $x_c$ is just the $x$-coordinate of $\MM P_c$.
By Lemma \ref{lem:good intersection point} we know $x_M > x_a$ and $x_M > x_b$.
By Definition \ref{def:the rest of M} $|y_a-y_M| \leq x_M - x_a$.
Then we have \[x_c - x_a \leq |y_a - y_c| =|y_a - y_M| \leq x_M - x_a.\]
However if $x_c-x_a = x_M -x_a$ either $x_c=x_M$ or the slope of the line connecting $(x_a,y_a)$ and $(x_c,y_c)$ is the negative of the slope connecting $(x_a,y_a)$ and $(x_M,y_M)$.
By Construction \ref{con:interior of M} and Definition \ref{def:the rest of M} we know $x_M\neq x_c$ and other case would imply $y_c\neq y_M$, both contradictions.
Thus, $x_c < x_M$.

Since either $|x_a-x_b|$ is strictly less than $y_b-y_a$ or $y_a,y_b\notin \{\pm\frac{\pi}{2}\}$, we see that $x_M < x_a+\pi$ and similarly for $x_b$.
As we have done before we start with $x_a - x_{-c} \leq |y_a - y_{-c}|$ and recall $y_M = -y_{-c}$.
\begin{align*}
x_a - x_{-c} &\leq |y_a - y_{-c}| \\
x_a- x_{-c} &\leq |y_M - (-y_a)| \\
x_a +\pi - (x_{-c}+\pi) &\leq x_a + \pi - x_M
\end{align*}
Thus, $x_M\leq x_{-c}+\pi$.
We know $\supp I_{-c}\neq \supp M_{|a,b|}$ by assumption.
By Proposition \ref{prop:lambda covers} this means that at least one of the $\lambda$ functions that determine $\MM I_{-c}$ must be different from the $\lambda$ functions that determine $\MM M_{|a,b|}$ and so one of the coordinates must differ.
Since $y_M=-y_{-c}$ the different coordinate must be the $x$-coordinate.
\end{proof}

\begin{proposition}\label{prop:contractablemapping}
The image of $\MM$ is contractable if and only if $A_\R$ has finitely many sinks and sources.
\end{proposition}
\begin{proof}
Suppose $S$ (the set of sinks and sources) is finite.
Then there are indecomposables with every pair of endpoints in $(\R\cup\{\pm\infty\})^2$.
Thus, every function $\lambda$ from Definition \ref{def:lambda functions} is used and by Proposition \ref{prop:lambda covers} every point $(x,y_c)$ for $x_c\leq x \leq x_{-c}+\pi$ is the intersection point of $\lambda_a$ and $\lambda_b$ corresponding to two projectives.
(The exception is the lack of projective or injective at $\pm \infty$, but the image of these non-existent indecomposables would be on the boundary of the range of $\MM$.)

Suppose $S$ is infinite.
Then $S$ is unbounded above or below.
If unbounded above, no indecomposable has endpoint $+\infty$ and so the line from $(x_{+\infty},\frac{\pi}{2})$ to $(x_{+\infty}+\pi,-\frac{\pi}{2})$ cannot be in the image of $\MM$.
However, there are injective representations to the right of this line.
Thus, the image of $\MM$ is not connected and therefore not contractable.
A similar argument holds if $S$ is unbounded below.
\end{proof}

\subsection{The Auslander-Reiten Topology}
In this subsection we define a topology on the isomorphism classes of indecomposable objects in $\repAR$, called the Auslander-Reiten topology.
This topology will help us define the Auslander-Reiten space in Section \ref{sec:AR space}.
We conclude with a proof that the interior of the image of the Hom support of an indecomposable is the same basic shape as in the discrete case.

\begin{definition}\label{def:AR topology} 
The \underline{Auslander-Reiten topology}, or AR-topology, on the set of (isomorphism classes of) indecomposables in $\repAR$ is the one where open sets are $\MM^{-1}(U)$ for each open set $U\subset\R^2$.
Note that the space of indecomposables with this topology is not Hausdorff.
\end{definition}
We provide two examples below.
In Propositions \ref{prop:between projectives and injectives} and \ref{prop:contractablemapping} we prove that these pictures are accurate.
\begin{example}\label{xmp:arspaces}
The AR-topology on indecomposable in $\repAR$, where $A_\R$ has the straight descending orientation, can be visualized as a triangle:
\begin{displaymath}\begin{tikzpicture}
\draw[thick] (-10,-2) -- (-6,2);
\draw[thick, ->] (-6,2) -- (-8,0);
\filldraw[fill=white] (-10,-2) circle[radius=.6mm];
\filldraw[fill=white] (-6,2) circle[radius=.6mm];
\draw (-8,-2.3) node[anchor=north] {$A_\R$};
\draw (0,-2.3) node[anchor=north] {AR-topology};
\draw[dashed] (-5,2) -- (5,2);
\draw[dashed] (-5,-2) -- (-4,-2);
\draw[dashed] (-4,-2) -- (5,-2);
\draw[draw=white!60!black,  pattern=crosshatch dots, pattern color=white!70!black] (-4,-2) -- (0,2) -- (4,-2) -- (-4,-2);
\draw[dashed] (-5,-1) -- (-2,2) -- (2,-2) -- (3,-1) -- (5,1);
\draw[dashed] (5, -1) -- (2,2) -- (-2,-2) -- (-5,1);
\draw[fill=black] (0,2) circle[radius=0.6mm];
\draw (0,2) node [anchor=south] {$(-\infty,+\infty)$};
\draw[fill=white] (-4,-2) circle[radius=0.6mm];
\draw[fill=white] (4,-2) circle[radius=0.6mm];
\draw[fill=black] (0,0) circle[radius=0.6mm];
\draw (0,-0.2) node [anchor=north] {$M_{|a,b|}$};
\draw[fill=black] (-1,1) circle[radius=0.6mm];
\draw (-1,1) node [anchor=east] {$P_{b|}$};
\draw[fill=black] (-3,-1) circle[radius=0.6mm];
\draw (-3,-1) node [anchor=east] {$P_{a|}$};
\draw[fill=black] (1,1) circle[radius=0.6mm];
\draw (1,1) node [anchor=west] {$I_{|a}$};
\draw[fill=black] (3,-1) circle[radius=0.6mm];
\draw (3,-1) node [anchor=west] {$I_{|b}$};
\draw[fill=black] (-2,-2) circle[radius=0.6mm];
\draw (-2,-2) node [anchor=north] {$M_{\{a\}}$};
\draw[fill=black] (2,-2) circle[radius=0.6mm];
\draw (2,-2) node [anchor=north] {$M_{\{b\}}$};
\draw (-5,-1) node [anchor=east] {$\lambda_b$};
\draw (-5,1) node [anchor=east] {$\lambda_a$};
\draw (-5,-2) node[anchor=east] {$-\frac{\pi}{2}$};
\draw (-5,2) node[anchor=east] {$\frac{\pi}{2}$};
\end{tikzpicture}\end{displaymath}
\end{example}
\begin{example}\label{xmp:arspaces2}
Suppose the sinks and sources are $s_{-2}=-\infty$, $s_{-1}=0$, $s_0=1$, and $s_1=+\infty$.
Recall that this means $0$ is a source and $1$ is a sink, from Definition \ref{def:AR}.
With this orientation, the AR-topology on the indecomposables of $\repAR$ can be visualized as:
\begin{displaymath}\begin{tikzpicture}
\draw[thick] (-10,-1) -- (-8,1) -- (-7,0) -- (-6,1);
\draw[thick,->] (-8,1) -- (-9,0);
\draw[thick,->] (-8,1) -- (-7.5,.5);
\draw[thick,->] (-6,1) -- (-6.5,.5);
\filldraw[fill=black] (-8,1) circle[radius=.6mm];
\filldraw[fill=black] (-7,0) circle[radius=.6mm];
\filldraw[fill=white] (-10,-1) circle[radius=.6mm];
\filldraw[fill=white] (-6,1) circle[radius=.6mm];
\draw (-8,-2.3) node[anchor=north] {$A_\R$};
\draw (0,-2.3) node[anchor=north] {AR-topology};
\draw[dashed] (-4, 2) -- (4,2);
\draw[dashed] (-4,-2) -- (4,-2);
\draw[draw=white!60!black, pattern=crosshatch dots, pattern color = white!70!black] (-3,-2) -- (-1,0) -- (-2,1) -- (-1,2) -- (1,2) -- (3,0) -- (2,-1) -- (3,-2) -- (-3,-2);
\draw[dashed] (-4,0) -- (-2,2) -- (2,-2) -- (3,-1) -- (4,0);
\draw[dashed] (-4, -2) -- (0,2) -- (4,-2);
\draw[fill=black] (-1,2) circle[radius=0.6mm];
\draw (-1,2) node [anchor =south] {$P_{+\infty}$};
\draw[fill=black] (1,2) circle[radius=0.6mm];
\draw (1,2) node [anchor =south] {$I_{-\infty}$};
\draw[fill=white] (-3,-2) circle[radius=0.6mm];
\draw[fill=white] (3,-2) circle[radius=0.6mm];
\draw[fill=black] (-1.5,1.5) circle[radius=0.6mm];
\draw (-1.5,1.5) node [anchor =east] {$P_{b|}$};
\draw[fill=black] (-1.5,0.5) circle[radius=0.6mm];
\draw (-1.5,0.5) node [anchor =east] {$P_{|a}$};
\draw[fill=black] (2.5,-1.5) circle[radius=0.6mm];
\draw (2.5,-1.5) node [anchor =west] {$I_{|b}$};
\draw[fill=black] (2.5,-0.5) circle[radius=0.6mm];
\draw (2.5,-0.5) node [anchor =west] {$I_{a|}$};
\draw[fill=black] (-1,1) circle[radius=0.6mm];
\draw (-1,1) node [anchor =west] {$M_{|a,b|}$};
\draw[fill=black] (0,2) circle[radius=0.6mm];
\draw (0,2) node [anchor =south] {$M_{\{a\}}$};
\draw[fill=black] (2,-2) circle[radius=0.6mm];
\draw (2,-2) node [anchor =north] {$M_{\{b\}}$};
\draw (-4,2) node [anchor=east] {$\frac{\pi}{2}$};
\draw (-4,-2) node [anchor=north east] {$-\frac{\pi}{2}$};
\draw (-4,0) node[anchor=east] {$\lambda_b$};
\draw (-4,-2) node[anchor = south east] {$\lambda_a$};
\end{tikzpicture}\end{displaymath}
\end{example}

\begin{lemma}\label{lem:hom support}
Let $V=M_{|a,b|}$ be an indecomposable such that $M_{|a,b|}$ is not simple, $M_{|a,b|}$ is not injective, and $|a,b|\neq [s_n,s_{n+1}]$.
Let $(x_V,y_V)=\MM V_{|a,b|}$ and $(x^V,y^V)=(x_V+\pi,-y_V)$.
If $M_{|a,b|}$ is not projective let $\lambda_a^*$ and $\lambda_b^*$ be as in Definition \ref{def:the rest of M}.
If $M_{|a,b|}$ is projective let $\lambda_a^*$ be the $\lambda$ function with positive slope at $(x_V,y_V)$ and $\lambda_b^*$ the $\lambda$ function with negative slope.
Let $H_V$ be the set of points $(x_W,y_W)$ in $\R^2$ such that $x_V<x_W<x^V$ and $\lambda_b^*(x_W) < y_W < \lambda_a^*(x_W)$.

If $\MM N\in H_V$ then $\Hom(V,N)\cong k$.
If $\MM N\notin \overline{H_V}$ then $\Hom(V,N)=0$.
\end{lemma}
Before we begin the proof we give the reader a visual guide to the statement of the proposition.
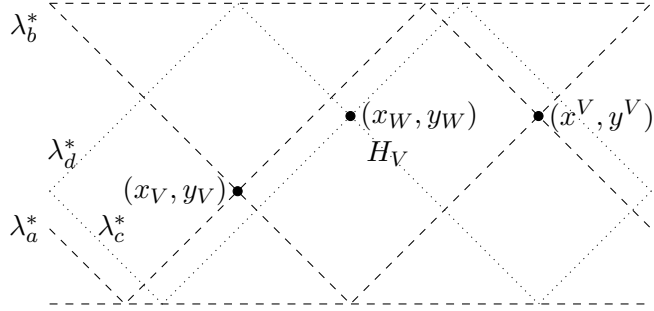
\begin{displaymath}\label{nice picture}\begin{tikzpicture}
\draw[dashed] (-4,2) -- (4,2);
\draw[dashed] (-4,-2) -- (4,-2);
\draw[dashed] (-4,2) -- (0,-2) -- (4,2);
\draw[dashed] (-4,-1) -- (-3,-2) -- (1,2) -- (4,-1);
\draw[fill=black] (-1.5,-0.5) circle[radius=0.6mm];
\draw (-1.5,-0.5) node[anchor=east] {$(x_V,y_V)$};
\draw[fill=black] (2.5,0.5) circle[radius=0.6mm];
\draw (2.5, 0.5) node [anchor=west] {$(x^V,y^V)$};
\draw[dotted] (-4,-0.5) -- (-1.5,2) -- (2.5,-2) -- (4,-0.5);
\draw[dotted] (-4,-0.5) -- (-2.5,-2) -- (1.5,2) -- (4,-0.5);
\draw[fill=black] (0,0.5) circle[radius=0.6mm];
\draw (0,0.5) node[anchor=west] {$(x_W,y_W)$};
\draw (-4,-1) node[anchor=east] {$\lambda^*_a$};
\draw (-4,2) node[anchor=north east] {$\lambda^*_b$};
\draw (-3.5,-1) node[anchor = west] {$\lambda^*_c$};
\draw (-3.5,0) node[anchor = east] {$\lambda^*_d$};
\draw (0.5,0) node {$H_V$};
\end{tikzpicture} \end{displaymath}
The reader should note the boundary of $H_V$ is not part of the proposition.
The boundary is more complicated and depends on exactly which indecomposable $V$ is used.
This will be covered in Section \ref{sec:AR space}.

\begin{proof}[Proof of Lemma \ref{lem:hom support}]
Let $W=M_{|c,d|}$.
Note that if $W$ is projective then so is $V$ and so $W$ would be on the boundary of $H_V$.
Thus we may assume $W$ is not projective.
Let $\lambda_c^*$ and $\lambda_d^*$ be as in Definition \ref{def:the rest of M}.
The following table summarizes necessary and sufficient conditions for $W\in H_V$ based on $\hp_a$, $\hp_b$, $\hp_c$, $\hp_d$, $\lambda_a^*$, and $\lambda_b^*$.
The fourth column is a consequence Proposition \ref{prop:p hat order} based on the second and third.
The fifth and sixth follow from Propositions \ref{prop:p hat order} and \ref{prop:lambda covers} given the second, third, and fourth.
In the table, $s_m$ and $s_n$ are the sink $s_m\preceq b$ and the source $a\preceq s_n$, respectively.
\smallskip

\centerline{\begin{tabular}{c|c|c|c|c|c}
Case & $\lambda_a^*$ & $\lambda_b^*$ & $\hp_a$ and $\hp_b$ & $\hp_c$ & $\hp_d$ \\ \hline
$- -$ & $\lambda_a^-$ & $\lambda_b^-$ & $\hp_a < \hp_b < x_{+\infty}$ & $\hp_a < \hp_c < \hp_b$ & $\hp_b < \hp_d < \hp_{s_n}$ \\
$+ -$ & $\lambda_a^+$ & $\lambda_b^-$ & $\hp_b < x_{+\infty} < \hp_a$ & $\hp_c < \hp_b$ or $\hp_a < \hp_c$ & $\hp_b < \hp_d < \hp_a$ \\
$- +$ & $\lambda_a^-$ & $\lambda_b^+$ & $\hp_a < x_{+\infty} < \hp_b$ & $\hp_a < \hp_c < \hp_{s_m}$ & $\hp_b < \hp_d < \hp_{s_n}$ \\
$+ + $ & $\lambda_a^+$ & $\lambda_b^+$ & $x_{+\infty} < \hp_b < \hp_a$ & $\hp_c<\hp_{s_m}$ or $\hp_a < \hp_c$ & $\hp_b < \hp_d < \hp_a$
\end{tabular}}
\smallskip

\noindent \underline{Claim}: These conditions imply $\Hom(V,W)\cong k$.

%We note that if $y_a<y_c$ then $a < c$ and vice versa; also, similarly with $y_b$, $y_d$, $b$, and $d$.
%We also see that the conditions imply $|a,b|\cap |c,d|\neq \emptyset$.
In the $- -$ case we immediately see that $a<c<b$ and for $l\in\Z$ such that $s_l \leq c \leq s_{l+1}$, $s_l$ is a sink.
The value of $d$ is either greater than $b$ or between $s_n$ and $b$ if $s_n<b$.
If $s_n < b < d$ then for $l\in\Z$ such that $s_l\leq c \leq s_{l+1}$, $s_l$ is a source.
Finally, $|a,b|\cap |c,d|\neq\emptyset$.
In this case we see that $\Hom(V,W)\neq 0$.

In the $+ -$ case we see $c<b$.
We also see $d> b$ or $a<d<b$ and for $l\in\Z$ such that $s_l\leq d \leq s_{l+1}$, $s_l$ is a source.
Again $|a,b|\cap |c,d|\neq\emptyset$ and $\Hom(V,W)\neq 0$.

The other two cases, $- +$ and $+ +$, follow similar reasoning.
In all cases, there exists a nontrivial morphism $V\to W$ and so $\Hom (V,W)\cong k$ by \cite[Theorem 3.0.1]{IgusaRockTodorov2019}.

Now suppose $W\notin \overline{H_V}$.
We break up $\R\times[-\frac{\pi}{2},\frac{\pi}{2}]$ into 7 regions, labeled 1--6 and $H_V$:
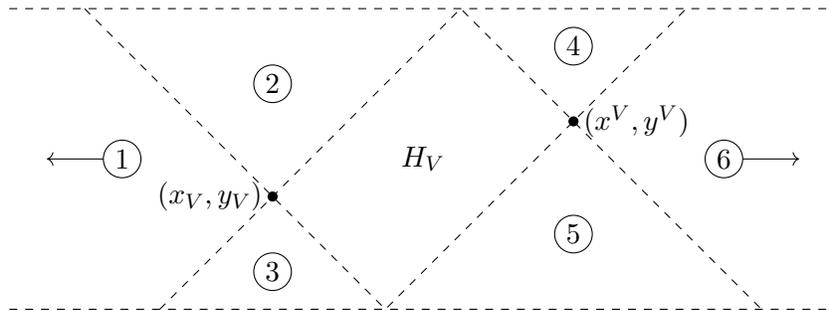
\begin{displaymath}\begin{tikzpicture}
\draw[dashed] (-5,2) -- (6,2);
\draw[dashed] (-5,-2) -- (6,-2);
\draw[dashed] (-4,2) -- (0,-2) -- (4,2);
\draw[dashed] (-3,-2) -- (1,2) -- (5,-2);
\draw[fill=black] (-1.5,-0.5) circle[radius=0.6mm];
\draw (-1.5,-0.5) node[anchor=east] {$(x_V,y_V)$};
\draw[fill=black] (2.5,0.5) circle[radius=0.6mm];
\draw (2.5, 0.5) node [anchor=west] {$(x^V,y^V)$};
\draw (0.5,0) node {$H_V$};
\draw[->] (-3.5,0) -- ( -4.5,0);
\draw[fill=white] (-3.5,0) circle[radius=2.5mm];
\draw (-3.5,0) node {1};
\draw[fill=white] (-1.5,1) circle [radius=2.5mm];
\draw (-1.5,1) node {2};
\draw[fill=white] (-1.5,-1.5) circle [radius=2.5mm];
\draw (-1.5,-1.5) node{3};
\draw[fill=white] (2.5,-1) circle [radius=2.5mm];
\draw (2.5,-1) node {5};
\draw[fill=white] (2.5,1.5) circle [radius=2.5mm];
\draw (2.5,1.5) node {4};
\draw[->] (4.5,0) -- (5.5,0);
\draw[fill=white] (4.5,0) circle[radius=2.5mm];
\draw (4.5,0) node {6};
\end{tikzpicture} \end{displaymath}
Then we have 6 regions to check.
Some regions have similar arguments.
In regions 2 and 5, $\hp_d$ meets the requirements in the table but $\hp_c$ does not.
In regions 3 and 4, $\hp_c$ meets the requirements but $\hp_d$ does not.
In regions 1 and 6, neither $\hp_c$ nor $\hp_d$  will meet the requirements of the table.

We first argue that if $\MM W$ is in regions 2 or 5 then $\Hom(V,W)=0$.
Assume we are in case $- -$ and suppose $\MM W$ is in region 2.
This means $\hp_c < \hp_a$.
So, there is $y\preceq x$ where: $x\geq a$ and $x\in|a,b|$ but $y\notin |a,b|$ and $y\in|c.d|$.
Then for any morphism $f:V\to W$, $f(x)$ must be 0.
This means $f$ is 0 and so $\Hom(V,W)=0$.
If $\MM W$ is in region 5 we instead have $\hp_c > \hp_b$ which means $c > b$.
Then $|a,b|\cap |c,d|=\emptyset$ and so $\Hom(V,W)=0$.

Now assume case $+ -$ and $\MM W$ is in region 2.
Then $x_{+\infty} \leq \hp_c < \hp_a$.
This means $b < c$ (and so $|a,b|\cap|c,d|=\emptyset$) or $a<c<b$.
In the latter case for $l\in\Z$ such that $s_l\leq c\leq s_{l+1}$ we have $s_l$ is a source.
However, then there exist $x$ and $y$ such that $x\in|a,b|\setminus |c,d|$, $y\in|a,b|\cap|c,d|$, and $y\preceq x$.
Thus $\Hom(V,W)=0$.
If $\MM W$ is in region 5 we have $\hp_b < \hp_c \leq x_{+\infty}$.
This forces $c>b$ and so $|a,c|\cap |b,c|=\emptyset$.

The arguments for cases $- +$ and $+ +$ are combinations of similar arguments.
The arguments for regions 3 and 4 are similar to those for regions 2 and 5 by considering $\hp_d$ instead of $\hp_c$.
This leaves regions 1 and 6.

If $\MM W$ is in region 1 then the consequences for region 2 apply to $\hp_c$ and the consequences for region 3 apply to $\hp_d$.
On the border of regions 1 and 2 (respectively the border of regions 1 and 3) the consequences for region 2 (respectively for region 3) still apply.
Thus, if $\MM W$ is in region 1 or the borders of region 1 and 2 or 1 and 3 then $\Hom(V,W)=0$.
By similar arguments if $\MM W$ is in region 6 or the borders of regions 4 and 6 or 5 and 6 then $Hom(V,W)=0$.
Therefore, if $W\notin\bar{H}_V$ then $\Hom(V,W)=0$.
\end{proof}

\section{Auslander-Reiten Sequences}\label{sec:AR sequences}
%
%\begin{definition}\label{def:almostsplit}
%A morphism $f:V\to W$ is \underline{irreducible} \cite{ARSequences} if $f$ is neither a retraction nor a section and if $f=gh$ then either $h$ is a section or $g$ is a retraction.
%
%A morphism $f:V\to W$ is \underline{left almost split} \cite{ARSequences} if it is not a section and for any morphism $g:V\to X$ that is not a section there exists $h:W\to X$ such that $g=h\circ f$.
%\underline{Right almost spit} is dual.
%\begin{displaymath}\text{left almost split }\xymatrix{ V\ar[r]^-f \ar[d]_-g & W \ar@{-->}[dl]^-{\exists h} & & V \ar[r]^-f & W
%\\ X & & & & Y\ar[u]_- g \ar@{-->}[ul]^-{\exists h} }\text{ right almost split} \end{displaymath}
%\end{definition}
%
%\begin{definition}\label{def:minimalmorphism}
%A morphism $f:V\to W$ is \underline{left minimal}, \cite{ARSequences} but presented as in \cite{bluebook}, if $\{g\in \text{End}(W):g\circ f = f\}\subset \text{Aut}(W)$.
%It is \underline{right minimal} if $\{h\in\text{End}(V):f\circ h=f\}\subset \text{Aut}(V)$.
%\end{definition}
%
%It is well known that if a left almost split morphism is mono it is irreducible and if a right almost split morphism is epi it is irreducible.
%Irreducible morphisms are, in turn, minimal.

Almost split sequences, commonly called Auslander-Reiten sequences, were introduced by Auslander and Reiten in \cite{ARSequences}.
In this section we will explicitly classify all Auslander-Reiten sequences in $\repAR$.
The general description of Auslander-Reiten sequences in the category of finite sums of interval indecomposable representations over a linearly-ordered poset were described by Gabriel and Ro\u{\i}ter in \cite{GabrielRoiter}.
However, an explicit description and proof in the contemporary language of representation theory specifically for continuous quivers of type $A$ is more enlightening for Section \ref{sec:AR space} and the future works in this series.

\subsection{Types of Auslander-Reiten Sequences}\label{sec:AR sequence types}
We first we recall the definition of an Auslander-Reiten sequence.

\begin{definition}\label{def:AR sequence} 
Let $0\to N\stackrel{f}{\to} E\to M\stackrel{g}{\to} 0$ be a short exact sequence in an abelian category.
It is called an \underline{almost split sequence}, or \underline{Auslander-Reiten sequence}, if the following hold:
\begin{itemize}
\item $f$ is not a section and $g$ is not a retraction
\item $N$ and $M$ are indecomposable objects
\item Any nontrivial morphism $N\to X$, respectively $X\to M$, of indecomposable objects where $N\not\cong X$, respectively $M\not\cong X$, factors through $f$, respectively $g$.
%\item Any nontrivial morphism $X\to M$ of indecomposable objects where $M\not\cong X$ factors through $g$.
\end{itemize}
\end{definition}

In this subsection we describe the 16 types of Auslander-Reiten sequences in $\repAR$, which are provided in Table \ref{tab:AR sequence table}.
They are proven to be the only types in Section \ref{sec:AR sequence classification}.

\begin{lemma}\label{lem:AR sequence existence}
Each of the 16 sequences in Table \ref{tab:AR sequence table} is an Auslander-Reiten sequence.
\end{lemma}

Before we begin with the proof, we provide pictures to give the reader intuition as to what these Auslander-Reiten sequences look like.
We refer the reader to Example \ref{xmp:arspaces2}, where $\bar{S}=\{-\infty,0,1,+\infty\}$ and $s_0=1$.
Let $V=M_{(0,5)}$ and $W=M_{(-\infty,5)}$.
There is an irreducible morphism $V\to W$. %(
To show this, let $U$ be the indecomposable with support $[0,5)$.
Any morphism $V\to U$ factors through an indecomposable $M_{(a,5)}$ for any $a<0$. %]

Further, any morphism from $V$ to such an indecomposable factors through another indecomposable $M_{(a-\e,5)}$ for all $\e>0$.
However, the morphism $V\to W$ does not factor through any other representation and is mono, thus is irreducible.
One can see this using the left picture in Example \ref{xmp:arspaces2}.
Consider the indecomposables $M_{(x,5)}$ as $x$ approaches $0$ from the left.
The support reaches 0 and ``spills over'' down to $-\infty$.
Afterwards the support can be ``drawn up'' back towards 0 from the right.

This described case a type (7) Auslander-Reiten sequence in Table \ref{tab:AR sequence table}.
The intuitive picture the reader should have is the following:
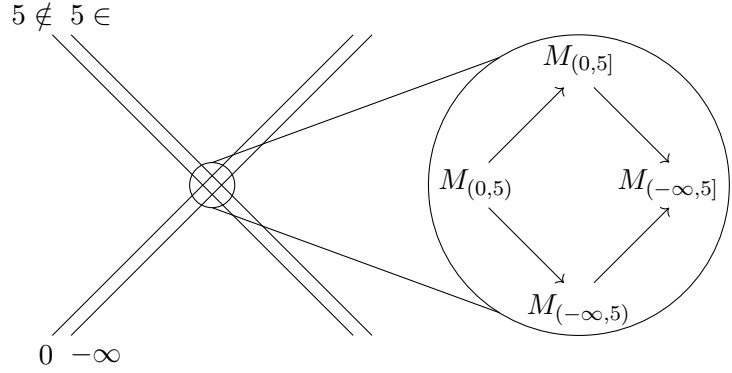
\begin{displaymath}\begin{tikzpicture}[scale=1]
\draw (0,0) -- (4,4);
\draw (0.25,0) -- (4.25,4);
\draw (0,4) -- (4,0);
\draw (0.25,4) -- (4.25,0);
\draw (0.15,0) node[anchor=north east] {$0$};
\draw (0.1,0) node[anchor=north west] {$-\infty$};
\draw (0.15,4) node[anchor=south east] {$5\notin $};
\draw (0.1,4) node[anchor=south west] {$5\in $};
\draw[fill opacity=0] (2.125,2) circle[radius=3mm];
\draw (2.125,2.3) -- (5.95,3.7);
\draw (2.125,1.7) -- (5.95,0.3);
\draw[fill opacity=0] (7,2) circle[radius=2cm];
\draw (5,2) node [anchor=west] {$M_{(0,5)}$}; %[ [
\draw (7,4) node [anchor=north] {$M_{(0,5]}$};
\draw (7,0) node [anchor=south] {$M_{(-\infty,5)}$};
\draw (9,2) node [anchor=east] {$M_{(-\infty,5]}$};
\draw[->] (5.8,2.3) -- (6.8,3.3);
\draw[->] (5.8,1.7) -- (6.8,0.7);
\draw[->] (7.2,3.3) -- (8.2,2.3);
\draw[->] (7.2,0.7) -- (8.2,1.7);
\end{tikzpicture}\end{displaymath}

In type (4), the values $a$ and $b$ are not sinks or sources and small neighborhoods around each have $\preceq$ that is the opposite of $\leq$.
That is, $a+\e\preceq a$ and similarly for $b$.
In that case we have the following picture:
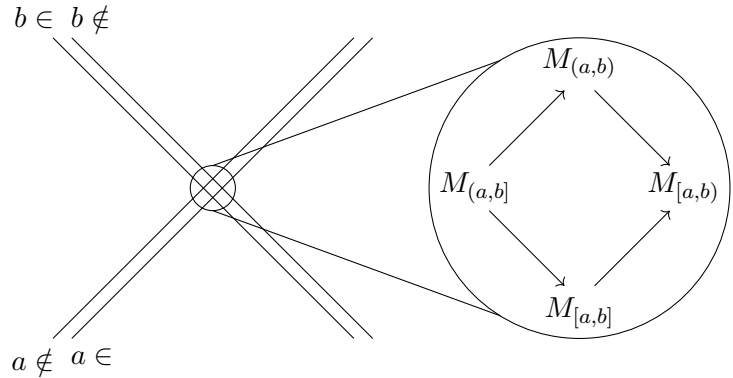
\begin{displaymath}\begin{tikzpicture}[scale=1]
\draw (0,0) -- (4,4);
\draw (0.25,0) -- (4.25,4);
\draw (0,4) -- (4,0);
\draw (0.25,4) -- (4.25,0);
\draw (0.15,0) node[anchor=north east] {{$a\notin$}};
\draw (0.1,0) node[anchor=north west] {{$a\in$}};
\draw (0.15,4) node[anchor=south east] {{$b\in$}};
\draw (0.1,4) node[anchor=south west] {{$b\notin$}};
\draw[fill opacity=0] (2.125,2) circle[radius=3mm];
\draw (2.125,2.3) -- (5.95,3.7);
\draw (2.125,1.7) -- (5.95,0.3);
\draw[fill opacity=0] (7,2) circle[radius=2cm]; %[
\draw (5,2) node [anchor=west] {{$M_{(a,b]}$}}; %)
\draw (7,4) node [anchor=north] {{$M_{(a,b)}$}};
\draw (7,0) node [anchor=south] {{$M_{[a,b]}$}}; %(
\draw (9,2) node [anchor=east] {{$M_{[a,b)}$}}; %]
\draw[->] (5.8,2.3) -- (6.8,3.3);
\draw[->] (5.8,1.7) -- (6.8,0.7);
\draw[->] (7.2,3.3) -- (8.2,2.3);
\draw[->] (7.2,0.7) -- (8.2,1.7);
\end{tikzpicture}\end{displaymath}

In type (15) both endpoints of a representation are in $\bar{S}$ and so the ``spilling over'' effect, as well as its dual, happens on both endpoints.
In this case, consider the four elements in $\bar{S}$: $s_{2m}<s_{2m+1}<s_{2n-1}<s_{2n}$.
We then have the following picture:
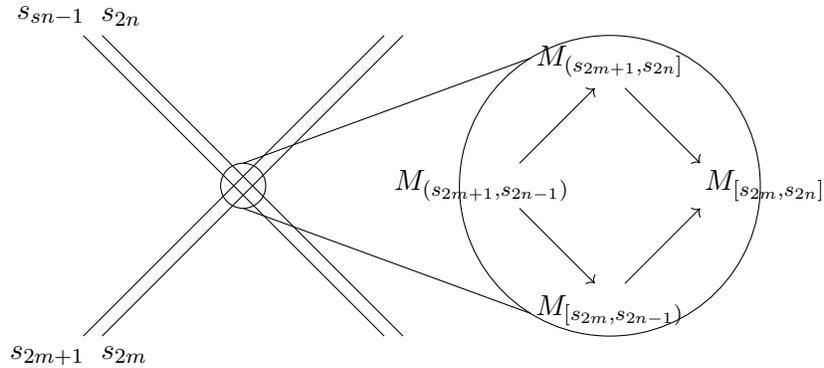
\begin{displaymath}\begin{tikzpicture}[scale=1]
\draw (0,0) -- (4,4);
\draw (0.25,0) -- (4.25,4);
\draw (0,4) -- (4,0);
\draw (0.25,4) -- (4.25,0);
\draw (0.15,0) node[anchor=north east] {{$s_{2m+1}$}};
\draw (0.1,0) node[anchor=north west] {{$s_{2m}$}};
\draw (0.15,4) node[anchor=south east] {{$s_{sn-1}$}};
\draw (0.1,4) node[anchor=south west] {{$s_{2n}$}};
\draw[fill opacity=0] (2.125,2) circle[radius=3mm];
\draw (2.125,2.3) -- (5.95,3.7);
\draw (2.125,1.7) -- (5.95,0.3);
\draw[fill opacity=0] (7,2) circle[radius=2cm];
\draw (4,2) node [anchor=west] {{$M_{(s_{2m+1},s_{2n-1})}$}}; %[
\draw (7,4) node [anchor=north] {{$M_{(s_{2m+1},s_{2n}]}$}}; %)(
\draw (7,0) node [anchor=south] {{$M_{[s_{2m},s_{2n-1})}$}}; %]
\draw (10,2) node [anchor=east] {{$M_{[s_{2m},s_{2n}]}$}};
\draw[->] (5.8,2.3) -- (6.8,3.3);
\draw[->] (5.8,1.7) -- (6.8,0.7);
\draw[->] (7.2,3.3) -- (8.2,2.3);
\draw[->] (7.2,0.7) -- (8.2,1.7);
\end{tikzpicture}\end{displaymath}

The table below describes the 16 types of Auslander-Reiten sequences in $\repAR$.
In Theorem \ref{thm:AR sequence classification} we prove there are no other types of Auslander-Reiten sequences in $\repAR$.
If a sink or source happens to be $\pm\infty$ we abuse notation and use $[$ or $]$ to avoid needlessly adding rows to the table.
The 16 types are grouped into 6 flavors depending on whether or not the endpoints $a<b$ of the supports of indecomposables are sinks or sources.
Types (1)--(4) have no endpoints that are sinks and sources.
Types (5)--(8) have $a$ as a sink or source.
Types (9)--(12) have $b$ as a sink or soruce.
Types (13)--(16) have both $a$ and $b$ as a sink or source.
Finally, each of the monomorphisms indicated are the diagonal map $\displaystyle\left[\begin{array}{r} 1\\1\end{array}\right]$.
The surjections are $[\begin{array}{rr} 1 & -1\end{array}]$. %\newpage

\centerline{\textbf{Table \ref{tab:AR sequence table}}}
\centerline{\refstepcounter{lemma}
\begin{tabular}{c|c|c|c}\label{tab:AR sequence table}
Type & Flavor & Condition & Auslander-Reiten sequence \\ \hline  % [ [ (
(1) & & $s_a<a, s_b<b$ & $M_{[a,b)}\hookrightarrow M_{[a,b]}\oplus M_{(a,b)} \twoheadrightarrow M_{(a,b]}$ \\
(2) & $a,b\notin \bar{S}$ & $s_a<a, s_b>b$ & $M_{[a,b]}\hookrightarrow M_{[a,b)}\oplus M_{(a,b]} \twoheadrightarrow M_{(a,b)}$ \\
(3) & & $s_a>a,s_b<b$ & $M_{(a,b)}\hookrightarrow M_{(a,b]}\oplus M_{[a,b)} \twoheadrightarrow M_{[a,b]}$ \\
(4) & & $s_a>a,s_b>b$ & $M_{(a,b]}\hookrightarrow M_{(a,b)}\oplus M_{[a,b]} \twoheadrightarrow M_{[a,b)}$ \\ \hline %] ) ) (
(5) & $s_{2n-1} < s_{2n}<b$ & $s_b<b$ & $M_{[s_{2n-1},b)} \hookrightarrow M_{[s_{2n-1},b]} \oplus M_{(s_{2n},b)} \twoheadrightarrow M_{(s_{2n},b]}$ \\
(6) & & $s_b>b$ & $M_{[s_{2n-1},b]} \hookrightarrow M_{[s_{2n-1},b)} \oplus M_{(s_{2n},b]} \twoheadrightarrow M_{(s_{2n},b)}$ \\ \hline %)
(7) & $s_{2n} < s_{2n+1} < b$ & $s_b<b$ & $M_{(s_{2n+1},b)} \hookrightarrow M_{(s_{2n+1},b]}\oplus M_{[s_{2n},b)} \twoheadrightarrow M_{[s_{2n},b]}$ \\
(8) & & $s_b>b$ & $M_{(s_{2n+1},b]} \hookrightarrow M_{(s_{2n+1},b)}\oplus M_{[s_{2n},b]} \twoheadrightarrow M_{[s_{2n},b)}$ \\ \hline
(9) & $a < s_{2n} < s_{2n+1}$ & $s_a<a$ & $M_{[a,s_{2n+1}]} \hookrightarrow M_{[a,s_{2n})} \oplus M_{(a,s_{2n+1}]} \twoheadrightarrow M_{(a,s_{2n})}$ \\
(10) & & $s_a>a$ &  $M_{(a,s_{2n+1}]} \hookrightarrow M_{(a,s_{2n})} \oplus M_{[a,s_{2n+1}]} \twoheadrightarrow M_{[a,s_{2n})}$ \\ \hline
(11) & $a < s_{2n-1} < s_{2n}$ & $s_a<a$ & $M_{[a,s_{2n-1})} \hookrightarrow M_{[a,s_{2n}]} \oplus M_{(a,s_{2n-1})} \twoheadrightarrow M_{(a,s_{2n}]}$ \\
(12) & & $s_a>a$ & $M_{(a,s_{2n-1})} \hookrightarrow M_{(a,s_{2n}]}\oplus M_{[a,s_{2n-1})} \twoheadrightarrow M_{[a,s_{2n}]}$ \\ \hline
(13) & & $\scriptstyle s_{2m} < s_{2m+1} < s_{2n} < s_{2n+1}$ & $M_{(s_{2m+1},s_{2n+1}]}\hookrightarrow M_{(s_{2m+1},s_{2n})}\oplus M_{[s_{2m},s_{2n+1}]} \twoheadrightarrow M_{[s_{2m},s_{2n})}$ \\
(14)& $a,b\in \bar{S}$ & $\scriptstyle s_{2m-1}<s_{2m} < s_{2n} < s_{2n+1}$ & $M_{[s_{2m-1},s_{2n+1}]}\hookrightarrow M_{[s_{2m-1},s_{2n})}\oplus M_{(s_{2m},s_{2n+1}]} \twoheadrightarrow M_{(s_{2m},s_{2n})}$ \\
(15) & & $\scriptstyle s_{2m} < s_{2m+1} < s_{2n-1} < s_{2n}$ & $M_{(s_{2m+1},s_{2n-1})}\hookrightarrow M_{(s_{2m+1},s_{2n}]}\oplus M_{[s_{2m},s_{2n-1})} \twoheadrightarrow M_{[s_{2m},s_{2n}]}$ \\
(16) & &  $\scriptstyle s_{2m-1} < s_{2m} < s_{2n-1} < s_{2n}$ & $M_{[s_{2m-1},s_{2n-1})}\hookrightarrow M_{[s_{2m-1},s_{2n}]}\oplus M_{(s_{2m},s_{2n-1})} \twoheadrightarrow M_{(s_{2m},s_{2n}]}$
\end{tabular}}

\begin{proof}[Proof of Lemma \ref{lem:AR sequence existence}]
The first and last term in each sequence is indecomposable.
It remains to check whether the monomorphisms and epimorphisms indicated form an exact sequence and satisfy Definition \ref{def:AR sequence}. 
Types (1)--(4) were essentially proven in \cite[Theorem 3.0.1]{IgusaRockTodorov2019}.

We now prove type (5).
Types (6), (9), and (10) are all similar.
Let $M_{|c,d|}$ be an indecomposable representation such that $M_{|c,d|}\not\cong M_{[s_{2n-1},b)}$ but $\Hom(M_{|c,d|},M_{[s_{2n-1},b)})\neq 0$.
By \cite[Theorem 2.3.2]{IgusaRockTodorov2019} we know that the Hom-space is then isomorphic to $k$ and in particular any morphism $h$ is determined by by any chosen morphism $h(x)$ where $x$ is in both supports.
Since $M_{|c,d|}\not\cong M_{[s_{2n-1},b)}$, at least one endpoint of their supports must be different.
If $b\in|c,d|$ then a morphism $h$ factors through $M_{[s_{2n},b]}$ by first including $b$ in the support.
The requisite diagrams for morphisms of representations will still be satisfied as they are the same for vertices not equal to $b$.

We now show that $c\geq s_{2n-1}$.
If $c<s_{2n-1}$ then any commutative square for $h(x)$ where $x<s_{2n-1}$ and $x\preceq s_{2n-1}$ would require $h(x)=0$ since such an $x$ is not in the support of $M_{[s_{2n-1},b)}$.
If $c> s_{2n-1}$ then $c\geq s_{2n}$ by the same reasoning.
In particular, if $c>s_{2n-1}$ then $s_{2n}\notin |c,d|$ and if $c=s_{2n-1}$ then $s_{2n-1}\in|c,d|$.
So, if $c>s_{2n-1}$ then $h$ must factor through $M_{(s_{2n-1},b)}$.
Therefore, the inclusion in type (5) satisfies Definition \ref{def:AR sequence}.
By dual arguments, the surjection in type (5) does also.

Of types (7), (8), (11), and (12) we prove type (7), since these are also all similar types.
The argument on the upper endpoint $b$ is the same as in type (5).
Let $M_{|c,d|}$ be an indecomposable that is not isomorphic to $M_{(s_{2n+1},b)}$ but $\Hom(M_{(s_{2n+1},b)},M_{|c,d|})\cong k$.
Note that $\Hom(M_{[s_{2n},b)},M_{[s_{2n+1},b)})\cong k$.
Thus if $c \geq s_{2n+1}$ then any morphism $h:M_{(s_{2n+1},b)}\to M_{|c,d|}$ factors through $M_{[s_{2n+1},b)}$ and so also through $M_{[s_{2n},b)}$.
If $c\leq s_{2n}$ then any $h$ also clearly factors through $M_{[s_{2n},b)}$.
If $s_{2n}<c<s_{2n-1}$ there is a map $M_{[s_{2n},b)}\to M_{|c,d|}$ since for all $s_{2n}<y\leq x <s_{2n+1}$, the commutative diagram involving $h(x)$ and $h(y)$ will still commute.
Therefore, the inclusion in type (7) satisfies Definition \ref{def:AR sequence}.
Again by dual arguments, so does the surjection in (7).

Types (13)--(16) are proven using the arguments about lower endpoints in type (5) or type (7), except on both endpoints.
Finally, we se that the monomorphisms are exactly the kernels of the epimorphisms and the epimorphisms are exactly the cokernels of the monomorphisms.
Therefore, each of the 16 sequences listed are Auslander-Reiten sequencs.
\end{proof}

\begin{proposition}\label{prop:M collapses AR sequences}
Let $V$ and $W$ be indecomposable representations that belong to the same Auslander-Reiten sequence.
Assume further that it is one of the types in Table \ref{tab:AR sequence table}.
Then $\MM V=\MM W$.
\end{proposition}
\begin{proof}
This is true by checking each type of sequence in Table \ref{tab:AR sequence table} against Definition \ref{def:the rest of M}.
\end{proof}

\subsection{Declaration and Proof of Classification}\label{sec:AR sequence classification}

\begin{theorem}\label{thm:AR sequence classification}
Let $0\to U\to V \to W \to 0$ be an Auslander-Reiten sequence in $\repAR$.
Then it is one of the 16 types in Table \ref{tab:AR sequence table}.
\end{theorem}
\begin{proof}
We will show that if $U$ is not one of the 16 possibilities for the left indecomposable then the sequence is not almost-split.
Dual arguments show that if $W$ is not one of the 16 possibilities for the right indecomposable then the sequence is not almost-split.

We start by showing the indecomposables in the middle of the 16 sequences cannot be the first term in an Auslander-Reiten sequence.
We will show types (1), (5), (6), and (13) as the other types are similar to one of these.
First we see any monomorphism $M_{[a,b]}\hookrightarrow M_{|c,d|}$ factors through $M_{[a,b+\e]}$ for sufficiently small $\e>0$.
Thus, there are no monomorphisms with source $M_{[a,b]}$ that satisfy Definition \ref{def:AR sequence} and thus no almost-split sequence that begins with $M_{[a,b]}$.

We instead examine $M_{(a,b)}$.
Similarly to $M_{[a,b]}$, any monomorphism must factor through $M_{(a,b+\e)}$ except $M_{(a,b)}\hookrightarrow M_{(a,b]}$.
However, the cokernel of this map is $M_{\{b\}}$, which by \cite[Theorem 3.0.1]{IgusaRockTodorov2019} cannot be the beginning or end of an Auslander-Reiten sequence.
Thus, there is no Auslander-Reiten sequence that begins with $M_{(a,b)}$.

Now consider $M_{[s_{2n-1},b]}$.
By the same argument as before there is no monomorphism with source $M_{[s_{2n-1},b]}$ that satisfies Definition \ref{def:AR sequence}.
Then again by the same argument, $M_{(s_{2n},b)}\hookrightarrow M_{(s_{2n},b]}$ satisfies the definition but the cokernel is simple so the epimorphism will not satisfy the definition.
The indecomposables in the middle of type (6) do not have any minimal monomorphisms by the above arguments.

We then consider $M_{(s_{2m+1},s_{2n})}$ in type (13).
If $s_{2n}\in|c,d|$ then $\Hom(M_{(s_{2m+1},s_{2n})},M_{|c,d|})=0$ by the proof of Lemma \ref{lem:AR sequence existence}.
Thus the monomorphism must extend the support below.
However, then it must factor through $M_{[s_{2m},s_{2n})}$.
The cokernel would then be $[s_{2m},s_{2m+1}]$, which we will show cannot be the beginning or end of an Auslander-Reiten sequence next.

Consider $M_{[s_n,s_{n+1}]}$.
Let $f:M_{[s_n,s_{n+1}]}\to M_{|a,b|}$ be a nontrivial morphism such that $M_{[s_n,s_{n+1}]}\not\cong M_{|a,b|}$.
Suppose $s_n$ is a source, since if $s_n$ is a sink the argument is dual.
Then $|a,b|=[s_n,b|$ by the arguments in the proof of Lemma \ref{lem:AR sequence existence}.
If $b>s_{n+1}$ then choose $c$ such that $s_{n+1}\preceq c$ and $s_{n+1}<c<b$. %(
Then there is a nontrivial composition $M_{[s_n,s_{n+1}]}\to M_{[s_n,c|}\to M_{[s_n,b|}$ but $\Hom(M_{[s_n,c|},M_{[s_n,s_{n+1}]})=0$ and $\Hom(M_{[s_n,b|},M_{[s_n,c|})=0$.
Finally, we know $|a,b|\not\subset [s_n,s_{n+1})$ since otherwise $\Hom(M_{[s_n,s_{n+1}]},M_{|a,b|})=0$. %(

Now suppose $f:M_{|a,b|}\to M_{[s_n,s_{n+1}]}$ is nontrivial and $M_{[s_n,s_{n+1}]}\not\cong M_{|a,b|}$.
Then $|a,b|=|a,s_{n+1}]$ by the proof of Lemma \ref{lem:AR sequence existence} again.
Dual to above, $a < s_n$ and so there exists $\e>0$ such that $f$ factors through $M_{[s_n-\e,s_{n+1}]}$.
Therefore, there can be no Auslander-Reiten sequence starting or ending with $M_{[s_n,s_{n+1}]}$.

This leaves projectives and injectives.
By \cite[Proposition 3.2.2]{IgusaRockTodorov2019} any morphism $P\to Q$ of projective indecomposables is a monomorphism or 0.
If $M_{\{s_{2n}\}}$ is a simple projective and $M_{|a,b|}\not\cong M_{\{s_{2n}\}}$, $\Hom(M_{|a,b|},M_{\{s_{2n}\}})=0$.
Any nontrivial morphism $M_{\{s\}}\to M_{|a,b|}$ factors through $M_{[s_{2n},s_{2n}+\e]}$ or $M_{[s_{2n}-\e,s_{2n}]}$ and is thus neither left nor right almost split.
Now suppose $P$ is projective that is not simple.
Then one or both endpoints of its supports are a sink and included.
At that endpoint, we have the same argument as with $M_{\{s_{2n}\}}$.

Thus, a monomorphism with projective source satisfying Definition \ref{def:AR sequence} must come from a projective with support $[s_{2n},b|$ or $|b,s_{2n}]$ where $b$ is not a sink.
If the support includes $b$ then there are no monomorphisms satisfying the definition by above arguments.
If $b$ is not included then the cokernel is simple or has support $[s_m,s_{m+1}]$ and so we do not have an Auslander-Reiten sequence.
Thus, a projective cannot begin an Auslander-Reiten sequence and dually an injective cannot end an Auslander-Reiten sequence.
By definition, an injective cannot begin an Auslander-Reiten sequence and a projective cannot end an Auslander-Reiten sequence.

Finally, we recall that, given a fixed source, targets of the morphisms in Definition \ref{def:AR sequence} are unique up to isomorphism.
(Targets of minimal monomorphisms are unique up to isomorphism.)
Therefore, the sequences in Table \ref{tab:AR sequence table} are the only Auslander-Reiten sequences in $\repAR$.
\end{proof}

\begin{cor}[to Theorem \ref{thm:AR sequence classification}]\label{cor:unique AR sequence}
Let $M_{|a,b|}$ be an indecomposable in $\repAR$ such that
\begin{itemize}
\item $M_{|a,b|}$ is not projective,
\item $M_{|a,b|}$ is not injective, and
\item $\MM M_{|a,b|}\neq (x,\pm\frac{\pi}{2})$, for some $x\in\R$.
\end{itemize}
Then, there exists a unique Auslander-Reiten sequence in $\repAR$ of one of the types in Table \ref{tab:AR sequence table} containing $M_{|a,b|}$. That is, an Auslander-Reiten sequence $0\to U\hookrightarrow V\twoheadrightarrow W\to 0$ in $\repAR$ such that $M_{|a,b|}\cong U$, $M_{|a,b|}\cong W$, or there exists $M_{|c,d|}$ such that $V\cong M_{|a,b|}\oplus M_{|c,d|}$.

If $M_{|a,b|}$ does not satisfy the above conditions then it does not belong to any Auslander-Reiten sequence.
\end{cor}
\begin{proof}
By Theorem \ref{thm:AR sequence classification} every Auslander-Reiten sequence in $\repAR$ has to be one of the types in Table \ref{tab:AR sequence table}.
With careful observation we see none of the indecomposable representations involved are projective, injective, simple, or have support $[s_m,s_{m+1}]$, where we do mean both $s_m$ and $s_{m+1}$ are in $\R$.
Furthermore, the 16 sequences exhibit every indecomposable representation that is not projective, injective, simple, or with support $[s_m,s_{m+1}]$.
If $M_{|a,b|}$ does not meet the requirements above then by the same theorem it is not part of an Auslander-Reiten sequence.
\end{proof}

\begin{remark}\label{rem:no AR translation}
The corollary forces us to accept that there can be no Auslander-Reiten transpose with the traditional properties.
To see this, consider a representation $M_{[a,b]}$ that belongs to an Auslander-Reiten sequence of type (2) in Table \ref{tab:AR sequence table}.
Then it cannot have the usual Auslander-Reiten sequence $\tau M_{[a,b]} \hookrightarrow E\twoheadrightarrow M_{[a,b]}$ as there is no Auslander-Reiten sequence of this form in $\repAR$.
\end{remark}

\subsection{Relation to $\MM$}
In this subsection we show show that $\MM M_{|a,b|}=\MM M_{|c,d|}$ if and only if $M_{|a,b|}$ and $M_{|c,d|}$ belong to the same Auslander-Reiten sequence.
We also show that $\repAR$ has the ``one way Hom'' property exhibited in representations of type $A_n$.

\begin{proposition}\label{prop:ARquotient}
Let $M_{|a,b|}\not\cong M_{|c,d|}$ be indecomposables in $\repAR$.
Then $\MM M_{|a,b|}=\MM M_{|c,d|}$ if and only if one of the following holds:
\begin{itemize}
\item they belong to the same Auslander-Reiten sequence or
\item they are both projectives at the same vertex or both injectives at the same vertex.
\end{itemize}
\end{proposition}
\begin{proof}
If $M_{|a,b|}$ and $M_{|c,d|}$ belong to the same Auslander-Reiten sequenc then by Theorem \ref{thm:AR sequence classification} it is one of the types in Table \ref{tab:AR sequence table}.
By Proposition \ref{prop:M collapses AR sequences}, $\MM M_{|a,b|}=\MM M_{|c,d|}$.
By Definitions \ref{def:M on projectives at s} and \ref{def:M on the rest of the projectives} and Proposition \ref{prop:M on injectives}, if $M_{|a,b|}$ and $M_{|c,d|}$ are both projective or both injective at the same $a\in\R$ then $\MM M_{|a,b|} = \MM M_{|c,d|}$.

We now assume $\MM M_{|a,b|}\neq \MM M_{|c,d|}$.
Suppose $M_{|a,b|}$ is a projective.
For contradiction, suppose $M_{|c,d|}$ is not projective but $\MM M_{|a,b|}=\MM M_{|c,d|}$ anyway.
Then in particular the $y$-coordinates are the same.
But by Lemma \ref{lem:good intersection point} the $x$-coordinate of $\MM M_{|c,d|}$ is strictly greater than the $x$-coordinate of $\MM M_{|a,b|}$, a contradiction.
We arrive a similar contradiction if $M_{|a,b|}$ is injective but $M_{|c,d|}$ is not.
Thus, $M_{|a,b|}$ and $M_{|c,d|}$ cannot both be projectives at the same vertex or both be injectives at the same vertex.

If both $M_{|a,b|}$ and $M_{|c,d|}$ are projective but $\MM M_{|a,b|}\neq\MM M_{|c,d|}$ then they are projectives at different vertices.
If both $M_{|a,b|}$ and $M_{|c,d|}$ are injective but $\MM M_{|a,b|}\neq\MM M_{|c,d|}$ then they are injectives at different vertices.

Now suppose neither $M_{|a,b|}$ nor $M_{|c,d|}$ is projective or injective.
Let $\lambda_a^*$, $\lambda_b^*$, $\lambda_c^*$, and $\lambda_d^*$ be as in Construction \ref{con:interior of M}.
Since $\MM M_{|a,b|} \neq \MM M_{|c,d|}$ we know either $\lambda_a^*\neq \lambda_c^*$ or $\lambda_b^* \neq \lambda_d^*$.
If $M_{|a,b|}$ belongs to an Auslander-Reiten sequence, the other three indecomposables in that sequence have the same pair of $\lambda$ functions by Proposition \ref{prop:lambda covers} and Proposition \ref{prop:M collapses AR sequences}.
Thus $M_{|c,d|}$ is not the same Auslander-Reiten sequence or else $\MM M_{|a,b|}=\MM M_{|c,d|}$, contradicting our assumption.

If $\MM M_{|a,b|}$ has $y$-coordinate $\pm\frac{\pi}{2}$ then by Corollary \ref{cor:unique AR sequence} $M_{|a,b|}$ does belong to an Auslander-Reiten sequence.
Thus $M_{|c,d|}$ certainly cannot belong the same one as $M_{|a,b|}$.
\end{proof}

This now allows us to state one of the expected properties of $\repAR$ that generalize from finitely generated representations of $A_n$.

\begin{proposition}\label{prop:one way hom}
Let $M_{|a,b|}\not\cong M_{|c,d|}$ be indecomposables in $\repAR$.
If $\Hom(M_{|a,b|},M_{|c,d|})\cong k$ then $\Hom(M_{|c,d|},M_{|a,b|})=0$.
\end{proposition}
\begin{proof}
By Lemma \ref{lem:hom support}, if $\MM M_{|a,b|}\neq \MM M_{|c,d|}$, then the statement follows.
If $\MM M_{|a,b|}=\MM M_{|c,d|}$ then by Proposition \ref{prop:ARquotient} either the morphisms $M_{|a,b|}\to M_{|c,d|}$ are irreducible or $M_{|a,b|}$ and $M_{|c,d|}$ are both projective at some vertex or both injective at some vertex.

In all cases, at one endpoint of the supports or the other we have one of the following cases.
\begin{itemize}
\item There is $x\in|c,d|\setminus |a,b|$ and $y\in|c,d|\cap|a,b|$ but $y\preceq x$.
\item There is $x\in|c,d|\cap |a,b|$ and $y\in|a,b|\setminus|c,d|$ but $y\preceq x$.
\end{itemize}
In both cases, any morphism $M_{|c,d|}\to M_{|a,b|}$ must be 0.
\end{proof}

\section{The Auslander-Reiten Space}\label{sec:AR space}
In this section we define the Auslander-Reiten space.
We prove that the properties about Auslander-Reiten sequences and other extensions in the Auslander-Reiten quiver for type $A_n$ generalize to this new space.

\subsection{The Auslander-Reiten Space}
In this subsection we introduce an extra generalized metric (Definition \ref{def:nonstandard metric}) so that we may introduce lines and slopes in Section \ref{sec:lines and slopes}.
We conclude the subsection with the definition of the Auslander-Reiten space.

\begin{remark}
Let $V\not\cong W$ be indecomposables in $\repAR$ such that one of the bulleted conditions in Proposition \ref{prop:ARquotient} hold.
Then $V$ and $W$ are topologically indistinguishable in the AR-topology on indecomposables of $\repAR$.
\end{remark}

\begin{definition}\label{def:position}
Let $V$ be an indecomposable in $\repAR$.
We define the \underline{position} of $V$ in the following way.
The positions are 1, 2, 3, or 4, thought to occupy one of four corners in a diamond:
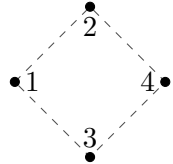
\begin{displaymath}
\begin{tikzpicture} \draw[dashed, draw opacity=.7] (0,0) -- (1,1) -- (2,0) -- (1,-1) -- (0,0);
\draw (0,0) node[anchor=west] {1};
\draw (1,1) node[anchor=north] {2};
\draw (1,-1) node[anchor=south]{3};
\draw (2,0) node[anchor=east]{4};
\filldraw[fill=black] (0,0) circle [radius=0.6mm];
\filldraw[fill=black] (1,1) circle [radius=0.6mm];
\filldraw[fill=black] (1,-1) circle [radius=0.6mm];
\filldraw[fill=black] (2,0) circle [radius=0.6mm];
 \end{tikzpicture}
\end{displaymath}
To use the words `above' and `below' we order the positions in a poset exactly as shown: $2> 1> 3$ and $2>4>3$ but $1$ and $4$ are not comparable.
\begin{itemize}
\item If $V$ is simple or of the form $[s_n,s_{n+1}]$ we define the position of $V$ to be 3 if the $y$-coordinate of $\MM V$ is $\frac{\pi}{2}$ and 2 if the $y$-coordinate is $-\frac{\pi}{2}$.
\item If $V$ is a projective or injective, then the position is more easily defined using pictures.
\\ \noindent Here are the projectives: 
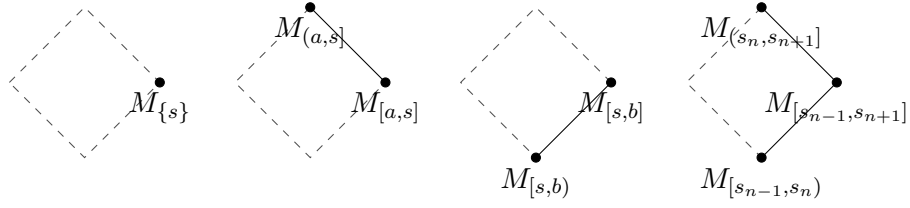
\begin{displaymath}
\begin{tikzpicture}
\draw[dashed, draw opacity=.7] (0,0) -- (1,1) -- (2,0) -- (1,-1) -- (0,0);
\draw[fill=black] (2,0) circle [radius=0.6mm];
\draw (2,0) node [anchor=north] {$M_{\{s\}}$};

\draw[dashed, draw opacity=.7] (4,1) -- (3,0) -- (4,-1) -- (5,0);
\draw (4,1) -- (5,0);
\draw[fill=black] (5,0) circle [radius=0.6mm];
\draw[fill=black] (4,1) circle [radius=0.6mm]; %[
\draw (4,1) node [anchor=north] {$M_{(a,s]}$}; %)
\draw (5,0) node[anchor=north] {$M_{[a,s]}$};

\draw[dashed, draw opacity=.7] (8,0) -- (7,1) -- (6,0) -- (7,-1);
\draw (7,-1) -- (8,0);
\draw[fill=black] (8,0) circle [radius=0.6mm];
\draw[fill=black] (7,-1) circle [radius=0.6mm];
\draw (8,0) node[anchor=north] {$M_{[s,b]}$}; %(
\draw (7,-1) node[anchor=north] {$M_{[s,b)}$}; %]

\draw[dashed, draw opacity=.7] (10,1) -- (9,0) -- (10,-1);
\draw (10,1) -- (11,0) -- (10,-1);
\draw[fill=black] (11,0) circle [radius=0.6mm];
\draw[fill=black] (10,-1) circle[radius=0.6mm];
\draw[fill=black] (10,1) circle[radius=0.6mm];
\draw (11,0) node[anchor=north] {$M_{[s_{n-1},s_{n+1}]}$}; %[
\draw (10,1) node[anchor=north] {$M_{(s_n,s_{n+1}]}$};
\draw (10,-1) node[anchor=north] {$M_{[s_{n-1},s_n)}$}; %]

\draw[opacity=0] (11,0) -- (13,0);
\draw[opacity=0] (-2,0) -- (0,0);
\end{tikzpicture}
\end{displaymath}
And here are the injectives:
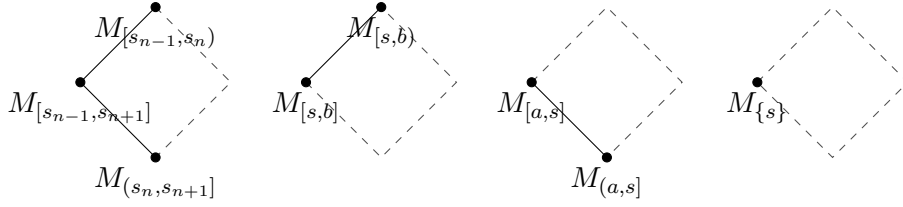
\begin{displaymath}
\begin{tikzpicture}
\draw[dashed, draw opacity=.7] (9,0) -- (10,1) -- (11,0) -- (10,-1) -- (9,0);
\draw[fill=black] (9,0) circle [radius=0.6mm];
\draw (9,0) node [anchor=north] {$M_{\{s\}}$};

\draw[dashed, draw opacity=.7] (6,0) -- (7,1) -- (8,0) -- (7,-1);
\draw (7,-1) -- (6,0);
\draw[fill=black] (6,0) circle [radius=0.6mm];
\draw[fill=black] (7,-1) circle [radius=0.6mm]; %[
\draw (7,-1) node [anchor=north] {$M_{(a,s]}$}; %)
\draw (6,0) node[anchor=north] {$M_{[a,s]}$};

\draw[dashed, draw opacity=.7] (3,0) -- (4,-1) -- (5,0) -- (4,1);
\draw (4,1) -- (3,0);
\draw[fill=black] (3,0) circle [radius=0.6mm];
\draw[fill=black] (4,1) circle [radius=0.6mm];
\draw (3,0) node[anchor=north] {$M_{[s,b]}$}; %(
\draw (4,1) node[anchor=north] {$M_{[s,b)}$}; %]

\draw[dashed, draw opacity=.7] (1,1) -- (2,0) -- (1,-1);
\draw (1,1) -- (0,0) -- (1,-1);
\draw[fill=black] (0,0) circle [radius=0.6mm];
\draw[fill=black] (1,-1) circle[radius=0.6mm];
\draw[fill=black] (1,1) circle[radius=0.6mm];
\draw (0,0) node[anchor=north] {$M_{[s_{n-1},s_{n+1}]}$}; %[
\draw (1,-1) node[anchor=north] {$M_{(s_n,s_{n+1}]}$};
\draw (1,1) node[anchor=north] {$M_{[s_{n-1},s_n)}$}; %]

\draw[opacity=0] (-2,0) -- (0,0);
\draw[opacity=0] (11,0) -- (13,0);
\end{tikzpicture}
\end{displaymath}
%Consider a pair of projectives or pair of injectives $V_1$ and $V_2$.
%We say $V_1$ is above $V_2$ if $V_1$'s position is above $V_2$'s position.
\item If $V$ does not fit the two cases above, by Corollary \ref{cor:unique AR sequence} it belongs to a unique Auslander-Reiten sequenc $V_1\hookrightarrow V_2\oplus V_3\twoheadrightarrow V_4$.
The representations $V_2$ and $V_3$ are algebraically interchangeable.
Using Table \ref{tab:AR sequence table}, we say $V_2$ and $V_1$ have the same lower bound on their support (including openness) and $V_3$ and $V_1$ have the same upper bound on their support (including openness).
So, if $V=V_i$ then we say the position of $V$ is $i$:
\begin{displaymath}
V_1 \hookrightarrow \begin{array}{c} V_2 \\ \oplus \\ V_3 \end{array} \twoheadrightarrow V_4.
\end{displaymath}
In Table \ref{tab:AR sequence table} the Auslander-Reiten sequences are written as $V_1\hookrightarrow V_2\oplus V_3\twoheadrightarrow V_4$.
\end{itemize}
\end{definition}

\begin{example}\label{xmp:position example}
Consider type (2) in Table \ref{tab:AR sequence table}.
Then $M_{[a,b]}$ has position 1, $M_{[a,b)}$ has position 2, $M_{(a,b]}$ has position 3, and $M_{(a,b)}$ has position 4.

In type (13), $M_{(s_{2m+1},s_{2n+1}]}$ has position 1,  $M_{(s_{2m+1},s_{2n})}$ has position 2, $M_{[s_{2m},s_{2n+1}]}$ has position 3, and $M_{[s_{2m},s_{2n})}$ has position 4.
\begin{displaymath}\begin{tikzpicture}
\draw[dashed, draw opacity=.7] (1,1) -- (0,0) -- (1,-1) -- (2,0) -- (1,1);
\draw[fill=black] (0,0) circle [radius=0.6mm]; %1
\draw (0,0) node[anchor=north] {$M_{[a,b]}$};
\draw[fill=black] (1,1) circle [radius=0.6mm]; %2
\draw (1,1) node[anchor=north] {$M_{[a,b)}$};
\draw[fill=black] (1,-1) circle [radius=0.6mm]; %3
\draw (1,-1) node[anchor=north] {$M_{(a,b]}$};
\draw[fill=black] (2,0) circle [radius=0.6mm]; %4
\draw (2,0) node[anchor=north] {$M_{(a,b)}$};
\draw (1,-2) node[anchor=north] {Type (2)};

\draw[dashed, draw opacity=.7] (5,0) -- (6,-1) -- (7,0) -- (6,1) -- (5,0);
\draw[fill=black] (5,0) circle [radius=0.6mm]; %1
\draw (5,0) node[anchor=north] {$M_{(s_{2m+1},s_{2n+1}]}$};
\draw[fill=black] (6,1) circle [radius=0.6mm]; %2
\draw (6,1) node[anchor=north] {$M_{(s_{2m+1},s_{2n})}$};
\draw[fill=black] (6,-1) circle [radius=0.6mm]; %3
\draw (6,-1) node[anchor=north] {$M_{[s_{2m},s_{2n+1}]}$};
\draw[fill=black] (7,0) circle [radius=0.6mm]; %4
\draw (7.3,0) node[anchor=north] {$M_{[s_{2m},s_{2n})}$}; %]
\draw (6,-2) node[anchor=north] {Type (13)};
\end{tikzpicture}\end{displaymath}
\end{example}

\begin{definition}\label{def:nonstandard metric}
Let $X$ be a set and $(G,\leq)$ a totally ordered abelian group (a totally ordered set where the commutative group operation preserves order).
Let $d:X\times X\to G$ be a function such that
\begin{enumerate}
\item $d(x_1,x_2)=e$ if and only if $x_1=x_2$,
\item $d(x_1,x_2)=d(x_2,x_1)$, and
\item $d(x_1,x_2) + d(x_2,x_3) \geq d(x_1,x_3)$.
\end{enumerate}
Then we call $d$ an \underline{extra generalized metric} on $X$.
\end{definition}

Generalized metrics are taken over arbitrary totally ordered fields, but we only want to use the abelian group structure on a ring that is not a field.
Hence, we need a notion of a metric that is extra generalized.
\begin{notation}\label{note:real metric}
We will denote by $\dR$ the standard metric on $\R^2$ and by $\dAR$ the extra generalized metric in Definition \ref{def:d function} (Proposition \ref{prop:d is an extra generalized metric}).
\end{notation}

\begin{definition}\label{def:d function}
Give the abelian group $\R\oplus\Z$ the total ordering $(x,m) \leq (y,n)$ if $x\leq y$ or if $x=y$ and $m\leq n$.
We will sometimes write $-(x,m)$ to mean $(-x,-m)$.
We define a function $\dAR$ from the (isomorphism classes of) indecomposable objects in $\repAR$ with the AR-topology to $\R\oplus\Z$.
For two indecomposables $V,W$, the $\R$-coordinate of $\dAR(V,W)$ is $\dR(\MM V,\MM W)$.
If $\MM V=\MM W$ then the $\Z$-coordinate is the number of edges between the positions of $V$ and $W$; this will be 0, 1, or 2.

If $\MM V\neq\MM W$, the $\Z$-coordinate depends on the relative locations of $\MM V$ and $\MM W$ in $\R^2$.
The line segment from $\MM V$ to $\MM W$ has some slope $r$ (possibly $\infty$).
We define four possible cases, interchanging the roles of $V$ and $W$ if necessary since this does not affect the $\R$-coordinate.
To calculate the $\Z$-coordinate we first merge two diamonds as shown below:
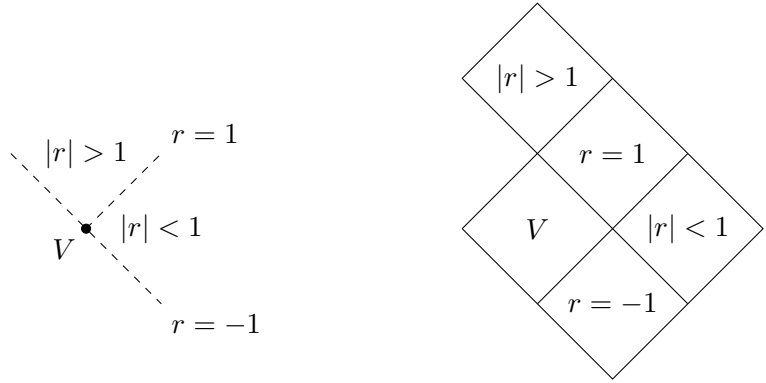
\begin{displaymath} \begin{tikzpicture}
\filldraw (0,0) circle [radius=0.6mm];
\draw[dashed] (-1,1) -- (1,-1);
\draw[dashed] (0,0) -- (1,1);
\draw(0,1) node {$|r| > 1$};
\draw(1,0) node {$|r| < 1$};
\draw (1,1) node[anchor=south west] {$r=1$};
\draw (1,-1) node[anchor=north west] {$r=-1$};
\draw (0,0) node [anchor=north east] {$V$};

\draw (5,0) -- (7, -2) -- (9,0) -- (6,3) -- (5,2) -- (6,1) -- (5,0);
\draw (6,1) -- (7,2);
\draw (6,1) -- (8,-1);
\draw (6,-1) -- (8,1);
\draw (6,0) node {$V$};
\draw (8,0) node {$|r| < 1$};
\draw (7,1) node {$r=1$};
\draw (7,-1) node {$r=-1$};
\draw (6,2) node {$|r|>1$};
\end{tikzpicture} \end{displaymath}
If the positions of $V$ and $W$ are the same, the $\Z$-coordinate is 0.
Otherwise, we use these tables to compute the $\Z$-coordinate.
\smallskip 

\begin{tabular}{|l|c|c|r|}
$r$ & $V$ & $W$ & $\Z$ \\ \hline
$|r| < 1$ & 1 & 2 & 1 \\
$|r| < 1$ & 1 & 3 & 1 \\
$|r| < 1$ & 1 & 4 & 2 \\
$|r| < 1$ & 2 & 1 & -1 \\
$|r| < 1$ & 2 & 3 & 0 \\
$|r| < 1$ & 2 & 4 & 1 \\
$|r| < 1$ & 3 & 1 & -1 \\
$|r| < 1$ & 3 & 2 & 0 \\
$|r| < 1$ & 3 & 4 & 1 \\
$|r| < 1$ & 4 & 1 & -2 \\
$|r| < 1$ & 4 & 2 & -1 \\
$|r| < 1$ & 4 & 3 & -1 
\end{tabular}\hfill
\begin{tabular}{|l|c|c|r|}
$r$ & $V$ & $W$ & $\Z$ \\ \hline
$r=1$ & 1 & 2 & 1 \\
$r=1$ & 1 & 3 & 1 \\
$r=1$ & 1 & 4 & 2 \\
$r=1$ & 2 & 1 & -1 \\
$r=1$ & 2 & 3 & 0 \\
$r=1$ & 2 & 4 & 1 \\
$r=1$ & 3 & 1 & 0 \\
$r=1$ & 3 & 2 & 1 \\
$r=1$ & 3 & 4 & 1 \\
$r=1$ & 4 & 1 & -1 \\
$r=1$ & 4 & 2 & 0 \\
$r=1$ & 4 & 3 & -1
\end{tabular}\hfill
\begin{tabular}{|l|c|c|r|}
$r$ & $V$ & $W$ & $\Z$ \\ \hline
$r=-1$ & 1 & 2 & 1 \\
$r=-1$ & 1 & 3 & 1 \\
$r=-1$ & 1 & 4 & 2 \\
$r=-1$ & 2 & 1 & 0 \\
$r=-1$ & 2 & 3 & 1 \\
$r=-1$ & 2 & 4 & 1 \\
$r=-1$ & 3 & 1 & -1 \\
$r=-1$ & 3 & 2 & 0 \\
$r=-1$ & 3 & 4 & 1 \\
$r=-1$ & 4 & 1 & -1 \\
$r=-1$ & 4 & 2 & -1 \\
$r=-1$ & 4 & 3 & 0
\end{tabular}\hfill
\begin{tabular}{|l|c|c|r|}
$r$ & $V$ & $W$ & $\Z$ \\ \hline
$|r|>1$ & 1 & 2 & 1 \\
$|r|>1$ & 1 & 3 & -1 \\
$|r|>1$ & 1 & 4 & 0 \\
$|r|>1$ & 2 & 1 & -1 \\
$|r|>1$ & 2 & 3 & -2 \\
$|r|>1$ & 2 & 4 & -1 \\
$|r|>1$ & 3 & 1 & 1 \\
$|r|>1$ & 3 & 2 & 2 \\
$|r|>1$ & 3 & 4 & 1 \\
$|r|>1$ & 4 & 1 & 0 \\
$|r|>1$ & 4 & 2 & 1 \\
$|r|>1$ & 4 & 3 & -1 
\end{tabular}%\hfill \begin{tabular}{c}{~} \\{~} \\{~} \\{~} \\{~} \\{~} \\{~} \\{~} \\{~} \\{~} \\{~} \\{~} \\ $\diamond$\end{tabular}\\
\smallskip

The reader can see that these values are obtained by ``teleporting'' from the diamond at $\MM V$ to the diamond at $\MM W$ while retaining the position of $V$, and then traveling forwards or backwards a number of edges to the position of $W$.
\end{definition}

\begin{proposition}\label{prop:d is an extra generalized metric}
The function $\dAR$ in Definition \ref{def:d function} is an extra generalized metric.
\end{proposition}
\begin{proof}
Conditions (1) and (2) in Definition \ref{def:nonstandard metric} are clear by construction.
Condition (3) follows from straightforward calculations.
\end{proof}

\begin{proposition}\label{prop:topology compatible with metric}
Open balls of radius $(\delta,3)$, $\delta>0\in\R$, generate the open sets in the indecomposables of $\repAR$.
\end{proposition}
\begin{proof}
An open ball of radius $(\delta,3)$ is the inverse image of an open ball of radius $\delta$ in $\R^2$.
Since the open balls generate the topology of $\R^2$ their inverse images generate the topology on the indecomposables of $\repAR$.
\end{proof}

\begin{definition}\label{def:AR space}
The set of isomorphism classes of indecomposables in $\repAR$ with the AR-topology (Definition \ref{def:AR topology}), positions for each indecomposable (Definition \ref{def:position}), and extra generalized metric $d$  (Definition \ref{def:d function}) is called the \underline{Auslander-Reiten space}, or AR-space, of $A_\R$.
\end{definition}

\subsection{Lines and Slopes}\label{sec:lines and slopes}
In this subsection we define lines and slopes in the AR-space which generalize lines and slopes in traditional metric spaces.

\begin{definition}\label{def:line segment}
Let $U$ and $W$ be indecomposables in $\repAR$.
Let $l$ be a set of indecomposables containing $U$ and $W$ such that
\begin{enumerate}
\item for any $V\in l$, $\dAR(U,V)+\dAR(V,W)=\dAR(U,V)$,
\item for any $V_1$, $V_2$, and $V_3$ in $l$, if $\dAR(V_1,V_2)<\dAR (V_1,V_3)$ and $\dAR(V_2,V_3)<\dAR(V_1,V_3)$ then $\dAR(V_1,V_2)+\dAR(V_2,V_3)=\dAR(V_1,V_3)$,
\item $l$ is maximal with respect to property (2).
\end{enumerate}
Then $l$ is a \underline{line segment} in $\repAR$ and $U$ and $W$ are its \underline{endpoints}.
The \underline{length} of a line segment is $\dAR(U,W)$.
If the length is $(0,0)$ we say the line segment is \underline{degenerate}, otherwise it is \underline{nondegenerate}.
\end{definition}

\begin{remark}
Let $l$ be a line segment in the AR-space.
Then $\MM l$, the set $\{\MM V: V\in l\}$, is a line segment in $\R^2$.
\end{remark}

\begin{proposition}\label{prop:M coordinate ordering in lines}
Let $l$ be a line segment with endpoints $U$ and $W$ and let $V\in l$.
Let $(x_U,y_U)=\MM U$, $(x_V,y_V)=\MM V$, and $(x_W, y_W)=\MM W$.
Then if $x_U<x_W$ and $y_U<y_W$,
\begin{align*}
x_U\leq &x_V\leq x_W \\
y_U\leq &y_V\leq y_W.
\end{align*}
\end{proposition}
\begin{proof}
This follows from (1) in Definition \ref{def:line segment}.
\end{proof}

\begin{definition}\label{def:slope}
The \underline{slope} of a nondegenerate line segment  $l$ is a pair $(r_1,r_2)$ in $(\R\cup\{\infty\})\times(\Q\cup\{\infty\})$.
We define $r_1$ to be the slope of $\MM l$ when that slope is defined.
If $\MM U = \MM W$ we define $r_1$ to be equal to the $r_2$ coordinate.
The second coordinate, $r_2$, is the slope of a line connecting the positions of the endpoints after adjoining the diamonds as in Definition \ref{def:d function} with $r=r_1$, counting horizontal and vertical movement as 1.
The $r_2$-coordinate can be $0$, $\pm\frac{1}{3}$, $\pm 1$, $\pm 3$, or $\infty$.
If the length of a line segment $l$ is $(0,0)$ we instead say the slope is \underline{undefined}.

We say two nondegenerate line segments $l$ and $l'$ are \underline{perpendicular} if $|\MM l_1\cap \MM l_2|=1$ and their slopes are $(r_1,r_2)$ and $(\frac{1}{r_1},\frac{1}{r_2})$, respectively, where we consider $0=\frac{1}{\infty}$ and $\infty=\frac{1}{0}$ for this purpose.
\end{definition}

\begin{example}\label{xmp:slope example}
Let $V$ and $W$ be indecomposables in $\repAR$.
Suppose $\MM V=(0,\frac{1}{4})$ and $\MM W = (\frac{1}{4},\frac{1}{4}+\e)$.
Then the slope of a line from $V$ to $W$ has $r_1$ between $-1$ and $1$.
Suppose $V$ has position 2 and $W$ has position 3; then $r_2=-1$.
If $V$ instead has position 1 then $r_2=-\frac{1}{3}$.
\end{example}

\begin{proposition}\label{prop:3 or fewer on a line}
Let $l$ be a line segment in the AR-space and suppose $V_1,V_2,V_3,V_4\in l$ such that $\MM V_i=\MM V_j$ for all $i,j$.
Then two of the $V_i$s must be isomorphic.
\end{proposition}
\begin{proof}
Let $(r_1,r_2)$ be the slope of $l$.
There are four cases, depending on $r_1$.
However, $r_1=0$ is dual to $r_1=\infty$ and $r_1=1$ is dual to $r_1=-1$.
So, we shall prove the cases $r_1=0$ and $r_1=1$.
For contradiction and without loss of generality, suppose the position of $V_i$ is $i$, so no two $V_i$s are isomorphic.

If $r_1=0$ then $\dAR(U,V_2)=\dAR(U,V_3)$ and $\dAR(V_2,W)=\dAR(V_3,W)$ but $V_2\neq V_3$.
If $r_1=1$ and $U$ has position 1 or 2 then $\dAR(U,V_2)+\dAR(V_2,V_3)\neq \dAR(U,V_3)$ and $\dAR(U,V_3)+\dAR(V_3,V_2)\neq \dAR(U,V_2)$.
If $r_1=1$ and $U$ has position 3 or 4 then $\dAR(U,V_3)+\dAR(V_2,V_3)\neq \dAR(U,V_2)$ and $\dAR(U,V_2)+\dAR(V_3,V_2)\neq \dAR(U,V_3)$.
\end{proof}

\begin{remark}
We frequently say ``a line'' instead of ``the line'' when using a pair of indecomposables as endpoints to define a line segment.
This is because two lines with the same slope and endpoints may be different.
For example, suppose the slope of a nondegenerate line segment $l$ from $\MM V$ to $\MM W$ is $0$ in $\R^2$.
Let $l_1$ be the line in the AR-space such that the position of every indecomposable in the line is 1,2, or 4 and $\MM l_1=l$.
Let $l_2$ be a similar line except the positions are 1, 3, or 4.
By Definition \ref{def:line segment} and Proposition \ref{prop:3 or fewer on a line} these are both valid line segments but $l_1\neq l_2$.
\end{remark}

\begin{proposition}\label{prop:no back and forth}
Let $l$ be a line segment with slope $(1,r_2)$ and suppose $\dAR(V_1,V_2)+\dAR(V_2,V_3)=\dAR(V_1,V_3)$ for $V_1,V_2,V_3\in l$.
If the positions of $V_1$ and $V_3$ are both 3 or 4, so is the position of $V_2$.
If the positions of $V_1$ and $V_3$ are both 1 or 2, so is the position of $V_2$.

Similarly, when the slope is $(-1,r_2)$ and the positions of $V_1$ and $V_3$ are 1 or 3 so is the position of $V_2$ and if the positions of $V_1$ and $V_3$ are 2 or 4 so is the position of $V_2$.
In these caese $r_2=-1$.
\end{proposition}
\begin{proof}
Note that if $r_1=1$ in the proposition then $r_2=1$ and if $r_1=-1$ then $r_2=-1$.
We prove the case where the slope is $(1,1)$ since the proof when the slope is $(-1,-1)$ is similar.

Suppose the positions of $V_1$ and $V_3$ are 1 or 2 and let $V$ be an indecomposable in $\repAR$ with position 3 such that $\MM V$ is on the line segment $\MM l$ between $\MM V_1$ and $\MM V_2$.
Without loss of generality, suppose the $x$-coordinate of $\MM V_1$ is less than the $x$-coordinate of $\MM V_2$.
The $\Z$-coordinate of $\dAR(V_1, V)$ is $1$ or $0$ if the position of $V_1$ is 1 or 2, respectively.
The $\Z$-coordinate of $\dAR(V,V_2)$ is $0$ or $1$ if the position of $V_2$ is 1 or 2, respectively.

If $\dAR(V_1,V)+\dAR(V,V_2)$ has $\Z$-coordinate 0 then $\dAR(V_1,V_2)$ has $\Z$-coordinate $-1$.
If $\dAR(V_1,V)+\dAR(V,V_2)$ has $\Z$-coordinate 1 then $\dAR(V_1,V_2)$ has $\Z$-coordinate $0$.
If $\dAR(V_1,V)+\dAR(V,V_2)$ has $\Z$-coordinate 2 then $\dAR(V_1,V_2)$ has $\Z$-coordinate $1$.
In all three cases, $\dAR(V_1,V)+\dAR(V,V_2)\neq \dAR(V_1,V_2)$ and so $V\notin l$.
If $V$ has position 4 instead of 3 the proof is similar.
Then, the proof if $V_1$ and $V_2$ have positions 3 or 4 and $V$ has position 1 or 2 is similar.
\end{proof}

\begin{proposition}\label{prop:diagonal endpoints}
Let $M_{|a,b|}$ and $M_{|c,d|}$ be indecomposable reprsentations in $\repAR$ such that the $x$-coordinate of $\MM M_{|a,b|}$ is less than or equal to the $x$-coordinate of $\MM M_{|c,d|}$ and $M_{|a,b|}\not\cong M_{|c,d|}$.
\begin{enumerate}
\item If the slope of a line in the AR-space with endpoints $M_{|a,b|}$ and $M_{|c,d|}$ is $(1,1)$ then $a=c$ and $a\in|a,b|$ if and only if $c\in|c,d|$.
\item If the slope of a line in the AR-space with endpoints $M_{|a,b|}$ and $M_{|c,d|}$ is $-(1,1)$ then $b=d$ and $b\in|a,b|$ if and only if $d\in|c,d|$.
\end{enumerate}
\end{proposition}
\begin{proof}
We'll prove (1) as the proof of (2) is similar.
First suppose $M_{|a,b|}$ is simple.
The argument for $|a,b|=[s_n,s_{n+1}]$ is similar.
Then $M_{|c,d|}$ is not simple or with support $[s_n,s_{n+1}]$ or else the slope of the line with endpoints $M_{|a,b|}$ and $M_{|c,d|}$ is not $\pm(1,1)$.
Since the slope is $(1,1)$, $\MM M_{|a,b|}$ has $y$-coordinate $-\frac{\pi}{2}$.
Thus $M_{|a,b|}$ has position 2 and so $M_{|c,d|}$ has position 1 or 2 by Proposition \ref{prop:no back and forth}.

Since $M_{|a,b|}$ is simple $|a,b|=\{a\}$ and so $\MM M_{|c,d|}$ lies on $\lambda_a^-$.
By Construction \ref{con:interior of M} if $a$ is a sink or source then $a\in|c,d|=|a,d|$.
If $a$ is neither a sink nor a source then, by Definition \ref{def:lambda functions}, for $l\in\Z$ such that $s_l<a<s_{l+1}$ we have $s_l$ is a sink.
Thus, by Definition \ref{def:position} and Table \ref{tab:AR sequence table}, $a\in|c,d|=|a,d|$.

Now suppose $M_{|a,b|}$ is not simple.
By assumption $\MM M_{|a,b|}$ and $\MM M_{|c,d|}$ are both on the graph of $\lambda_a^*$.
By Proposition \ref{prop:no back and forth} if the position of $\MM_{|a,b|}$ is 1 or 2 so is the position of $M_{|c,d|}$.
Then $M_{|c,d|}$ is neither simple nor has support of the form $[s_n,s_{n+1}]$.
By Corollary \ref{cor:unique AR sequence} $M_{|a,b|}$ belongs to a unique Auslander-Reiten sequence and $M_{|c,d|}$ belongs to a (possibly different) unique Auslander-Reiten sequence.

If $a$ is neither a sink nor a source we use Table \ref{tab:AR sequence table} and Definition \ref{def:position} and see that statement (1) follows.
If $a$ is a sink or source the same table and definition show $a$ is a source.
Since $M_{|c,d|}$ also has position 1 or 2 statement (1) follows again.

Now suppose $M_{|a,b|}$ has position 3 or 4.
If $a$ is not a sink or source and $M_{|c,d|}$ is simple then $|c,d|=\{a\}$ by Definition \ref{def:the rest of M} and $a\in|a,b|$ by Table \ref{tab:AR sequence table} and Definition \ref{def:position} again.
If $M_{|c,d|}$ is not simple then by the same table and definition we see statement (1) follows.
If $a$ is a sink or source then by the same table and definition again $a$ is a sink.
If $|c,d|=[s_n,s_{n+1}]$ for some $n\in\Z$ then $s_n$ is a sink and by the table and definition again, along with Construction \ref{con:interior of M}, $s_n=a\in|a,b|$.
If $|c,d|$ is not of the form $[s_n,s_{n+1}]$ then by the same table, definition, and construction again statement (1) follows.
\end{proof}

\begin{proposition}\label{prop:diagonal endpoints converse}
The converse to the statements in Proposition \ref{prop:diagonal endpoints} are also true.
\end{proposition}
\begin{proof}
We prove the contrapositive to (1) in Proposition \ref{prop:diagonal endpoints} as the contrapositive to (2) is similar.
Suppose $M_{|a,b|}\not\cong M_{|a,d|}$ are indecomposables in $\repAR$ such that $a\in|a,b|$ if and only if $a\in|a,d|$.
If $M_{|a,b|}$ and $M_{|a,d|}$ belong to the same Auslander-Reiten sequence in Table \ref{tab:AR sequence table} then the proposition follows.
So, suppose they do not.

By Definition \ref{def:the rest of M} both $\MM M_{|a,b|}$ and $\MM M_{|a,d|}$ lie on the graph of $\lambda_a^*$.
First suppose neither is simple or with support of the form $[s_n,s_{n+1}]$.
Then using Table \ref{tab:AR sequence table} and Definition \ref{def:position} we see that $M_{|a,b|}$ and $M_{|a,d|}$ must both have position 1 and/or 2 or both have position 3 and/or 4.
In either case the converse to (1) in Proposition \ref{prop:diagonal endpoints} holds.

If $M_{|a,b|}$ is simple or has support $[s_n,s_{n+1}]$ then $M_{|a,d|}$ does not.
So, $M_{|a,b|}$ has position 2.
Observing Table \ref{tab:AR sequence table} and Definition \ref{def:position} we see $M_{|a,d|}$ must have position 2 or 3.
Therefore, the slope is $(1,1)$.
If instead $M_{|a,d|}$ is simple or with support $[s_n,s_{n+1}]$ we use a similar argument. 
\end{proof}

By Proposition \ref{prop:diagonal endpoints converse}, line segments with endpoints in the left two columns below have slope in the third column (assuming both representations are nonzero).
\smallskip

\centerline{\textbf{Table \ref{tab:slope table}}}
\centerline{\begin{tabular}{c|c|r}\refstepcounter{lemma}\label{tab:slope table}
$V$ & $W$ & $(r_1,r_2)$ \\ \hline
$M_{|a,b|}$ & $M_{|a,d|}$ & $(1,1)$ \\
$M_{|a,b|}$ & $M_{|c,b|}$ & $-(1,1)$ \\
$M_{|a,d|}$ & $M_{|c,d|}$ & $-(1,1)$ \\
$M_{|c,b|}$ & $M_{|c,d|}$ & $(1,1)$
\end{tabular}}
\smallskip

\begin{proposition}\label{prop:homline}
Let $l$ be a line segment in the AR-space of $A_\R$ with slope $\pm(1,1)$ and endpoints $V,W$.
Then either $\Hom(V,W)\cong k$ or $\Hom (W,V)\cong k$.
\end{proposition}
\begin{proof}
If $\MM V=\MM W$ the statement is true since they separated only by a minimal morphism.
So, suppose $\MM V\neq \MM W$ and by symmetry, suppose the $x$-coordinate of $\MM V$ is less than the $x$-coordinate of $\MM W$.
We'll assume the slope is $(1,1)$ since if the slope is $(-1,-1)$ the argument is similar.

Since the slope is  $(1,1)$, $\MM V$ and $\MM W$ lie on the same $\lambda_a^*$ from Definition \ref{def:lambda functions} and by Proposition \ref{prop:no back and forth} the positions of $V$ and $W$ are, without loss of generality, (i) 1 and/or 2 or (ii) 3 and/or 4.
By Proposition \ref{prop:diagonal endpoints} we see that the lower bounds of $V$ and $W$ are the same.
First, assume the $y$-coordinates of $\MM V$ and $\MM W$ are not $\pm\frac{\pi}{2}$.

So, $W$ is on the boundary of $H_V$ and region 2 in Lemma \ref{lem:hom support}.
Then the $\hp$ values for the upper bounds of the supports of $V$ and $W$ satisfy the conditions in the table in the proof of Lemma \ref{lem:hom support}.
Thus, there is a nontrivial morphism and by \cite[Theorem 3.0.1]{IgusaRockTodorov2019} $\Hom(V,W)\cong k$.
In either case, there is a nonzero map $V\to W$.

For contradiction, suppose the $y$-coordinates of both $\MM V$ and $\MM W$ are $\pm\frac{\pi}{2}$.
But then by Definition \ref{def:slope} the slope of the line from $V$ to $W$ is not $(1,1)$.
Thus at least one of $\MM V$ or $\MM W$ has $y$-coordinate not equal to $\pm\frac{\pi}{2}$.

If $V$ is a non-projective simple or has support of the form $[s_n,s_{n+1}]$ then $\MM V=-\frac{\pi}{2}$ by assumption of the slope of the line.
If $V$ is the simple at some $a$ then for $l\in\Z$ such that $s_l < a < s_{l+1}$ then $s_l$ is a sink.
Then there is a nontrivial morphism from $V$ to any indecomposable with support of the form $[a,b|$, including $W$.
If $V$ has support of the form $[s_n,s_{n+1}]$ then $s_n$ is a source.
Since the position of $V$ in this case must be 2 the position of $W$ is 1 or 2.
Then $W$ belongs to an Auslander-Reiten sequence of type (5), (6), (14), or (16) (Table \ref{tab:AR sequence table}).
In each case the support of $W$ contains the support of $V$ and there is a nontrivial morphism $V$ to $W$.
\end{proof}

\begin{proposition}\label{prop:homline slope}
Let $l$ be a line segment in the AR-space of $A_\R$ with endpoints $V$ and $W$.
If the slope of $l$ is greater than $(1,1)$ or less than $(-1,-1)$, $\Hom(V,W)=0=\Hom(W,V)$.
\end{proposition}
\begin{proof}
In Lemma \ref{lem:hom support} it was shown that if the slope is $(r_1,r_2)$ and $|r_1|>1$ then $\Hom(V,W)=0=\Hom(W,V)$.
It remains to check the slopes $(1,3)$ and $(-1,-3)$.
However, the argument is the same as in the proof of Proposition \ref{prop:one way hom}.
\end{proof}

\subsection{Properties of the AR-space}
In this subsection we prove that the AR-space has the desired properties.
To do this we define rectangles in the AR-space, which generalize rectangles in metric spaces that permit them.
We conclude by proving that extensions with two indecomposables are in one-to-one correspondence with rectangles in the AR-space.

\begin{definition}\label{def:rectangle}
A \underline{rectangle} is a set of four line segments $l_1,l_2,l_3,l_4$ where
\begin{itemize}
\item the image $\MM( l_1\cup l_2\cup l_3\cup l_4 )$ is a (possibly degenerate) rectangle in $\R^2$;
\item the slopes of $l_i$ and $l_j$ are the same (possibly undefined) if and only if $l_i\cap l_j=\emptyset$ or $l_i=l_j$; and
\item the slopes of $l_i$ and $l_j$ are distinct (possibly one undefined) if and only if $|l_i\cap l_j|=1$.
\end{itemize}
The intersection points of the $l_i$s are called the \underline{corners} of the rectangle.
If one or more of the $l_i$s has length $(0,0)$ we say the rectangle is \underline{degenerate}.
\end{definition}

We now finish the description of Hom-support started in Lemma \ref{lem:hom support}.
Proposition \ref{prop:homline} and Proposition \ref{prop:homline slope} show what happens if $\MM W$ is on left-most two lines of the boundary of $H_V$ for some $V$.

\begin{proposition}\label{prop:can't reach too far}
Let $V$ and $W$ be indecomposables in $\repAR$ such that $\MM W$ is on the boundary of $H_V$ from Lemma \ref{lem:hom support}.
(In particular, the $x$-coordinate of $\MM V$ is not $\pm\frac{\pi}{2}$.)
Further assume $\MM W$ is on a part of the boundary of $H_V$ that borders region 4, 5, or 6 in the same proposition.
\begin{itemize}
\item If the position of $V$ is 1 then $\Hom(V,W)=0$.
\item If the position of $V$ is 2 and $\MM W$ does not border region 5 then $\Hom(V,W)=0$.
\item If the position of $V$ is 3 and $\MM W$ does not border region 4 then $\Hom(V,W)=0$.
\end{itemize}
\end{proposition}
\begin{proof}
Note that $V$ cannot be projective if its position is 1, by Definition \ref{def:position}.
We know it cannot be injective or else there would exist no $W$ in the proposition, by Lemma \ref{lem:good intersection point}.
Furthermore, $\MM W$ cannot be on the corner bordering region 6 by the same lemma.
Let $|a,b|=\supp V$.

Suppose $\MM W$ borders region 4.
Then since $\MM V$ and $\MM W$ are both on the graph of $\lambda_a^*$, we must have $\lambda_a^*=\lambda_a^+$.
If $a$ is neither a sink nor a source then for $l\in\Z$ such that $s_l<a<s_{l+1}$ we have $s_l$ is a source and $a$ is the upper bound of $\supp W$.
Then $a\notin\supp V$ and so $\Hom(V,W)=0$.

If $a$ is a sink or source then $a\notin\supp |a,b|$ by Construction \ref{con:interior of M} and Definition \ref{def:the rest of M}.
We know $V$ belongs to a unique Auslander-Reiten sequence by Corollary \ref{cor:unique AR sequence}.
Since $a=s_n$ is a sink or source, using Theorem \ref{thm:AR sequence classification} and Table \ref{tab:AR sequence table}, we see that $V$ must be a source and belong to an Auslander-Reiten sequence of type (7), (8), (13), or (15).

Also by Construction \ref{con:interior of M} and Definition \ref{def:the rest of M} we see the upper limit of $\supp W$ must be a sink or source and contained in the support.
In particular, it must be the source $s_n$ or the sink $s_{n-1}$.
However, $s_n\notin\supp V$ and so $\supp V\cap \supp W=\emptyset$.
A similar argument applies if $\MM W$ borders region 5 in Lemma \ref{lem:hom support}.
Furthermore, if $V$ has position 2 the argument for region 4 applies and if $V$ has position 3 the argument for region 5 applies.

Now suppose $V$ is projective and has position 2.
The argument if $V$ is projective and has position 3 is similar.
Then $V=P_{(a}$ by Definition \ref{def:position}.
If $W$ is injective, $W=I_{a)}$ or $W=I_a$; in either case $\supp V\cap\supp W=\emptyset$.

If $W$ is not injective suppose $\MM W$ borders region 4.
Again, $\MM V$ and $\MM W$ are both on the graph of $\lambda_a^*$.
As before $a$ cannot be a sink and $a$ is an upper bound on the support of $W$.
However, $a$ is again a lower bound on the support of $V$ but not an element of the support of $V$.
Therefore, $\supp V\cap \supp W$.
\end{proof}

\begin{lemma}\label{lem:all double extensions}
Let $V=M_{|a,b|}$ and $W=M_{|c,d|}$ be indecomposables in $\repAR$ such that $V\not\cong W$ and $\Hom(V,W)\cong k$.
Then there is, up to isomorphism and scaling, a unique nontrivial extension $V\hookrightarrow M_{|a,d|}\oplus M_{|c,b|}\twoheadrightarrow W$ if and only if the indecomposables in the exact sequence form a non-degenerate rectangle.
\end{lemma}
\begin{proof}
We first note that if $\Ext^1(W,V)\neq 0$ then, up to isomorphism and scaling, there is a unique extension, by \cite[Theorem 3.0.1]{IgusaRockTodorov2019}.
Thus if we have the non-degenerate rectangle we need only to show the existence of one nontrivial extension.
If the described indecomposables form a non-degenerate rectangle then by Proposition \ref{prop:homline} we have the following maps
\begin{displaymath}\xymatrix@C=10ex{ 0\ar[r] & M_{|a,b|} \ar[r]^-{\left[\begin{array}{c}1\\1\end{array}\right]} & M_{|a,d|}\oplus M_{|c,b|} \ar[r]^-{\left[\begin{array}{rr} 1 & -1\end{array}\right]} & M_{|c,d|} \ar[r] & 0}\end{displaymath}
We see this is an exact sequence.

Now suppose we have the exact sequence.
By Table \ref{tab:slope table} the slopes of lines connecting the nonzero representations in the exact sequence as desired are all $\pm(1,1)$.
By Definition \ref{def:d function} we see the line segments are also of the correct lengths.
This concludes the proof.
\end{proof}

\begin{lemma}\label{lem:all single extensions}
Let $V=M_{|a,b|}$ and $W=M_{|c,d|}$ be indecomposables in $\repAR$ such that $V\not\cong W$, $\Hom(V,W)=0$, and the $x$-coordinate of $\MM V$ is less than or equal to the $x$-coordinate of $\MM W$.
Then there is, up to isomorphism and scaling, a unique nontrivial extension $V\hookrightarrow E \twoheadrightarrow W$ where $E$ is indecomposable if and only if $|a,b|\cup|c,d|$ is an interval and $|a,b|\cap|c,d|=\emptyset$.
\end{lemma}
\begin{proof}
First suppose $|a,b|\cup|c,d|$ is an interval and $|a,b|\cap|c,d|=\emptyset$.
Then either $c=b$ or $a=d$.
We assume $c=b$ by symmetry.
Using Table \ref{tab:slope table} and Proposition \ref{prop:homline} we see $\Hom(V,M_{|a,d|})\cong k$ and $Hom(M_{|a,d|},W)\cong k$.
In particular $0\to V\to M_{|a,d|} \to W \to 0$ is a short exact sequence.
By \cite[Theorem 3.0.1]{IgusaRockTodorov2019} this is unique up to isomorphism and scaling.

Now suppose $|a,b|\cup|c,d|$ is not an interval or $|a,b|\cap|c,d|\neq\emptyset$.
Let $E$ be an extension of $W$ by $V$.
Then $\dim E(x) = \dim V(x) + \dim W(x)$ for all $x\in\R$.
But then $E = V\oplus W $ and so the extension is trivial.

Now suppose $|a,b|\cap|c,d|\neq\emptyset$.
For contradiction, suppose there is a nontrivial extension $E$.
Since $|a,d|$ and $|c,b|$ are nonempty and so we have the exact sequence from Lemma \ref{lem:all double extensions} and so $E$ is not indecomposable.
But then there is a nonzero composition of morphisms $h:V\stackrel{f}{\to} M_{|a,d|}\stackrel{g}{\to} W$ by taking $h(x)=g(x)\circ h(x)$ for all $x\in\R$.
This contradicts $\Hom(V,W)=0$.
\end{proof}

\begin{definition}\label{def:almost complete line}
Let $l$ be a totally ordered set in AR-space that meets the following conditions.
\begin{itemize}
\item There exists a minimal element $V$.
\item For any pair $U<W$ in $l$, the set $\{X\in l : U\leq X\leq W\}$ is a line segment with endpoints $U$ and $W$.
\item The slope of any line segment with endpoints $U< W$ in $l$ is equal to the slope of any other line segment with endpoints $U'<W'$ in $l$.
\item If $U\in l$ and a line segment with endpoints $V$ and $W$ contains $U$ then $W\in l$ and $V<U<W$.
\end{itemize}
If $l$ has no maximal element we call $l$ an \underline{almost complete line segment}.
Since slopes must be constant, the slope of $l$ is the slope of any line segment with endpoints both in $l$.

A \underline{phantom end point} of an almost complete line segment is an indecomposable $E$ in $\repAR$ such that $\MM l \cup \MM E$ is a line segment in $\R^2$ with endpoints $\MM V$ and $\MM E$.
\end{definition}
\begin{remark}\label{rem:almost complete slope remark}
The definition immediately implies that the slope of an almost complete line segment is $\pm(1,1)$.
\end{remark}
\begin{proposition}\label{prop:almost complete coordinate order}
Let $l$ be an almost complete line segment with $V$ its minimal element and $(x_U,y_U)=\MM U$ for all $U\in l$.
Then $\{y_U: U\in l\}$ is totally ordered by the usual order of $\R$ and $y_V$ is the maximal or minimal element.
\end{proposition}
\begin{proof}
By Remark \ref{rem:almost complete slope remark} and symmetry assume the slope of $l$ is $(1,1)$.
Let $l'\subset l$ be a line segment that contains $V$; let its endpoints be $U, W\in l$.
Since $l$ is totally ordered, suppose by symmetry $U< W$.

The slope of $l'$ must be $(1,1)$ so if $U$ has position 1 or 2 so do $V$ and $W$ and similar statement is true if $U$ has position 3 or 4.
The same is true for the line segment $\{X\in l : V\leq X \leq W\}$, which contains $U$.
By the distance requirement in Definition \ref{def:line segment}, $l'=\{X\in l : V\leq X \leq W\}$ and so $U=V$.
Then the proposition follows by Proposition \ref{prop:M coordinate ordering in lines} and Definition \ref{def:almost complete line} .
\end{proof}
\begin{proposition}\label{prop:phantom endpoint existence}
Let $l$ be an almost complete line segment with minimal element $V$.
Then the phantom endpoint of $l$ exists and is unique if and only if there exists $V'\in l$ and  a line segment $l'$ such that
\begin{itemize}
\item $\MM \{X\in l: X\geq V'\} \subset \MM l'$ and
\item $l\cap l'=\emptyset$.
\end{itemize}
\end{proposition}
\begin{proof}
Let $l$ be an almost complete line segment and $V$ its minimal element.
By symmetry suppose the slope of $l$ is $(1,1)$.
For all $U\in l$ let $(x_U,y_U)=\MM U$.
By further symmetry and Proposition \ref{prop:almost complete coordinate order}, suppose that if $\MM U\neq \MM V$ for $U\in V$ then $y_U>y_V$.

Let $U_0=V$ and let $U_1\in l$ such that $\MM U_1 \neq \MM U_0$.
We now inductively choose $U_i$, for $i>1$, in the following way.
Let $z_i=\frac{1}{2}(\frac{\pi}{2} + y_{U_{i-1}}) - y_{U_{i-1}}$.
By Proposition \ref{prop:between projectives and injectives} there is either an indecomposable $W$ such that $\MM W = (x_{U_{i-1}}+z_i,y_{U_{i-1}}+z_i)$ or there is an injective $I$ such that $y_I < z$.
But then $l$ would have a maximal element and so $W$ exists instead.
By maximality of $l$ again, a $U_i$ that exists the slope of a line segment containing both $V$ and $U_i$ must be $(1,1)$.
Then there is an infinite sequence of $U_i\in l$ such that $\{\MM U_i\}$ converges to some $(\bar{x},\frac{\pi}{2})$.

If there is a line segment $l'$ as in the proposition then there is a representation $M$ that is simple or has support $[s_n,s_{n+1}]$ such that $\MM M=(\bar{x},\frac{\pi}{2})$.
By definition $M$ is the phantom endpoint.
Since no other indecomposable is sent to $(\bar{x},\frac{\pi}{2})$ by Definition \ref{def:the rest of M}, $M$ is unique (up to isomorphism).

If the phantom endpoint $M$ exists it must be unique as the inverse image $\MM^{-1}(x,\frac{\pi}{2})$ contains 1 or 0 elements.
For any $U > V$, $U$ belongs to a unique Auslander-Reiten sequence by Corollary \ref{cor:unique AR sequence}.
If $V$ has position 3 or 4 then the slope from $V$ to the phantom endpoint would be $(1,1)$, a contradiction.
Thus, $V$ must have position 1 or 2 and so must $U$.
Then choose $V'$ in the same Auslander-Reiten sequence as $U$ such that the position of $V'$ is 3.
A line segment with endpoints $V'$ and $M$ satisfy the requirements in the proposition.
\end{proof}

\begin{definition}\label{def:almost complete rectangle}
An \underline{almost complete rectangle} is a collection of three line segments $l_1,l_2,l_3$ and an almost complete line segment $l_4$ such that the following hold.
\begin{itemize}
\item (i) $l_1$ and $l_3$ are parallel or (ii) one is degenerate and the other has length $(0,1)$.
\item $l_2$ and $l_4$ have the same slope and no intersection.
\item $l_2$ is perpendicular to whichever $l_1$ and $l_3$ are not degenerate.
\item $\MM (l_1\cup l_2\cup l_3\cup l_4)$ is a (possibly degenerate) rectangle in $\R^2$.
\end{itemize}
\end{definition}
\begin{remark}\label{rem:almost complete rectangle}
The almost complete line segment $l_4$ has a phantom endpoint that is an endpoint of $l_1$ or $l_3$.
\end{remark}

\begin{theorem}\label{thm:extensions are rectangles}
Let $V=M_{|a,b|}$ and $W=M_{|c,d|}$ be indecomposables in $\repAR$ such that $V\not\cong W$.
Then there is a nontrivial extension $V\hookrightarrow E\twoheadrightarrow W$ if and only if there exists a rectangle or almost complete rectangle whose corners are the indecomposables in the sequence with $V$ as the left-most corner and $W$ as the right-most corner.

\begin{itemize}
\item If the rectangle is complete $E$ is a direct sum of two indecomposables.
\item If the rectangle is almost complete $E$ is indecomposable.
\end{itemize}

Furthermore, there is a bijection
\begin{displaymath}\begin{tikzpicture}
\draw (0,0) node {$\{$rectangles and almost complete rectangles with slopes of sides $\pm(1,1)$ in AR-space of $\repAR\}$};
\draw [<->, thick] (0,-.2) -- (0,-1.2);
\draw (0,-.7) node[anchor=west] {$\cong$};
\draw (0,-1.4) node {$\{$nontrivial extensions of indecomposables by indecomposables up to scaling and isomorphisms$\}$};
\end{tikzpicture}\end{displaymath}
\end{theorem}
\begin{proof}
If $\Hom(V,W)\neq 0$ this follows from Lemma \ref{lem:all double extensions}.
If $\Hom(V,W)=0$ we know by Lemma \ref{lem:all single extensions} that a nontrivial extension $E$ is indecomposable and we already have two line segments.

By symmetry suppose $|a,d|$ is an interval.
Then $b=c$ and $b\in|a,b|$ or $c\in|c,d|$.
In either case, we have a third line segment.
By more symmetry suppose $b\in|a,b|$, including $|a,b|=\{b\}$.
By Table \ref{tab:slope table} the slope of a line segment $l_2$ with endpoints $M_{|a,b|}$ and $M_{|a,d|}$ is $(1,1)$.
Similarly, the slope of a line segment $l_1$ with endpoints $M_{|a,d|}$ and $M_{|c,d|}$ is $-(1,1)$.

Since $b$ is the upper bound of $|a,b|$ and the lower bound of $|c,d|$ we see that $\MM V$ and $\MM W$ lie on the graph of $\lambda_b^-$.
If $b$ is not a sink or source then the slope of a line segment $l_3$ with endpoints $V$ and $D:=M_{\{b\}}$ is $-(1,1)$ or the line segment $l_3$ is degenerate.
If $b=s_{n+1}$ is a sink or source then by Definition \ref{def:the rest of M} it is a sink and the slope of a line segment $l_3$ with endpoints $V$ and $D:=M_{[s_n,s_{n+1}]}$ is $-(1,1)$ or the line segment $l_3$ is degenerate.
In either case, observing Table \ref{tab:AR sequence table} we see then that there is an irreducible morphism of indecomposables $U\to W$ whose kernel is $D$.
Then taking the almost complete line segment $l_4$ with minimal element $W$, slope $(1,1)$, where all $U\in l_4$ have lower bound $b$ that is not included.
By Proposition \ref{prop:phantom endpoint existence} $l_4$ is an almost complete line segment with phantom endpoint $D$.
Then $l_1$, $l_2$, $l_3$, and $l_4$ form an almost complete rectangle.

An almost complete rectangle has one phantom vertex that is simple or has support $[s_n,s_{n+1}]$.
This immediately gives an exact sequence of $M_{|a,b|}\to M_{|a,d|}\to M_{|c,d|}$ or $M_{|a,b|}\to M_{|c,b|}\to M_{|c,d|}$.

To see the bijection, note that changing the indecomposables changes the extensions and thus the rectangle or almost complete rectangle.
In the other direction, changing the rectangle changes the endpoints and thus changes the indecomposables and thus the extension.
\end{proof}

\subsection{Kernels and Cokernels}
We conclude this section with a subsection on the geometry of kernels and cokernels in the AR-space.

\begin{notation}\label{note:shifted Gamma}
Let $V$ be an indecomposable in $\repAR$ and let $(x,y)=\MM V$.
Then we denote by $\MM[n]U$ the point $(x+n\pi, (-1)^n y)$ in $\R^2$.
\end{notation}

\begin{proposition}\label{prop:geometrickernelsandcokernels}
Let $V=M_{|a,b|}$ and $W=M_{|c,d|}$ be indecomposables in $\repAR$ such that $\Hom(V,W)\cong k$ and let $f:V\to W$ be some nonzero morphism.
Then $\MM W$, $\MM\coker f$, $\MM[1]\ker f$, $\MM[1] V$, and possibly a point on one of $y=\pm\frac{\pi}{2}$ are the corners of a rectangle in $\R^2$ whose sides have slope $\pm 1$.
The slopes of the lne segments that exist in the AR-space are $\pm(1,1)$.
\end{proposition}
\begin{proof}
By Proposition \ref{prop:homline slope} we know the slope from $V$ to $W$ is at least $-(1,1)$ and at most $(1,1)$.
First suppose the slope is not $\pm(1,1)$.
By Propositions \ref{prop:diagonal endpoints} and \ref{prop:diagonal endpoints converse} $V$ and $W$ do not share an endpoint.
Thus, we have four cases as displayed below:
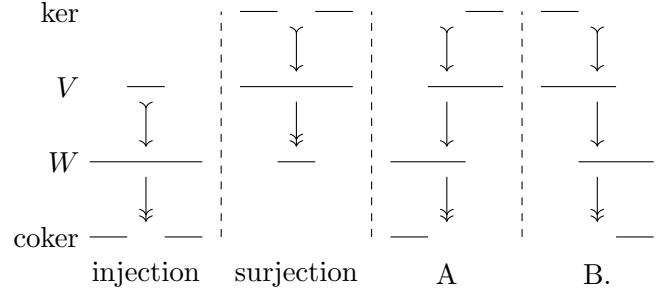
\begin{displaymath}\begin{tikzpicture}
\draw(0,1) -- (1.5,1);
\draw(2.5,1) -- (3,1);
\draw(4,1) -- (5,1);
\draw(6.5,1)--(7.5,1);

\draw (0.5,2) -- (1,2);
\draw (2,2) -- (3.5,2);
\draw (4.5,2) -- (5.5,2);
\draw (6,2) -- (7,2);

\draw (0,0) -- (0.5,0);
\draw (1,0) -- (1.5,0);
\draw (2,3) -- (2.5,3);
\draw (3,3) -- (3.5,3);
\draw (4,0) -- (4.5,0);
\draw (5,3) -- (5.5,3);
\draw (6,3) -- (6.5,3);
\draw (7,0) -- (7.5,0);

\draw[>->](0.75,1.8) -- (0.75,1.2);
\draw[->>](0.75,0.8) -- (0.75,0.2);
\draw[>->](2.75,2.8) -- (2.75,2.2);
\draw[->>](2.75,1.8) -- (2.75,1.2);
\draw[>->](4.75,2.8) -- (4.75,2.2);
\draw[->](4.75,1.8) -- (4.75,1.2);
\draw[->>](4.75,0.8) -- (4.75,0.2);
\draw[>->](6.75,2.8) -- (6.75,2.2);
\draw[->](6.75,1.8) -- (6.75,1.2);
\draw[->>](6.75,0.8) -- (6.75,0.2);

\draw (0,0) node [anchor=east] {$\coker$};
\draw (0,1) node [anchor=east] {$W$};
\draw (0,2) node [anchor=east] {$V$};
\draw (0,3) node [anchor=east] {$\ker$};
\draw[dashed] (1.75,0) -- (1.75,3);
\draw[dashed] (3.75,0) -- (3.75,3);
\draw[dashed] (5.75,0) -- (5.75,3);
\draw (0.75,-0.5) node {injection};
\draw (2.75,-0.5) node {surjection};
\draw (4.75,-0.5) node {A};
\draw (6.75,-0.5) node {B.};
\end{tikzpicture} \end{displaymath}
We know $\MM V$ is the intersection of $\lambda_a^*$ and $\lambda_b^*$ from Definition \ref{def:lambda functions} and $\MM[1]V$ is the next intersection point of $\lambda_a^*$ and $\lambda_b^*$.
$\MM W$ is the intersection of $\lambda_c^*$ and $\lambda_d^*$.
Again since the slope is not $\pm(1,1)$ the kernel and cokernel have two path components in their combined support and are thus given by two indecomposables.
Their image under $\MM$ is given by the intersection of $\lambda_a^*$ and $\lambda_c^*$ and by the intersection of $\lambda_b*$ and $\lambda_d^*$.
The image of the cokernel (if it exists) will be between the image of $\MM W$ and $\MM[1] V$.
The image of the kernel under $\MM[1]$ will also be between $\MM W$ and $\MM[1] V$.
However these are just the intersections or next intersections of the same $\lambda$ functions as seen below:
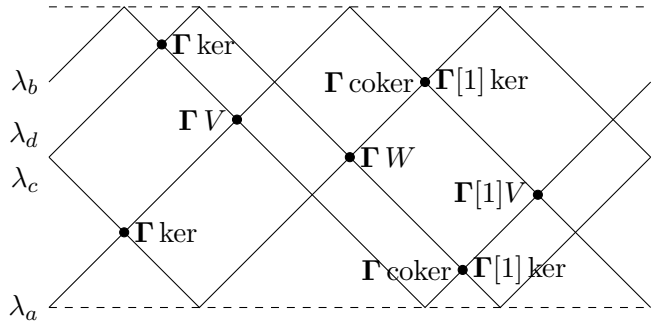
\begin{displaymath}\begin{tikzpicture}
\draw[dashed] (-4,-2) -- (4,-2);
\draw[dashed] (-4,2) -- (4,2);
\draw (-4,-2) -- (0,2) -- (4,-2);
\draw (-4, 0) -- (-2,-2) -- (2,2) -- (4,0);
\draw (-4,0) -- (-2,2) -- (2,-2) -- (4,0);
\draw (-4,1) -- (-3,2) -- (1,-2) -- (4,1);
\filldraw[fill=black]  (-2.5,1.5) circle [radius=0.6mm];
\draw (-2.5,1.5) node [anchor=west] {$\MM\ker$};
\filldraw[fill=black]  (-3,-1) circle [radius=0.6mm];
\draw (-3,-1) node [anchor=west] {$\MM\ker$};
\filldraw[fill=black] (-1.5,0.5) circle [radius=0.6mm];
\draw (-1.5,0.5) node [anchor=east] {$\MM V$};
\filldraw[fill=black] (2.5,-0.5) circle [radius=0.6mm];
\draw (2.5,-0.5) node [anchor=east] {$\MM[1] V$};
\filldraw[fill=black]  (0,0) circle [radius=0.6mm];
\draw (0,0) node [anchor=west] {$\MM W$};
\filldraw[fill=black]  (1.5,-1.5) circle [radius=0.6mm];
\draw (1.5,-1.5) node [anchor=east] {$\MM \coker$};
\draw (1.5,-1.5) node [anchor=west] {$\MM[1]\ker$};
\filldraw[fill=black]  (1,1) circle [radius=0.6mm];
\draw (1,1) node [anchor=east] {$\MM\coker$};
\draw (1,1) node [anchor=west] {$\MM[1]\ker$};
\draw (-4,1) node[anchor=east] {$\lambda_b$};
\draw (-4,-2) node[anchor=east] {$\lambda_a$};
\draw (-4,0) node[anchor= south east] {$\lambda_d$};
\draw (-4,0) node[anchor=north east] {$\lambda_c$};
\end{tikzpicture}\end{displaymath}
%
%Now suppose the slope from $V$ to $W$ is $\pm(1,1)$.
%Then $\MM W$, $\MM[1]V$, the shifted imaage of the kernel or the image of the cokernel, and a point on one of $y=\frac{\pi}{2}$ or $y=-\frac{\pi}{2}$ will also form a rectangle whose sides have slopes $\pm 1$ as seen below (example with $-(1,1)$):
%\begin{displaymath}\begin{tikzpicture}
%\draw[dashed] (-4,-2) -- (4,-2);
%\draw[dashed] (-4,2) -- (4,2);
%\draw (-4,-2) -- (0,2) -- (4,-2);
%\draw (-4, 0) -- (-2,-2) -- (2,2) -- (4,0);
%\draw (-4,1) -- (-3,2) -- (1,-2) -- (4,1);
%\filldraw[fill=black] (-1.5,0.5) circle [radius=0.6mm];
%\draw (-1.5,0.5) node [anchor=east] {$\MM V$};
%\filldraw[fill=black] (-0.5,-0.5) circle [radius=0.6mm];
%\draw (-0.5,-0.5) node [anchor=west] {$\MM W$};
%\filldraw[fill=black] (2.5,-0.5) circle [radius=0.6mm];
%\draw (2.5,-0.5) node [anchor=east] {$\MM[1] V$};
%\filldraw[fill=black]  (1,1) circle [radius=0.6mm];
%\draw (1,1) node [anchor=east] {$\MM\coker$};
%\draw (1,1) node [anchor=west] {$\MM[1]\ker$};
%\filldraw[fill=black] (1,-2) circle [radius=0.6mm];
%\draw (1,-2) node [anchor=north] {$\partial $};
%\draw (-4,1) node[anchor=east] {$\lambda_b$};
%\draw (-4,-2) node[anchor=east] {$\lambda_a$};
%\draw (-4,0) node[anchor=east] {$\lambda_c$};
%\end{tikzpicture}\qedhere\end{displaymath}

Note that kernels and cokernels share endpoints $V$ and $W$.
Therefore, the slopes in the AR-space must be $\pm(1,1)$.

In the case where the slope of a line segment with endpoints $V$ and $W$ is $\pm (1,1)$, $W$ is then the extension of the cokernel by $V$ and we have an almost complete rectangle by Theorem \ref{thm:extensions are rectangles}.
The image of a rectangle or almost complete rectangle in AR-space is a rectangle in $\R^2$ by definition.
\end{proof}

\section{AR-Space of the Bounded Derived Category}\label{sec:derived}
In this short section we describe the Auslander-Reiten space of the bounded derived category of $\repAR$, denoted $\DbAR$.
Those familiar with representations of quivers and their derived categories can skip Section \ref{sec:derived properties}.
In Section \ref{sec:derived AR space} we define the AR-space and prove it has the desired properties one would expect.

\subsection{Essential Properties}\label{sec:derived properties}
We begin by defining $\DbAR$ to be the standard bounded derived category of $\repAR$.
\begin{definition}\label{def:derivedcategory}
Taking bounded chains of representations in $\rep(A_\R)$ and adding the formal inverse of maps that induce isomorphisms on homology gives us the homotopy category $\mathcal{K}(A_\R)$.
Then, modding out by those inverses gives us $\DbAR$ in the usual way.
\end{definition}

\begin{proposition}\label{prop:derived is krull-schmidt}
The category $\DbAR$ is Krull-Schmidt.
The indecomposable objects are shifts of indecomposables in $\repAR$.
\end{proposition}
\begin{proof}
Recall $\repAR$ is Krull-Schmidt and hereditary (\cite[Theorems 2.1.16 and 3.0.1]{IgusaRockTodorov2019}).
Using a result essentially from \cite{Keller2005} but explicitly proved in \cite{ChenRingel2018} it follows that any chain in $\DbAR$ is isomorphic to the direct sum of its homology which is in turn a direct sum of shifted indecomposables in $\repAR$.
Since each index in the chain is a finite sum and we are taking bounded chains, the proposition follows. 
\end{proof}

\subsection{The AR-Space}\label{sec:derived AR space}
Recall Notation \ref{note:shifted Gamma}.
\begin{definition}\label{def:derived Gamma}
We define a function $\MM^b$ from the (isomorphism classes of) indecomposables of $\DbAR$ to $\R\times[-\frac{\pi}{2},\frac{\pi}{2}]$.
Let $W$ be an indecomposable in $\DbAR$.
By Proposition \ref{prop:derived is krull-schmidt}, $W\cong V[n]$ for some $V$ from $\repAR$ and $n\in\Z$.
Define $\MM^b W:= \MM[n] V$.
\end{definition}

\begin{proposition}\label{prop:thewholestrip}
The map $\MM^b$ is surjective onto $\R\times[-\frac{\pi}{2},\frac{\pi}{2}]$ if and only if $A_\R$ has finitely-many sinks and sources.
In that case: if $(x,y)\in \R\times (-\frac{\pi}{2},\frac{\pi}{2})$ then $ (\MM^b)^{-1} (x,y) $ contains exactly 4 indecomposables up to isomorphism and if $(x,y)\in \R\times\{-\frac{\pi}{2},\frac{\pi}{2}\}$ then $ (\MM ^b)^{-1} (x,y) $ contains exactly 1 indecomposable up to isomorphism.
\end{proposition}
\begin{proof}
If $A_\R$ has infinitely-many sinks and sources, then there is no projective or injective indecomposable at one of $\pm\infty$ (possibly both).
Then there are infinitely-many points not in the image of $\MM^b$.

Now suppose $A_\R$ has finitely-many sinks and sources.
By Proposition \ref{prop:lambda covers} and Definition \ref{def:derived Gamma} we see that indeed the image of $\MM^b$ is all of $\R\times[-\frac{\pi}{2},\frac{\pi}{2}]$.
Since there cannot exist both a projective and injective at each $\pm\infty$, but one must exist when there are finitely-many sinks and sources, we see that for each $x\in\R$, $(\MM^b)^{-1}(x,\frac{\pi}{2})$ contains 1 indecomposable up to isomorphism.
A near-identical statement is true for $(x,-\frac{\pi}{2})$.

Suppose $V$ in $\repAR$ is not projective or injective and $\MM V$ does not have $y$-coordinate $\pm\frac{\pi}{2}$.
Then, for all indecomposables $W$ in $\DbAR$ such that $\MM^b V[n] = \MM^b W$, $W\cong U[n]$ for an indecomposable $U$ in the same Auslander-Reiten sequence as $V$.
Thus $(\MM^b)^{-1}(\MM^b V[n])$ contains exactly 4 indecomposables, up to isomorphism.

Finally note that the projectives at $a$ shifted $n$ times map to the same point as the injectives shifted $n-1$ times.
By the classifications in \cite[Remarks 2.4.14 and 2.4.16]{IgusaRockTodorov2019} and Proposition \ref{prop:M on injectives} we conclude the proof.
\end{proof}

\begin{definition}\label{def:derived position}
Let $W$ be an indecomposable in $\DbAR$.
Then $W\cong V[n]$ for some $V$ in $\repAR$ and $n\in\Z$.
If $n$ is even define the \underline{position} of $W$ to be the same as the position of $V$.
If $n$ is odd define the \underline{position} to be
\begin{itemize}
\item the position of $V$ if $V$ has position 1 or 4,
\item 2 if $V$ has position 3, and
\item 3 if $V$ has position 2.
\end{itemize}
\end{definition}

\begin{remark}
We can now use the definitions of our extra generalized metric (\ref{def:d function}), line segment (\ref{def:line segment}), slope (\ref{def:slope}), rectangle (\ref{def:rectangle}), almost complete line segment (\ref{def:almost complete line}), and almost complete rectangle (\ref{def:almost complete rectangle}) as before.
When we say ``nontrivial triangle`` we mean a distinguished triangle that is not the direct sum of triangles whose morphisms are isomorphisms and 0s.
\end{remark}

\begin{definition}\label{def:derived AR space}
The \underline{Auslander-Reiten Space}, or AR-space, of $\DbAR$ is the set of (isomorphism classes of) indecomposables in $\DbAR$ with each indecomposable given their position as in Definition \ref{def:derived position} and the extra generalized metric from $\repAR$ (Definition \ref{def:d function}) extended to $\DbAR$.
Its topology has open sets $(\MM^b)^{-1} U$ where $U$ is open in $\R^2$.
\end{definition}

\begin{remark}
For the purposes of making proofs easier, we will consider the AR-space of $\repAR$ as a subspace of the AR-space of $\DbAR$ where we've included the indecomposables at the 0th degree as one usually does for $\repAR$ and $\DbAR$ as categories.
\end{remark}

\begin{example}
Suppose $A_\R$ has the straight descending orientation.
Then part of the AR-space of $\DbAR$ is:
\begin{displaymath} \begin{tikzpicture}
\draw[draw=white!60!black] (-5,2) -- (10,2);
\draw[draw=white!70!black] (-5,-2) -- (10,-2);
\draw[draw=white!70!black, dotted, pattern=crosshatch dots, pattern color=white!75!black] (-5,-1) -- (-5,-2) -- (-4,-2) -- (-5,-1);
\draw[draw=white!70!black, dotted, pattern=crosshatch dots, pattern color=white!60!black] (-5,2) -- (-5,-1) -- (-4,-2) -- (0, 2) -- (-5,2);
\draw[draw=white!70!black, dotted, pattern=crosshatch dots, pattern color=white!75!black] (-4,-2) -- (0,2) -- (4,-2) -- (-4,-2);
\draw[draw=white!70!black, dotted, pattern=crosshatch dots, pattern color=white!60!black] (0,2) -- (4,-2) -- (8,2) -- (0,2);
\draw[draw=white!70!black, dotted, pattern=crosshatch dots, pattern color=white!75!black] (4,-2) -- (8,2) -- (10,0) -- (10,-2) -- (4,-2);
\draw[draw=white!70!black, dotted, pattern=crosshatch dots, pattern color=white!60!black] (8,2) -- (10,0) -- (10,2) -- (8,2);
\draw[dashed] (-5,-1) -- (-2,2) -- (2,-2) -- (6,2) -- (10,-2);
\draw[dashed] (-5, 1) -- (-2,-2) -- (2,2) -- (6,-2) -- (10,2);
\draw[fill=black] (0,2) circle[radius=0.6mm];
\draw (0,2) node [anchor=south] {$M_{(-\infty,+\infty)}[0]$};
\draw[fill=black] (8,2) circle[radius=0.6mm];
\draw (8,2) node [anchor=south] {$M_{(-\infty,+\infty)}[2]$};
\draw[fill=black] (-4,-2) circle[radius=0.6mm];
\draw (4,-2) node [anchor=north] {$M_{(-\infty,+\infty)}[1]$};
\draw[fill=black] (4,-2) circle[radius=0.6mm];
\draw (-4,-2) node [anchor=north] {$M_{(-\infty,+\infty)}[-1]$};
\draw[fill=black] (0,0) circle[radius=0.6mm];
\draw (0,-0.2) node [anchor=north] {$M_{|a,b|}[0]$};
\draw[fill=black] (-4,0) circle[radius=0.6mm];
\draw (-4,0.2) node [anchor=south] {$M_{|a,b|}[-1]$};
\draw[fill=black] (4,0) circle[radius=0.6mm];
\draw (4,0.2) node [anchor=south] {$M_{|a,b|}[1]$};
\draw[fill=black] (8,0) circle[radius=0.6mm];
\draw (8,-0.2) node [anchor=north] {$M_{|a,b|}[2]$};
\draw[fill=black] (-1,1) circle[radius=0.6mm];
\draw (-1,1) node [anchor=south] {$P_{b|}[0]$, $I_{|b}[-1]$};
\draw[fill=black] (-3,-1) circle[radius=0.6mm];
\draw (-3,-1) node [anchor=north] {$P_{a|}[0]$, $I_{|a}[-1]$};
\draw[fill=black] (1,1) circle[radius=0.6mm];
\draw (1,1) node [anchor=north] {$P_{a|}[1]$,$I_{|a}[0]$};
\draw[fill=black] (3,-1) circle[radius=0.6mm];
\draw (3,-1) node [anchor=south] {$P_{b|}[1]$,$I_{|b}[0]$};
\draw[fill=black] (7,1) circle[radius=0.6mm];
\draw (7,1) node [anchor=south] {$P_{b|}[2]$, $I_{|b}[1]$};
\draw[fill=black] (5,-1) circle[radius=0.6mm];
\draw (5,-1) node [anchor=north] {$P_{a|}[2]$,$I_{|a}[1]$};
\draw[fill=black] (9,1) circle[radius=0.6mm];
\draw (9,1) node [anchor=north] {$P_{a|}[3]$,$I_{|a}[2]$};
\draw[fill=black] (-2,2) circle[radius=0.6mm];
\draw (-2,2) node [anchor=south] {$M_{\{b\}}[-1]$};
\draw[fill=black] (-2,-2) circle[radius=0.6mm];
\draw (-2,-2) node [anchor=north] {$M_{\{a\}}[0]$};
\draw[fill=black] (2,-2) circle[radius=0.6mm];
\draw (2,-2) node [anchor=north] {$M_{\{b\}}[0]$};
\draw[fill=black] (2,2) circle[radius=0.6mm];
\draw (2,2) node [anchor=south] {$M_{\{a\}}[1]$};
\draw[fill=black] (6,2) circle[radius=0.6mm];
\draw (6,2) node [anchor=south] {$M_{\{b\}}[1]$};
\draw[fill=black] (6,-2) circle[radius=0.6mm];
\draw (6,-2) node [anchor=north] {$M_{\{a\}}[2]$};
\draw (-5,-1) node [anchor=east] {$\lambda_b$};
\draw (-5,1) node [anchor=east] {$\lambda_a$};
\draw(-5,-2) node [anchor=east] {$-\frac{\pi}{2}$};
\draw(-5,2) node [anchor=east] {$\frac{\pi}{2}$};
\end{tikzpicture} \end{displaymath}
\end{example}

\begin{proposition}\label{prop:inverse shift of almost complete}
Let $V$ and $W$ be the left-most and right-most corners of an almost-complete rectangle in the AR-space of $\DbAR$.
Then the line segment with endpoints $W[-1]$ and $V$ is $\pm(1,1)$.
\end{proposition}
\begin{proof}
By Remark \ref{rem:almost complete slope remark} we know the slopes of the sides of the almost complete rectangle are $\pm(1,1)$.
Thus, left-most and right-most are well defined.

Assume the image of phantom endpoint under $\MM^b$ has $y$-coordinate $-\frac{\pi}{2}$; the other case is symmetric.
Then we have the following implications direct from the definition of almost complete rectangle:
\begin{itemize}
\item If $V$ has position 1 or 3 then $W$ has position 1 or 2 and so $W[-1]$ has position 1 or 3.
\item If $V$ has position 2 or 4 then $W$ has position 3 or 4 and so $W[-1]$ has position 2 or 4.
\end{itemize}
Finally we see that $\MM^b W[-1]$ and $\MM^b V$ are the endpoints of a line segment with slope $\pm 1$ in $\R^2$; by our assumption it must specifically be $-1$.
Therefore, the proposition holds.
\end{proof}

\begin{lemma}\label{lem:hom support in Db}
Let $V\cong M_{|a,b|}[n]$ be an indecomposable in $\DbAR$ such that $\MM^b V$ has $y$-coordinate not equal to $\pm \frac{\pi}{2}$.
Let $\lambda_a^*$ and $\lambda_b^*$ be the $\lambda$ functions that intersect at $\MM^b V=(x_V,y_V)$ such that $\frac{\text{d}}{\text{d}z}\lambda_a^*$ at $x_V$ is $+1$ and $\frac{\text{d}}{\text{d}z}\lambda_b^*$ at $x_v$ is $-1$.
Let $H_V$ be the set of points $(x_W,y_W)$ in $\R^2$ such that $x_V<x_W<x_V+\pi$ and $\lambda_b^*(x_W) < y_W < \lambda_a^*(x_W)$.

If $\MM^b U\in H_V$ then $\Hom(V,U)\cong k$.
If $\MM^b U\notin \overline{H_V}$ then $\Hom(V,U)=0$.
If $U\in \overline{H_V}$ and the slope of a line segment with endpoints $V$ and $U$ is $\pm(1,1)$ then $\Hom(V,U)\cong k$.
\end{lemma}

The reader may use the same picture used for Lemma \ref{lem:hom support} on page \pageref{lem:hom support} for an intuitive understanding of the statement of the lemma.

\begin{proof}[Proof of Lemma \ref{lem:hom support in Db}]
If $M_{|a,b|}$ is an injective at a point $x\in\R$, the set $H_V$ is the same for the a projective $P$ at $a$ shifted $n+1$ times.
For points $\MM^b U= (x_U,y_U)$ such that $U\cong M_{|c,d|}[n]$, the statement follows from Lemma \ref{lem:hom support}.
Now suppose $U\cong M_{|c,d|}[n+1]$ and $\MM^b U\in H_V$ (this is the only possibility if $M_{|a,b|}$ is injective).
Then $\Hom(U[-1],V)\cong k$ by Lemma \ref{lem:hom support} again.
By assumption the slope of the line segment with endpoints $U[-1]$ and $V$ is not $\pm(1,1)$.
Thus, by Lemma \ref{lem:all double extensions} we have $\Hom (M_{|a,b|}, M_{|c,d|}[1])\cong k$ and so $\Hom(V,U)\cong k$.

Now suppose $Hom(V,U)\cong k$ for $U\cong M_{|c,d|}[m]$.
Again $m=n$ or $m=n+1$.
If $m=n$ then $U\in \overline{H_V}$ by Lemma \ref{lem:hom support}.
If $m=n+1$ then $\Ext (M_{|a,b|},M_{|c,d|})\cong k$ and by Lemma \ref{lem:all double extensions} we have $U\in\overline{H_V}$ again.
The final statement follows from Proposition \ref{prop:homline} and another shift argument.
\end{proof}

For the statement of the final theorem we consider a distinguished triangle to be distinct from any of its rotations.

\begin{theorem}\label{thm:triangles are rectangles}
Let $V=M_{|a,b|}[m]$ and $W=M_{|c,d|}[n]$ be indecomposables in $\DbAR$ such that $V\not\cong W$.
Then there is a nontrivial distinguished triangle $V\to U\to W\to$ if and only if there exists a rectangle or almost complete rectangle in the AR-space of $\DbAR$ whose corners are the indecomposables in the triangle with $V$ as the left-most corner and $W$ as the right-most corner.

\begin{itemize}
\item If the rectangle is complete $U$ is a direct sum of two indecomposables.
\item If the rectangle is almost complete $U$ is indecomposable.
\end{itemize}

Furthermore, there is a bijection
\begin{displaymath}\begin{tikzpicture}
\draw (0,0) node {$\{$rectangles and almost complete rectangles with slopes of sides $\pm(1,1)$ in AR-space of $\DbAR\}$};
\draw [<->, thick] (0,-.2) -- (0,-1.2);
\draw (0,-.7) node[anchor=west] {$\cong$};
\draw (0,-1.4) node {$\{$nontrivial triangles with first and third term indecomposable up to scaling and isomorphisms$\}$};
\end{tikzpicture}\end{displaymath}
\end{theorem}
\begin{proof}
\textbf{Starting with the (almost) complete rectangle.}
We'll first assume $m=n$.
Consider an almost complete rectangle in the AR-space of $\DbAR$.
By Remark \ref{rem:almost complete slope remark} we know the slopes of the sides are $\pm(1,1)$.
Suppose its left-most and right-most corners are $V$ and $W$, respectively.
Since $m=n$, applying the $[-1]$ functor $m$ times to every indecomposable in the rectangle we have an almost complete rectangle in the AR-space of $\repAR$, which by Theorem \ref{thm:extensions are rectangles} gives us an extension $M_{|a,b|}\hookrightarrow E\twoheadrightarrow M_{|c,d|}$ where $E$ is indecomposable.
This yields a distinguished triangle $M_{|a,b|}\to U\to M_{|c,d|}\to$; shifted $m$ times we have $V\to U[m]\to W\to$ where $U[m]$ is indecomposable.
If we start with a complete rectangle in the AR-space of $\DbAR$ such that the left-most and right-most corners have the same shift we obtain the triangle $V\to U[m]\to W\to$ where $U[m]$ is the direct sum of two nonzero indecomposables.

Now suppose $n>m$ (by definition $n$ must be at least $m$).
Suppose we have an almost complete rectangle in the AR-space of $\DbAR$ with the left-most and right-most corners as before.
Since three of the corners of an almost complete rectangle exist, the distance from $\MM^b V$ to $\MM^b W$ can be no more than $\pi$.
Thus, $n=m+1$.
By Proposition \ref{prop:inverse shift of almost complete} we see that $W[-1]$ and $V$ are the endpoints of a line segment with slope $\pm(1,1)$.
In particular, this means $\Hom(W[-1],V)\cong k$ and that these morphisms come from $\repAR$.
In $\repAR$ the indecomposables $M_{|a,b|}$ and $M_{|c,d|}$ share an endpoint by Proposition \ref{prop:diagonal endpoints}.
Thus, the kernel or cokernel is trivial and the other is indecomposable.
Therefore, we have a triangle $W[-1]\to V\to U\to$ where $U$ is indecomposable.
Rotating this triangle gives us the desired triangle; by Proposition \ref{prop:geometrickernelsandcokernels} $U$ is the third existing corner in the almost complete rectangle.
If we start with a complete rectangle then, by Proposition \ref{lem:hom support in Db}, $\Hom(W[-1],V)\cong k$.
Use Proposition \ref{prop:geometrickernelsandcokernels} again and we have the desired result.
We have also seen that an almost complete or a complete rectangle determines a unique (up to isomorphisms and scaling) distinguished triangle.

\textbf{Starting with the distinguished triangle.}
Suppose we are given the distinguished triangle $V\to U\to W\to $ where $V\cong M_{|a,b|}[m]$, $W\cong M_{|c,d|}[n]$, and all the morphisms are nontrivial.
If $m=n$ then this is an extension in $\repAR$ which determines a rectangle or almost complete rectangle in its AR-space.
Shifting by $m$ yields the rectangle or almost complete rectangle in the AR-space of $\DbAR$.

Now suppose $n=m+1$ (the only other possibility).
Then $W[-1]\to V\to U\to$ is also a distinguished triangle.
Since $\repAR$ is hereditary $U$ is the direct sum of the kernel and cokernel of the map $M_{|c,d|} \to M_{|a,b|}$ in $\repAR$.
By Proposition \ref{prop:geometrickernelsandcokernels} this gives us the desired rectangle or almost complete rectangle in the AR-space of $\DbAR$.
We note that the distinguished triangle determined a unique rectangle or almost complete rectangle.
In particular, if we repeat the first part of this proof we obtain the same distinguished triangle again.
Therefore the theorem holds.
\end{proof}

\end{document}